\documentclass[11pt]{amsart}
\usepackage[hyphens]{url}
\usepackage{hyperref}
\usepackage[hyphenbreaks]{breakurl}
\usepackage[backend=biber,doi=false,isbn=false,url=false, style=alphabetic,]{biblatex}

\usepackage[margin= 1in]{geometry}
\usepackage{graphicx,dsfont}
\usepackage{amssymb}
\usepackage{mathtools}
\usepackage{tikz-cd}
\usepackage{color}
\usepackage[title]{appendix}
\usepackage{amsmath, mathrsfs, stmaryrd, color,slashed,bbm}

\usepackage{float}
\usepackage{verbatim}
\usepackage{comment}
\usepackage[utf8]{inputenc}
\usepackage[english]{babel}
\usepackage{enumitem}
\usepackage{booktabs}

\newtheorem{theorem}{Theorem}
\newtheorem{remark}[theorem]{Remark}
\newtheorem{claim}[theorem]{Claim}
\newtheorem{definition}[theorem]{Definition}
\newtheorem{proposition}[theorem]{Proposition}
\newtheorem{corollary}[theorem]{Corollary}
\newtheorem{lemma}[theorem]{Lemma}

\newcommand{\N}{\mathbb{N}}
\newcommand{\R}{\mathbb{R}}
\newcommand{\Z}{\mathbb{Z}}
\newcommand{\E}{\mathbb{E}}
\newcommand{\bP}{\mathbb{P}}

\newcommand{\cL}{\mathcal{L}}

\newcommand{\Fext}{\mathcal{F}_{\textup{ext}}}

\newcommand{\free}{\textup{free}}

\numberwithin{theorem}{section}
\numberwithin{equation}{section}

\begin{document}
\title[]{The Bessel line ensemble}



\author[X. Wu]{Xuan Wu}
\address{Department of Mathematics, University of Chicago, 
Chicago IL, 60637} 
\email{xuanw@uchicago.edu}

\begin{abstract}
In this paper, we construct the Bessel line ensemble, a countable collection of continuous random curves. This line ensemble is stationary under horizontal shifts with the Bessel point process as its one-time marginal. Its finite dimensional distributions are given by the extended Bessel kernel. Furthermore, it enjoys a novel resampling invariance with respect to non-intersecting squared Bessel bridges. The Bessel line ensemble is constructed by extracting the hard edge scaling limit of a collection of independent squared Bessel processes starting at the origin and never being conditioned to intersect. This process is also known as the Dyson Bessel process, and it arises as the evolution of the eigenvalues of the Laguerre unitary ensemble with i.i.d. complex Brownian entries.
\end{abstract}

\maketitle

\vspace{1.5cm}

\section{Introduction}\label{sec:Introduction}
Over the past decades, there has been significant attention in non-intersecting paths formed by one-dimensional Markov processes conditioned never to intersect. Such path structures naturally arise in the study of random matrix theory, growth processes, directed polymers, interacting particle systems and tiling problems (see the surveys \cite{Fer,FS,Joh,OC,Spo}).  A famous example of non-intersecting paths is the collection of Brownian motions conditioned never to collide, known as the Dyson Brownian motion. In this paper, we study another model of non-intersecting random curves, called the {\em Dyson Bessel process}, and focus on its hard edge scaling limit.

\subsection{Non-intersecting squared Bessel process} The Bessel process is one of the most important one-dimensional diffusion processes. Let $d\geq 1$ be an integer. The $d$-dimensional Bessel process is defined as the distance to the origin of a $d$-dimensional Brownian motion \cite[Ch. XI]{RY}. The index $\alpha=\frac{d-2}{2}$ serves as another natural parametrization for the $d$-dimensional Bessel process. We will use the index $\alpha$ and call it the $\alpha$-Bessel process. See Section~\ref{sec:BB} for more details. Throughout this paper, we consider the parameter range $\alpha\geq 0$. Let $X^\alpha(t)$ be an $\alpha$-Bessel process.  Taking the square of $X^{\alpha}(t)$, one obtains the \textit{squared $\alpha$-Bessel process} (BESQ),\, $Y^{\alpha}(t)=\left(X^{\alpha}(t)\right)^2$. \\[-0.35cm]

In this paper, we are mainly interested in the non-intersecting squared Bessel process. This is also known as the Dyson (squared) Bessel process, analogous to the Dyson Brownian motion. Fix $N\in\N$ and $\alpha\geq 0 $. Let $Y^{N,\alpha}_1(t), Y^{N,\alpha}_2(t),\cdots, Y^{N,\alpha}_{N}(t)$ be $N$ independent squared $\alpha$-Bessel processes with zero initial values, i.e. $Y^{N,\alpha}_i(0)=0$ for all $1\leq i\leq N$. The non-intersecting squared Bessel process is obtained through conditioning on $\{Y^{N,\alpha}_1(t)< Y^{N,\alpha}_2(t)<\cdots< Y^{N,\alpha}_N(t)\}$ for all $t\in (0,\infty)$. This is a singular conditioning and could be made rigorous via the Doob's $h$-transform \cite{KO}.\\[-0.3cm]

The non-intersecting squared Bessel process enjoys a beautiful interpretation as the eigenvalue evolution of the Laguerre process \cite{KO}. Fix $N\in\mathbb{N}$ and $\alpha\geq 0 $. Take $A(t)$ to be an $N\times (N+\alpha)$ matrix with independent standard complex Brownian entries (so that the real and the imaginary part both have variance $t$) and set $M(t)=A(t)A(t)^*$. For $t=1/2$, $M(1/2)$ is known as the Wishart ensemble, one of the earliest random matrix ensembles, introduced by Wishart \cite{Wis} in 1928. The joint density function of the eigenvalues of $M(1/2)$ takes the following form
\begin{align}\label{equ:06261715}
C(N,\alpha)\prod_{j=1}^N (x_j)^\alpha e^{-x_j}\times \big(\Delta(\vec{x})\big)^2 \mathbbm{1}\{0\leq x_1<x_2<\dots <x_N\}  \prod_{j=1}^N dx_j.
\end{align}
Here $\Delta(\vec{x})=\prod_{1\leq i<j\leq N}(x_j-x_i)$ is the Vandermonde determinant and $C(N,\alpha)$ is an explicitly computable constant, see \cite[(1.5b)]{For93}. The Wishart ensemble is also referred to as the Laguerre unitary ensemble (LUE) since it is unitarily invariant and its joint eigenvalue distributions involve Laguerre polynomials. The study of the asymptotic behavior of the eigenvalues of $M(t)$ and its variants has been an important topic in the random matrix community, see for instance \cite{For93, For94, FT, JV, TW}. There are three natural asymptotic regimes, the {\em bulk} scaling, the {\em soft edge }scaling and the {\em hard edge} scaling. 

In this paper we investigate the non-intersecting squared $\alpha$-Bessel process, $Y^{N,\alpha}:=\{Y^{N,\alpha}_1(t)\leq \cdots Y^{N, \alpha}_N(t)| t\geq 0, 1\leq i\leq N\}$, under the hard edge scaling. The hard edge scaling zooms in near $Y^{N,\alpha}_1$ and it bears the name due to the fact that all of $Y^{N,\alpha}_j$ are non-negative and may not cross zero. More precisely, for $ 1\leq i\leq N $ and $t\in[-4N,\infty)$, define 
\begin{align}\label{eqn:edgeLE1}
\mathcal{L}_i^{N,\alpha}(t):=4N\cdot Y^{N,\alpha}_i\left(1+t/(4N)\right).
\end{align}
We refer to
\begin{align}\label{eqn:edgeLE2}
\mathcal{L}^{N,\alpha}:=\left\{\cL_i^{N,\alpha}(t)|\, t\in [-4N,\infty) ,\ 1\leq i\leq N\right\}.
\end{align}
as the {\em scaled non-intersecting squared $\alpha$-Bessel process}.  The finite dimensional convergence of $\cL^{N,\alpha}(t)$ has been well studied  by analyzing {\em determinantal formulas} \cite[section 11.7.3]{For10}. At time $t=0$, $\cL^{N,\alpha}(0)$ converges weakly to the $\alpha$-Bessel point process. This implies the convergence for any fixed time due to the following scaling invariance
\begin{align}\label{equ:06261716}
\left\{Y^{N,\alpha}_i(t),\ 1\leq i\leq N\right\} \, \overset{(d)}{=\joinrel=} \, \left\{ t\cdot Y^{N,\alpha}_i(1),\ 1\leq i\leq N\right\}.
\end{align}
Moreover, for any finite set $I\subset\mathbb{R} $, $\{\cL_i^{N,\alpha}(t), t\in I\}_{i=1}^N$ converge weakly to the \textit{extended $\alpha$-Bessel point process} with the correlation kernel (known as the extended Bessel kernel) $K^{\textup{ext}}:(\mathbb{R}\times[0,\infty) )^2\rightarrow\mathbb{R}$ given by
\begin{equation}\label{def:KEXT}
\begin{split}
K^{\textup{ext}}((t,x),(s,y))\coloneqq \left\{\begin{array}{cc}
 -\displaystyle\int_{1/8}^\infty e^{- 2(s-t) z}   J_\alpha ( 2\sqrt{zx})J_\alpha (2\sqrt{zy})\, dz, & t< s  ,\\[0.5cm]
 \displaystyle\int_0^{1/8} e^{- 2(s-t) z}   J_\alpha ( 2\sqrt{zx})J_\alpha (2\sqrt{zy})\, dz, & t\geq s.  
\end{array}\right.
\end{split}
\end{equation}
Here $J_\alpha$ is the Bessel function of the first kind.

\subsection{Gibbsian line ensembles}
Aside from the perspective of determinantal structures, the non-intersecting Bessel process is also worth investigating from the perspective of Gibbsian line ensembles.  A line ensemble is a countable collection of random discrete or continuous curves on some interval in $\mathbb{R}$ (all defined on the same probability space). The defining property of a Gibbsian line ensemble, the Gibbs property, is a resampling invariance. Let us illustrate the Gibbs property using the Dyson Brownian motion, which is the law of $N$ independent Brownian motions, $B_1, B_2, \cdots, B_N$, all starting at the origin at time zero and conditioned never to intersect. 

The Dyson Brownian motion enjoys the Brownian Gibbs property, introduced in \cite{CH14}, a resampling invariance under the following action. Select an index $1\leq k\leq N$ and erase $B_k$ on a fixed time interval $(a, b)$; then replace this erased curve with a new curve on $(a, b)$ according to the law of a Brownian bridge between the two existing endpoints $(a, B_k(a))$ and $(b, B_k(b))$, conditioned to intersect neither the curve above nor the one below. The invariance of the total law under this action is the Brownian Gibbs property. The precise definition of the Brownian Gibbs property is slightly more general than this and one may resample multiple neighboring paths simultaneously. It is convenient to think of a line ensemble that satisfies the Brownian Gibbs property as $N$ random curves which locally have the distribution of $N$ avoiding Brownian bridges.

Gibbsian line ensembles come in two different flavors where the underlying paths are continuous or discrete. The corresponding Gibbs properties are often named after the nature of the path measures, e.g. the Brownian Gibbs property, geometric Gibbs property and exponential Gibbs property
 
Initiated in \cite{CH14} for the construction of the Airy line ensemble, there has been a fruitful development of techniques \cite{CH16,CD,DW21,CIW,Wu19} which leverage the Gibbs property of Gibbsian line ensembles to prove their tightness under scalings to the Airy line ensemble and its closely related analogue, the KPZ line ensemble \cite{CH16}. The Gibbs property has also served as a powerful tool to establish path regularity for the Airy / KPZ line ensemble \cite{CHH,Ham,Wu21a,Wu21b}. 

\subsection{Main results}
To our best knowledge, the study of Gibbsian line ensembles has been focused in the Airy /KPZ regime. Moreover, the Gibbs property being investigated is either the Brownian Gibbs property or its variants (positive temperature or discrete analogue). 

The main object of study in this paper, the {\em non-intersecting squared Bessel process} enjoys a novel Gibbs property, the squared Bessel Gibbs property such that it locally resembles avoiding squared Bessel bridges. Through the squared Bessel Gibbs property, we study its asymptotic behavior under the hard edge limit and prove a functional limit theorem (Theorem \ref{thm:main}(i)) for $\cL^{N,\alpha}$.  This is in essence tightness for this family of curves. Furthermore, we prove that the squared Bessel Gibbs property is preserved under the subsequential limit.

\begin{theorem}\label{thm:main}
Fix $\alpha\geq 0$. Let $\cL^{N,\alpha}$ be defined as in \eqref{eqn:edgeLE2}. The following statements hold true.
\begin{enumerate}[label=(\roman*)]
\item $\cL^{N,\alpha}$ is tight as $N$ goes to infinity. \\[-0.3cm]
\item Any subsequential limit $\cL^{\infty,\alpha}$ enjoys the squared $\alpha$-Bessel Gibbs property.
\end{enumerate}
\end{theorem}

Together with the finite dimensional convergence result (see Theorem~\ref{thm:extended BE}), we prove the existence of the Bessel line ensemble with index $\alpha$.
\begin{corollary}\label{cor:main}
Fix $\alpha\geq 0$. There exists a continuous non-intersecting Gibbsian line ensemble $\mathcal{B}^{\alpha}:=\{\mathcal{B}^{\alpha}_i(t), i\in\mathbb{N},\, t\in\mathbb{R}\}$ with $0<\mathcal{B}^{\alpha}_1(t)<\mathcal{B}^{\alpha}_2(t)<\cdots$ such that the following statements hold. For any finite set $I\subset\mathbb{R}$, the point process given by $\{\mathcal{B}^\alpha_i (t), i\in \mathbb{N},\, t\in I\}$ is a determinantal point process whose correlation kernel is given by $K^{\textup{ext}}$ defined in \eqref{def:KEXT}. Furthermore, $\mathcal{B}^\alpha$ enjoys the squared $\alpha$-Bessel Gibbs property.
\end{corollary}

\subsection{Comparison between the Airy line ensemble and the Bessel line ensemble}
The Airy line ensemble $\mathcal{A}$ is well known as a universal limit in the KPZ universality class \cite{Cor, QS15}, particularly the {\em soft edge} scaling limit of the Dyson Brownian motion. The Bessel line ensemble $\mathcal{B}$ is constructed in this paper as a {\em hard edge} scaling limit of the Dyson (squared) Bessel process. These two Gibbsian line ensembles share many basic properties in common --- both are non-intersecting, determinantal, stationary under horizontal shifts. In this section, we make a comparison of their differences. In doing so, we aim to illustrate the new challenges we encounter when adapting the Brownian Gibbsian resampling techniques to the current setting.

One apparent difference is that the Airy line ensemble and the Bessel line ensembles have different Gibbs properties. The Gibbs property of the Airy line ensemble uses Brownian bridges to resample random curves. On the other hand, squared Bessel bridges play this role for the Bessel line ensembles. Brownian motions/bridges are the most well-studied stochastic processes and many exact formulas are available. In contrast, calculations involving squared Bessel processes/bridges are more difficult. We overcome this difficulty by viewing (squared) Bessel processes/bridges as solutions to stochastic differential equations. This point of view allows us to obtain certain basic controls of the squared Bessel bridges such as the modulus of continuity. This difference also leads to the requirement for extra effort to prove stochastic monotonicity which we explain in Section~\ref{sec:intromonotonicity} in more detail.

Another major difference is that the Bessel line ensembles are stationary while the parabolic Airy line ensemble has a parabolic shape. In the construction of the parabolic Airy line ensemble in \cite{CH14}, this parabolic shape plays a crucial role in providing a uniform lower bound for the random curves. Roughly speaking, if the $k$-th curve $\mathcal{A}_k$ drops too low over an interval, it does not provide enough support for the curve above it to configure parabolically. Without such a parabolic shape, we need another approach for the Bessel line ensembles which we explain below.

Recall the definition of the scaled non-intersecting squared Bessel process $\mathcal{L}^{N,\alpha}$ in \eqref{eqn:edgeLE2}. We want to show that for any $\varepsilon>0$ and $k\in\mathbb{N}$, there exists $r>0$ such that
\begin{align}\label{eqn:uniforminf_k}
\mathbb{P}\left( \inf_{t\in [0,1]}\mathcal{L}_k^{N,\alpha}(t)<r \right)<\varepsilon.
\end{align} 
Due to the ordering of the curves in $\mathcal{L}^{N,\alpha}$ , it suffices to prove \eqref{eqn:uniforminf_k} for the lowest curve, i.e. $k=1$. 

For fixed $t_0\in\mathbb{R}$, $\mathcal{L}^{N,\alpha}_1(t_0)$ is supported on $(0,\infty)$. Therefore, for any $\varepsilon>0$, there exists $r(\varepsilon)>0$ such that \\[-0.75cm]
\begin{align*}
\mathbb{P}\left(  \mathcal{L}^{N,\alpha}_1(t_0)<r(\varepsilon) \right)<\varepsilon.\\[-0.7cm]
\end{align*}
We want to use the squared Bessel Gibbs property to propagate the above bound to a small interval containing $t_0$. Let $d>0$ be a small number. Suppose there exists a $t_1\in [t_0-d,t_0+d]$ such that $\mathcal{L}^{N,\alpha}_1(t_1)<r(\varepsilon)/2$. By choosing $d$ small enough, through a Gibbs resampling argument (see Section \ref{sec:uniformbounds}) we obtain that with a high probability $\mathcal{L}^{N,\alpha}_1(t_0)<r(\varepsilon)$. This implies
\begin{align*}
\mathbb{P}\left( \inf_{t\in [t_0-d,t_0+d]} \mathcal{L}^{N,\alpha}_1(t)<r(\varepsilon)/2 \right) \lesssim \varepsilon.
\end{align*}
Covering the interval $[0,1]$ by intervals with length $2d$, we get
\begin{align*}
\mathbb{P}\left( \inf_{t\in [0,1]} \mathcal{L}^{N,\alpha}_1(t)<r(\varepsilon)/2 \right) \lesssim d^{-1} \varepsilon.
\end{align*}
Recall the scaling invariance of the squared Bessel process (space and time are of the same scale), we have $d\sim r(\varepsilon)$. Therefore,
\begin{align*}
\mathbb{P}\left( \inf_{t\in [0,1]} \mathcal{L}^{N,\alpha}_1(t)<r(\varepsilon)/2 \right) \lesssim r(\varepsilon)^{-1} \varepsilon.
\end{align*}
In order to obtain a meaningful estimate, we need the tail estimate of $\mathcal{L}^{N,\alpha}(t_0)$ near 0. Combining the one-time convergence and the asymptotics for the Bessel point process, we have $r(\varepsilon)\sim \varepsilon^{1/(1+\alpha)}$. We then arrive at
\begin{align}\label{equ:try}
\mathbb{P}\left( \inf_{t\in [0,1]} \mathcal{L}^{N,\alpha}_1(t)<r(\varepsilon)/2 \right) \lesssim \varepsilon^{\alpha/(1+\alpha)}.
\end{align}
For $\alpha>0$, \eqref{equ:try} is sufficient by picking a new $\varepsilon_1$ as $\varepsilon^{(1+\alpha)/\alpha}$. However, \eqref{equ:try} degenerates when $\alpha=0$. 

To deal with the degenerate case $\alpha=0$, a key observation we have is that for the tightness, it suffices to show \eqref{eqn:uniforminf_k} for \textbf{some} $k$. Because the curves are ordered $\mathcal{L}^{N,\alpha}_1 <\mathcal{L}^{N,\alpha}_2 <\dots$, a lower bound for $\mathcal{L}^{N,\alpha}_k$, $k\geq 2$ should decay faster than the one for $\mathcal{L}^{N,\alpha}_1$. This is indeed the case for $k=2$. It can be proved that 
$$\mathbb{P}(\mathcal{L}^{N,0}_2(t_0)\leq \varepsilon^{1/2})\lesssim \varepsilon.$$ 
Therefore, in estimating $\mathcal{L}^{N,0}_2$, we could replace $r(\varepsilon)$ above by $\varepsilon^{1/2}$ which is much larger than $\varepsilon$. 

We adapt the strategy above and sample both $\mathcal{L}^{N,0}_1$ and $\mathcal{L}^{N,0}_2$ together. The difficulty then translates to controlling the joint density of two \textbf{non-intersecting} squared Bessel bridges. The joint density is given in a determinantal form using the Karlin-Mcgregor formula. In order to estimate lower and upper bounds of the joint density, we derive a few inequalities regarding the modified Bessel functions (see Appendix \ref{sec:BE}, Corollary \ref{cor:qqqq} and Lemma \ref{lem:2sample}). These inequalities serve as part of the main ingredients for the desired uniform infimum bounds on a unit interval. 

We remark that once having the infimum bound for $\mathcal{L}^{N,0}_2$ on an interval, one can use the Gibbs property again to obtain the infimum bound for $\mathcal{L}^{N,0}_1$. However, we do not pursue it as the control on $\mathcal{L}^{N,0}_2$ is enough for proving tightness.

\subsection{The stochastic monotonicity Proposition \ref{lem:monotonicity}}\label{sec:intromonotonicity}
The stochastic monotonicity for non-intersecting Brownian bridges was first proved in \cite{CH14} and has played an important role in the construction of the Airy line ensemble. Roughly speaking, the stochastic monotonicity says that the bridges almost surely increase (or decrease) when the boundary values (endpoints, upper and lower barrier curves) increase (or decrease). Stochastic monotonicity has served as a crucial tool among the applications of Gibbsian line ensemble, mainly because it helps reduce the complexity of the system, e.g. possibly random boundary data to manageable and deterministic ones. The key idea of proving the monotonicity in \cite{CH14} is to construct the monotone coupling through invariant measures of two Monte-Carlo Markov chains (known as Glauber dynamics), which are monotonically coupled. The authors of \cite{CH14} achieve so by exploiting monotone coupling for non-intersecting Bernoulli random walk bridges and taking the diffusive limits to non-intersecting Brownian bridges. 

In our case, $\mathcal{L}^{N,\alpha}$ enjoys the squared $\alpha$-Bessel Gibbs property. It is natural to adapt the same general framework to prove the stochastic monotonicity for squared Bessel bridges. But unlike the convergence from simple random walk bridges to a Brownian bridge, there is no such obvious choice of discrete random walk bridges which converge to the squared Bessel bridge. We discretize the space $C(\{1,2,\dots, k\}\times [a,b],\mathbb{R})$ of $k$ continuous functions on the interval $[a,b]$ and construct random walk bridges on it. Then we run the same type of Markovian dynamics, i.e. the Glauber dynamics. 

In order to show that the ordering is preserved by the Markov chains, we reduce the desired result to a convexity condition (see \eqref{q:convex}) on the transition density function of the squared $\alpha$-Bessel processes and prove it through ODE comparison. Interestingly, this convexity was studied by Gronwall \cite{Gro} (for a slightly smaller parameter regime) motivated by a problem in wave mechanics. 

It remains to verify the random walk bridges we construct converge to the squared Bessel bridges. This further requires estimates (uniform supnorm and $L^1$ norm) on the transition density functions, which we establish in Appendix~\ref{sec:discrete}.

\subsection*{Outline}
This paper is organized as follows. In Section~\ref{sec:BB}, we introduce the (squared) Bessel process/bridge and some of their basic properties. Section \ref{sec:Basics of line ensembles} contains various definitions necessary to describe squared Bessel Gibbsian line ensembles. Section \ref{sec:Stochastic monotonicity} contains the stochastic monotonicity for non-intersecting (squared) Bessel line ensembles. Sections \ref{sec:uniformbounds} and \ref{sec:proofofZ} provide controls on the uniform upper/lower bound and on the normalizing constants respectively. These lead to the proof of the main Theorem \ref{thm:main} in Section \ref{sec:proofmainthm}. In Appendix \ref{sec:BE}, we prove several properties for modified Bessel functions. Appendix \ref{sec:discrete} records a technical step towards proving the stochastic monotonicity. We derive the correlation kernel for the non-intersecting squared Bessel process and prove its convergence in Appendix \ref{sec:BEkernel}.

\subsection*{Notation}
We would like to explain some notation here. We use $\mathbb{R}_+$ to denote non-negative real numbers $[0,\infty)$. The natural numbers are defined to be $\N = \{1, 2, . . .\}.$ For integers $k_1<k_2$, let $[k_1,k_2]_{\Z} := \{k_1,k_1+1,\ldots,k_2\}$. For a closed set $A\subset \mathbb{R}$ or $A\subset \mathbb{N}\times \mathbb{R}$, we denote by $C(A,\mathbb{R})$ the collection of continuous functions defined on $A$. We equip $C(A,\mathbb{R})$ with the topology of uniform convergence on compact subsets of $A$ and denote the sigma-field generated by Borel sets by $\mathcal{C}(A,\mathbb{R})$. For $f\in C([a,b],\mathbb{R})$ and $r>0$, the modulus of continuity is defined by
\begin{equation}\label{def:modu-1}
\omega_{[a,b]}\big(f ,r\big)\coloneqq \sup_{\substack{s,t\in [a,b]\\|s-t|\leq r}}\vert f (s)-f (t)\vert.
\end{equation}  
More generally, for $f\in C( [1,k]_{\mathbb{Z}}\times [a,b], \mathbb{R})$ and $r>0$, the modulus of continuity is defined as
\begin{equation}\label{def:modu-2}
\omega_{[a,b],k}\big(f ,r\big)\coloneqq \sup_{1\leq i \leq k} \sup_{\substack{s,t\in [a,b]\\|s-t|\leq r}}\vert f(i,s)-f(i,t)\vert.
\end{equation} 
We denote by $\mathbb{W}^N_+$ the Weyl chamber restricted on non-negative reals
\begin{align*}
\mathbb{W}^N_+\coloneqq \{(x_1,x_2,\dots, x_N)\in\mathbb{R}^N\, |\, 0\leq x_1<x_2<\dots <x_N\}.
\end{align*}
Events are denoted in a special font $\mathsf{E}$, their indicator functions are written as $\mathbbm{1}_{\mathsf{E}}$ and the complements are written as $\mathsf{E}^{\textup{c}}$. 

\subsection*{Acknowledgment}
The author extends thanks to Ivan Corwin for helpful comments on a draft of this paper and to Patrik Ferrari and Peter Forrester for pointing out many references. The author is very grateful to Greg Lawler for many valuable discussions and for his initial contributions to an earlier draft of this project. 

\section{Squared Bessel process}\label{sec:BB}
\subsection{Basic properties for squared Bessel processes}
In this section we introduce the squared Bessel processes and collect some of their basic properties. We fix $\alpha\geq 0$ throughout this section. For brevity, we often omit the dependence of $\alpha$. For instance, we call a squared $\alpha$-Bessel process simply a squared Bessel process. 

For $x\geq 0$, a squared Bessel process starting at $x^2$ is the solution to the following stochastic differential equation (SDE):
\begin{equation}\label{equ:SQBESDE}
dY(t)=(2\alpha+2)dt+2\sqrt{Y(t)} dB(t),\ Y(0)=x^2,
\end{equation} 
where $B(t)$ is a Brownian motion with diffusion parameter $1$. It is known that \eqref{equ:SQBESDE} has a unique strong solution which stays positive for all $t>0$ \cite[Ch. XI, \S 1]{RY}. The transition density of a squared $\alpha$-Bessel process is given by \cite[Ch. XI, Corollary 1.4]{RY}
\begin{equation}\label{def:q}
\begin{split}
q_t(x,y)=\left\{ \begin{array}{cc}
(2t)^{-1}(y/x)^{\alpha/2} e^{-(x+y)/(2t)}I_\alpha \left(  t^{-1}\sqrt{xy} \right), & x>0, y\geq 0,\\
\frac{2^{-\alpha-1}}{\Gamma(\alpha+1)}t^{-\alpha-1} y^{\alpha}e^{-y/(2t)}, & x=0, y\geq 0.
\end{array} \right.
\end{split}
\end{equation}
Here $I_\alpha(z)$ is the modified Bessel function of the first kind.
Note that $q_t(x,y)$ enjoys a scaling invariance,
\begin{equation}\label{q_scaling}
\begin{split}
q_t(x,y)=&t^{-1} q_1(t^{-1}x,t^{-1}y).
\end{split}
\end{equation}
To see $q_t(x,y)$ is continuous in $x$, define $h_\alpha(z)\coloneqq z^{-\alpha}I_\alpha(z)$. $h_\alpha(z)$ has the following power series expansion \cite[(9.6.10)]{AS}
\begin{align}\label{equ:hexpansion}
h_\alpha(z)=2^{-\alpha}\sum_{n=0}^\infty \frac{(z/2)^{2n}}{n!\Gamma(n+\alpha+1)}.
\end{align}
From this expansion, it is easy to check that $h_\alpha(z)$ is an entire function. Expressing $q_t(x,y)$ in terms of  $h_\alpha(z)$, we have
\begin{align}\label{def:qF}
q_t(x,y)= 2^{-1}t^{-\alpha-1}y^{\alpha}e^{-(x+y)/(2t)} h_\alpha\left( t^{-1}\sqrt{xy} \right)
\end{align}

Let $Y(t)$ be a squared Bessel process which solves \eqref{equ:SQBESDE}. $X(t)\coloneqq \sqrt{Y(t)}$ is called a Bessel process. From Ito's formula, $X(t)$ solves the SDE
\begin{equation}\label{equ:BESDE}
dX(t)= \frac{\alpha+1/2}{X(t)}dt+dB(t),\ X(0)=x.
\end{equation}
Let $p_t(x,y)$ be the transition probability for a Bessel process. From \eqref{def:qF}, we have
\begin{align}\label{def:pF}
p_t(x,y)=t^{-\alpha-1}y^{2\alpha+1}e^{-(x^2+y^2)/(2t)}h_\alpha \left( t^{-1}xy  \right).
\end{align}
It enjoys the Brownian scaling invariance $p_t(x,y)=t^{-1/2} p_1(t^{-1/2}x,t^{-1/2}y).$

We are interested in the \textit{bridge processes} induced from (squared) Bessel processes. Intuitively, a Bessel bridge on $[0,T]$ is a Bessel process $X(t)$ starting at $x$ conditioned on $X(T)=y$ for some $x,y\geq 0$. Precisely, a process $\mathcal{S}(t)$, $0\leq t\leq T$ is called a Bessel bridge on $[0,T]$ with entrance and exit data $(x,y)$ if given $0<t_1<t_2<\dots <t_k<T$, the joint density of $(\mathcal{S}(t_1),\mathcal{S}(t_2),\dots, \mathcal{S}(t_k))$ equals
\begin{equation}\label{def:Besslebridgedensity}
p_{t_1}(x,z_1)p_{t_2-t_1}(z_1,z_2)\cdots p_{T-t_k}(z_k,y)/p_{T}(x,y). 
\end{equation} 

A Bessel bridge can be obtained through the Doob's $h$-transform. Let $X(t)$ be a Bessel process starting at $x$ which solves \eqref{equ:BESDE}. It can be checked directly that $M(t)\coloneqq p_{T-t}(X(t),y)/p_T(x,y)$ (the case $y=0$ is understood by taking a limit) is a non-negative martingale for $t\in [0,T)$. Moreover, through tilting the measure by $M(t)$, $X(t)$ restricted on $t\in [0,T)$ is a Bessel bridge with entrance and exit data $(x,y)$. Next, we record the SDE for Bessel bridges. Define $r_t(x,y)\coloneqq \frac{\partial}{\partial x}\log p_t(x,y)$. Fix $T>0$ and $x,y\geq 0$. From the Girsanov theorem \cite[Ch. VIII]{RY}, a Bessel bridge on $[0,T]$ with entrance and exit data $(x,y)$ satisfies the following SDE:
\begin{equation}\label{equ:BEBSDE}
\begin{split}
dX(t)=\left( \frac{\alpha+1/2}{X(t)}+r_{T-t}(X(t),y) \right) dt+dB(t),\ X(0)=x.
\end{split}
\end{equation}
From \eqref{equ:BEBSDE}, we have the following comparison between Bessel bridges and Brownian motions. This is a special case of \cite[Theorem 1']{KadotaShepp}.

\begin{lemma}\label{lem:abscontinuous}
Fix $T>0$, $x> 0$ and $y\geq 0$. Let $\mathcal{S}(t)$ be a Bessel bridge defined on $[0,T]$ with entrance and exit data $(x,y)$ and let $B(t)$ be a Brownian motion starting at $x$. Then for any $T'\in (0,T)$, the following holds. The law of $\mathcal{S}(t)\big|_{t\in [0,T']}$, viewed as a Borel measure on $C([0,T'],\mathbb{R})$, is absolutely continuous with respect to the law of $B(t)\big|_{t\in [0,T']}$. 
\end{lemma}

Squaring a Bessel bridge with entrance and exit data $(x,y)$, we obtain a squared Bessel bridge with entrance and exit data $(x^2,y^2)$. 

\begin{lemma}\label{lem:not touching}
Fix $T>0$. Let $g:[0,T]\to (0,\infty]$ be a lower semi-continuous function and $f:[0,T]\to [0,\infty)$ be an upper semi-continuous function. Let $x,y>0$ satisfy $f(0)<x<g(0) $ and $f(T)<y<g(T)$. Let $\mathcal{Q}$ be a squared Bessel bridge on $[0,T]$ with entrance and exit data $(x,y)$. Then it holds that 
\begin{align*}
\mathbb{P} \left( \inf_{t\in [0,T]} (g(t)-\mathcal{Q} (t))=0 \right)=0\ \textup{and}\  \mathbb{P} \left( \inf_{t\in [0,T]} (  \mathcal{Q} (t)-f(t))=0 \right)=0.
\end{align*} 
\end{lemma}
\begin{proof}
Let $B(t)$ be a Brownian motion with $B(0)=\sqrt{x}$. We start by showing that for all $T'>0$,
\begin{align*}
\mathbb{P} \left( \inf_{t\in [0,T']} (\sqrt{g(t)}-B(t))=0 \right)=0\ \textup{and}\  \mathbb{P} \left( \inf_{t\in [0,T']} (B(t)-\sqrt{f(t)})=0 \right)=0.
\end{align*}
Fix $T'>0$. Denote $\mathsf{E}_g=\{\inf_{t\in [0,T']} (\sqrt{g(t)}-B(t))=0\}$ and $\mathsf{E}_f=\{\inf_{t\in [0,T']} (B(t)-\sqrt{f(t)})=0 \}.$ Note that $B(T')$ and the bridge part $B(t)-\frac{T'-t}{T'}B(0)-\frac{t}{T'}B(T')$ are independent. Conditioned on any realization of $B(t)-\frac{T'-t}{L'}B(0)-\frac{t}{T'}B(T')$, there exists a unique value of $B(T')$ such that $\mathsf{E}_g$ occurs. Since $B(T')$ is Gaussian, from the independence, we have $\mathbb{P}(\mathsf{E}_g)=0$. The proof of $\mathbb{P}(\mathsf{E}_f)=0$ is similar. In view of Lemma~\ref{lem:abscontinuous}, it holds that
\begin{align*}
\mathbb{P} \left( \inf_{t\in [0,T']} (g(t)-\mathcal{Q} (t))=0 \right)=0\ \textup{and}\  \mathbb{P} \left( \inf_{t\in [0,T']} (  \mathcal{Q} (t)-f(t))=0 \right)=0.
\end{align*}    
Since $f(T)<y<g(T)$, the assertion follows by taking $T'$ approach $T$.
\end{proof}
\begin{lemma}\label{lem:QQ'}
Fix $T>0$ and $x,x',y,y'> 0$. Let $\mathcal{Q}$ and $\mathcal{Q}'$ be independent squared Bessel bridges on $[0,T]$ with entrance and exit data $(x,y)$ and $(x',y')$ respectively. Then
\begin{align*}
\mathbb{P}\left( \inf_{t\in [0,T]}( \mathcal{Q}'(t)-\mathcal{Q}(t))=0 \right)=0. 
\end{align*}
\end{lemma}
\begin{proof}
Using Lemma~\ref{lem:abscontinuous} to compare the square root of $\mathcal{Q}$ and $\mathcal{Q}'$ with Brownian motions, we have $\mathbb{P}\left( \inf_{t\in [0,T/2]}( \mathcal{Q}'(t)-\mathcal{Q}(t))=0 \right)=0. $ From \eqref{def:q} and \eqref{def:Besslebridgedensity}, squared Bessel bridges are reversible. Therefore, we have $\mathbb{P}\left( \inf_{t\in [T/2,T]}( \mathcal{Q}'(t)-\mathcal{Q}(t))=0 \right)=0.$ Then the assertion follows.
\end{proof}

Through a time translation, we can easily define a Bessel bridge on an interval $[a,b]\subset\mathbb{R}$. The next proposition concerns a coupling of squared Bessel bridges. The proof is postponed to the next subsection.
\begin{proposition}\label{lem:BEcoupling}
Fix an interval $[a,b]\subset\mathbb{R}$. There exists a probability space $(\Omega,\mathbb{P},\mathcal{F})$ and a map $\mathcal{Q}$ from $[0,\infty)^2\times \Omega$ to $C([a,b],\mathbb{R})$ which satisfies the following properties. For each $x,y\in [0,\infty)$, $\mathcal{Q}(x,y,\cdot)$ is $\mathcal{F}$-measurable and is distributed as a Bessel bridge on $[a,b]$ with entrance and exit data $(x,y)$. Moreover, for any sequences $x_j\to x_0$, $y_j\to y_0$ and $\omega\in\Omega$, it holds that 
\begin{equation}\label{equ:Becoupling}
\mathcal{Q}(x_j,y_j,\omega)\ \textup{converges to }\ \mathcal{Q}(x_0,y_0,\omega)\ \textup{uniformly on}\ [a,b]. 
\end{equation}  
\end{proposition}
The next lemma concerns the modulus of continuity (defined in \eqref{def:modu-1}) of squared Bessel bridges.
\begin{lemma}\label{lem:Besselmodulus}
Fix $R,\eta,\rho>0$ and $[a,b]\subset\R$. There exists $r=r(\alpha,R,\eta,\rho,b-a)>0$ such that the following holds. Let $\mathcal{Q}$ be a squared Bessel bridge defined on $[a,b]$ with entrance and exit data $(x,y)\in [0,R]^2$. Then we have
\begin{align*}
\mathbb{P}\left(\omega_{[a,b]}(\mathcal{Q},r) <\rho\right)>1-\eta.
\end{align*}
\end{lemma}
\begin{proof}
Let $(\Omega,\mathbb{P},\mathcal{F})$ and $\mathcal{Q}$ be the probability space and the map given in Proposition~\ref{lem:BEcoupling}. Assume the assertion fails. There exists $R,\eta,\rho>0$ such that the following holds. For any $n\in\mathbb{N}$, there exists $x_n,y_n \in [0,R]$ such that $\mathbb{P}\left( \omega_{[a,b]}( \mathcal{Q}( x_n,y_n), n^{-1}) > 2^{-1}\rho \right)\geq \eta.$ 
This implies for any $m\leq n$, $\mathbb{P}\left( \omega_{[a,b]}(\mathcal{Q}( x_n,y_n), m^{-1}) > 2^{-1}\rho \right)\geq   \eta.$ Without loss of generality, we assume $(x_n,y_n)$ converges to $(x,y)$. From lemma~\ref{lem:uniform2}, $\mathcal{Q}( x_n,y_n)$ converges to $\mathcal{Q}( x,y)$ uniformly. In particular, $\mathcal{Q}( x_n,y_n)$ converges to $\mathcal{Q}( x,y)$ in distribution. This implies for all $m\in\mathbb{N}$,
\begin{align*}
\mathbb{P}\left( \omega_{[a,b]}(\mathcal{Q}( x ,y ), m^{-1}) > 2^{-1}\rho \right)\geq   \eta.
\end{align*}
It is impossible because $\mathcal{Q}( x ,y )$ is a continuous process.
\end{proof}

\begin{lemma}\label{lem:twoBessel}
Fix $R,\eta>0$ and $[a,b]\subset \mathbb{R}$. There exists $\rho=\rho(\alpha,R,\eta,b-a) $ such that the following holds. For any $x,x',y,y'\in [R^{-1},R]$, let $\mathcal{Q}$ and $\mathcal{Q}'$ be independent squared Bessel bridges on $[0,T]$ with entrance and exit data $(x,y)$ and $(x',y')$ respectively. Then
\begin{align*}
\mathbb{P}\left( \inf_{t\in [a,b]}( \mathcal{Q}'(t)-\mathcal{Q}(t))\in (-\rho,\rho)  \right)<\eta. 
\end{align*} 
\end{lemma}
\begin{proof}
Assume the assertion fails. Arguing as in the proof of Lemma~\ref{lem:Besselmodulus}, there exists  $\eta>0$ and independent squared Bessel bridges $\mathcal{Q}$, $\mathcal{Q}'$ with positive entrance and exit data such that $\mathbb{P}\left( \inf_{t\in [a,b]}( \mathcal{Q}'(t)-\mathcal{Q}(t))=0  \right)\geq \eta. $ In view of Lemma~\ref{lem:QQ'}, this is impossible.
\end{proof}
A similar contradiction argument yields the following general lemma.
\begin{lemma}\label{lem:compactcoupling}
Fix an interval $[a,b]\subset\mathbb{R}$. Let $U$ be an open subset of $C([a,b],\mathbb{R})$.  Suppose that $\mathbb{P}(\mathcal{Q}\in U)>0$ for all squared Bessel bridges defined on $[a,b]$. Then for all $R>0$ there exists $A=A(\alpha,R,U,b-a)>0$ such that $\mathbb{P}(\mathcal{Q}\in U)\geq A$ for all squared Bessel bridges defined on $[a,b]$ with entrance and exit data $(x,y)\in [0,R]^2$.
\end{lemma}
\subsection{Proof of Proposition~\ref{lem:BEcoupling}.}
In this section, we prove Proposition~\ref{lem:BEcoupling}. We begin by writing the SDE \eqref{equ:BEBSDE} in the integral form.  $(X,B)$ defined on a filtered probability space $(\Omega,\mathcal{F}_t,\mathbb{P})$ is a weak solution to \eqref{equ:BEBSDE} if the following holds.
\begin{enumerate}
\item $B(t)$ is an $\mathcal{F}_t$-adapted Brownian motion.
\item $X(t), t\in [0,T]$ is a continuous process which is adapted to $\mathcal{F}_t$.
\item Almost surely $X(t)>0$ for all $t\in (0,T)$, $X(T)=y$ and for all $t\in [0,T)$,
\begin{align}\label{equ:SDE_integral}
X(t)=x+ \int_0^t \frac{\alpha+1/2}{X(s)}+r_{T-s}(X(s),y)\, ds+B(t).
\end{align} 
\end{enumerate}

The next lemma concerns the monotonicity of the integral equation~\eqref{equ:SDE_integral}.
\begin{lemma}\label{lem:bridge_monotone}
Fix $x_2\geq x_1\geq 0$, $y_2\geq y_1\geq 0$ and a continuous function $f(t)$. Let $g_1(t)$, ${g}_2(t)$ be two continuous functions such that for $i=1,2$ $g_i(0)=x_i$, $g_i(T)=y_i$ and $g_i(t)>0$ for all $t\in (0,T)$. Suppose for $i=1,2$, $g_i$ satisfies the equation \eqref{equ:SDE_integral} with $(x,y,B(t))$ replaced by $(x_i,y_i,f(t))$. Then $g_2(t)\geq g_1(t)$ for all $t\in [0,T]$.   
\end{lemma} 
\begin{proof}
From the assumption, $h(t)=g_2(t)-g_1(t)$ is differentiable on $(0,T)$ and satisfies the equation
\begin{align*}
h'(t) =-(\alpha+1/2)\frac{h(t)}{g_1(t)g_2(t)}+r_{T-t}(g_2(t) ,y_2)-r_{T-t}(g_1(t) ,y_1).
\end{align*}
We aim to show $h(t)\geq 0$ for all $t\in [0,T]$. From \eqref{p:convex} and $y_2\geq y_1$, we have $ r_{T-t}(g_1(t),y_2)-r_{T-t}(g_1(t) ,y_1) \geq 0$. This implies
\begin{align}\label{equ:SDE_diff}
h'(t) \geq    -(\alpha+1/2)\frac{h(t)}{g_1(t)g_2(t)}+r_{T-t}(g_2(t) ,y_2)-r_{T-t}(g_1(t),y_2).
\end{align}
From \eqref{def:pF}, $r_t(x,y)=-t^{-1}x + t^{-1}y    (h'_\alpha/h_\alpha)\left(t^{-1}  {xy} \right).$ In view of \eqref{equ:hexpansion}, $r_t(x,y)$ can be extended as a smooth function for $(t,x,y)\in (0,\infty)\times\mathbb{R}^2$. 

We divide the discussion into two cases. In case $1$, we assume $x_1>0$. For any $T_0\in[0,T)$, in view of \eqref{equ:SDE_diff} and the smoothness of $r_t(x,y)$, there exists a constant $C(T_0)$ such that for $t\in [0, T_0]$, $h'(t)\geq -C(T_0)|h(t)|$. Together with $h(0)=x_2-x_1\geq 0 $, we have $h(t)\geq 0$ for all $t\in [0,T)$. Note that $h(T)=y_2-y_1\geq 0$. The proof for case $1$ is finished.

In case $2$, we assume $x_1=0$. By the continuity of $g_1(t)$ and the smoothness of $r_t(x,y)$, there exists $\varepsilon>0$ such that for all $t\in (0,\varepsilon)$, 
\begin{align}\label{equ:nearzero}
 (\alpha+1/2)\left|\frac{h(t)}{g_1(t)g_2(t) }\right|\geq 2\left| r_{T-t}(g_2(t) ,y_2)-r_{T-t}(g_1(t) ,y_2) \right|.
\end{align}
We now show that $h(t)\geq 0 $ for $t\in (0,\varepsilon)$. Suppose it fails. There exists some $t_0\in(0,\varepsilon)$ with $h(t_0)<0 $. From \eqref{equ:SDE_diff} and \eqref{equ:nearzero}, we have $h'(t_0)> 0$. It is then simple to show that $h(t)<0$ and $h'(t)> 0$ for all $t\in (0,t_0]$. This implies $x_2-x_1=h(0) \leq h(t_0)<0$, which contradicts the assumption. Having $h(t)\geq 0$ for $t\in (0,\varepsilon)$, an argument similar to case 1 ensures $h(t)\geq 0$ for $t\in [0,T]$. 
\end{proof}
A direct consequence of Lemma~\ref{lem:bridge_monotone} is that \eqref{equ:BEBSDE} enjoys \textit{pathwise uniqueness} (see \cite[Ch. IX, $\S$ 1]{RY} for the definition). From the Yadama-Watanabe theorem \cite[Ch. IX, Theorem 1.7]{RY}, every weak solution to \eqref{equ:BEBSDE} is also a strong solution. Moreover, there exists a Borel measurable map $\Phi^{x,y}$ from  $C([0,\infty),\mathbb{R})$ to $C([0,T],\mathbb{R} )$ such that for any Brownian motion $\tilde{B}$, $(\Phi^{x,y}(\tilde{B}),\tilde{B})$ is a strong solution to \eqref{equ:BEBSDE}.

From now on, we fix a filtered probability space $(\Omega,\mathcal{F}_t,\mathbb{P})$ and an adapted Brownian motion $B(t)$. For $x,y\in [0,\infty)\cap \mathbb{Q}$, we define $\mathcal{S}^{x,y}\coloneqq \Phi^{x,y}(B)$. Let $\Omega_0$ be a full measure subset of $\Omega$ such that for all $\omega\in\Omega_0$ and $x,y\in \mathbb{Q}\cap [0,\infty)$, \eqref{equ:SDE_integral} holds with $X(t)$ replaced by $\mathcal{S}^{x,y}(t)$.

\begin{lemma}\label{lem:uniform}
Let $x_n, y_n$ be two convergent sequences of non-negative rational numbers with $x=\lim_{n\to\infty} x_n$ and $y=\lim_{n\to\infty} y_n$. Then for all $\omega\in \Omega_0$, $\mathcal{S}^{x_n,y_n}(\omega)$ converges uniformly on $[0,T]$. The limit solves \eqref{equ:SDE_integral}. Moreover, the limit depends only on $x$ and $y$ but not on the sequence $x_n$ or $y_n$. 
\end{lemma}  
\begin{proof}
Throughout the proof, we fix $\omega\in\Omega_0$. We denote $\mathcal{S}^{x,y}(\omega)$ by $\mathcal{S}^{x,y}$ for brevity. Let $x^\pm_n,y^\pm_n,$ be sequences of non-negative rational numbers such that 
\begin{itemize}
\item $x^+_n$, $y^+_n$ are non-increasing and $x^-_n$, $y^-_n$ are non-decreasing.
\item $x^-_n\leq x_n\leq x^+_n$, $y^-_n\leq y_n\leq y^+_n$ for all $n$.
\item $\lim_{n\to\infty} x_n^\pm=x$ and $\lim_{n\to\infty} y_n^\pm=y$.
\end{itemize}
From Lemma~\ref{lem:bridge_monotone}, $\mathcal{S}^{x_n^+,y_n^+}$ is non-increasing in $n$ and $\mathcal{S}^{x_n^-,y_n^-}$ is non-decreasing in $n$. This implies for all $t\in [0,T],$ $\mathcal{S}^\pm(t) \coloneqq \lim_{n\to\infty}\mathcal{S}^{x_n^\pm, y_n^\pm}(t)$ exist. By the dominated convergence theorem, $\mathcal{S}^\pm(t) $ both satisfy \eqref{equ:SDE_integral}. This implies $\mathcal{S}^\pm(t)$ are continuous on $t\in (0,T)$. From $\mathcal{S}^{x_n^-,y_n^-}(t)\leq \mathcal{S}(t)^\pm\leq \mathcal{S}^{x_n^+,y_n^+}(t)$, we have $x_n^-\leq \liminf_{t\to 0}\mathcal{S}^\pm(t)$ and $ \limsup_{t\to 0}\mathcal{S}^\pm(t)\leq x_n^+$. This ensures $\mathcal{S}^\pm(t)$ are continuous at $t=0$. The continuity of $\mathcal{S}^\pm(t)$ at $t=T$ can be similarly derived. We can then apply Lemma~\ref{lem:bridge_monotone} to get $\mathcal{S}^+(t)=\mathcal{S}^-(t)=:\mathcal{S}(t)$.

From Dini's theorem, both $\mathcal{S}^{x_n^\pm, y_n^\pm} $ converge to $\mathcal{S} $ uniformly. Together with $\mathcal{S}^{x_n^-,y_n^-}(t)\leq \mathcal{S}^{x_n,y_n}(t)\leq \mathcal{S}^{x_n^+,y_n^+}(t)$, the uniform convergence of $\mathcal{S}^{x_n,y_n}$ follows. The fact that $\mathcal{S}(t)$ depends only on $x$ and $y$ but not on $x_n$ or $y_n$ follows Lemma~\ref{lem:bridge_monotone}.
 
\end{proof}

Next, we define $\overline{\mathcal{S}}^{x,y}(t)$ on $\Omega$ for all $x,y\geq 0$ through approximation. For any $x,y\geq 0 $, let $(x_n,y_n)$ be a sequence of non-negative rational pairs which converges to $(x,y)$. Define 
\begin{equation}\label{equ:approximation}
\overline{\mathcal{S}}^{x,y} (\omega)\coloneqq \left\{ \begin{array}{cc}
\lim_{n\to\infty} \mathcal{S}^{x_n,y_n} (\omega), & \omega\in\Omega_0,\\
\textup{linear interpolation between}\ x\ \textup{and}\ y  , & \omega\notin\Omega_0
\end{array} \right. 
\end{equation}
Lemma~\ref{lem:uniform} guarantees that the limit exists and that the limit does not depend on how we chose the sequence. In particular, $\overline{\mathcal{S}}^{x,y}(\omega)=\mathcal{S}^{x,y}(\omega)$ for all $\omega\in\Omega_0$ and $x,y\in\mathbb{Q}\cap [0,\infty) $.	 

The next lemma shows that for all $\omega\in\Omega$, $\overline{\mathcal{S}}^{x,y}$ depends continuously on $x,y$. The proof is similar to the one for Lemma~\ref{lem:uniform} and we omit the details.

\begin{lemma}\label{lem:uniform2}
Let $x_n, y_n$ be two convergent sequences of non-negative numbers with $x=\lim_{n\to\infty} x_n$ and $y=\lim_{n\to\infty} y_n$. Then for all $\omega\in \Omega$, $\overline{\mathcal{S}}^{x_n,y_n}(\omega)$ converges uniformly to $\overline{\mathcal{S}}^{x,y}(\omega)$ on $[0,T]$.
\end{lemma}  
\begin{proof}[Proof of Proposition~\ref{lem:BEcoupling}]
Without loss of generality, we may assume $a=0$ and $b=T$. Let $ \overline{\mathcal{S}}^{x,y}(t)$, $x,y\in [0,\infty)$ be a family of random curves coupled in the same probability space $(\Omega,\mathcal{F}_t,\mathbb{P})$ constructed above. From Lemma~\ref{lem:uniform}, $(\overline{\mathcal{S}}^{x,y},B)$ is a weak solution to \eqref{equ:BEBSDE} for all $x,y\geq 0$. In particular, $\overline{\mathcal{S}}^{x,y} $ is distributed as a Bessel bridge with entrance and exit data $(x,y)$. Therefore, defining $\mathcal{Q}( x,y)\coloneqq \left(\overline{\mathcal{S}}^{\sqrt{x},\sqrt{y}}\right)^2$ gives a coupling of squared Bessel bridges. The uniform convergence is ensured by Lemma~\ref{lem:uniform2}. 
\end{proof}

\section{Basics of line ensembles}\label{sec:Basics of line ensembles}
In this section we introduce basic notions necessary to define the squared Bessel Gibbs property.

\begin{definition}[Line ensembles]\label{def:line-ensemble}
Let $\Sigma\subset\mathbb{N}$ and $\Lambda\subset\mathbb{R}$ be intervals that are closed. A $\Sigma\times\Lambda $-indexed line ensemble $\mathcal{L}$ is a random variable on a probability space $(\Omega,\mathcal{F},\mathbb{P})$, taking values in $C(\Sigma\times\Lambda,\R)$ such that $\mathcal{L}$ is a measurable function from $\mathcal{F}$ to $\mathcal{C}(\Sigma\times\Lambda,\R)$.
\end{definition}

We think of such line ensembles as multi-layer random curves. We will generally write $\mathcal{L}:\Sigma\times\Lambda\to \R$ even though it is not $\mathcal{L}$, but rather $\mathcal{L}(\omega)$ for each $\omega\in \Omega$ which is such a function. We will also sometimes specify a line ensemble by only giving its law without reference to the underlying probability space.  We write $\mathcal{L}_i(\cdot):= \big(\mathcal{L}(\omega)\big)(i,\cdot)$ for the label $i\in \Sigma$ curve of the ensemble $\mathcal{L}$. 

\begin{definition}[Convergence of line ensembles]\label{def:weakconvergence} 
Given a $\Sigma\times\Lambda$-indexed line ensemble $\mathcal{L}$ and a sequence of such ensembles $\left\{\mathcal{L}^N\right\}_{N\geq 1}$, we will say that $\mathcal{L}^N$ converges to $\mathcal{L}$ weakly
as a line ensemble if for all bounded continuous functions $F:C(\Sigma\times\Lambda,\R)\to \R$, it holds that as $N\to \infty$,
	\begin{equation*}
		\int F\big(\mathcal{L}^N(\omega)\big)d\mathbb{P}^N(\omega) \to \int F\big(\mathcal{L}(\omega)\big) d\mathbb{P}(\omega).
	\end{equation*}
This is equivalent to weak-$*$ convergence in $C(\Sigma\times\Lambda ,\R)$ endowed with the topology of uniform convergence on compact subsets of $\Sigma\times\Lambda$.
\end{definition}

The following definition gives the class of functions $(f,g)$ which will serve as the upper and lower boundary data for non-intersecting squared Bessel bridges. Even though we allow $f$ and $g$ to have certain discontinuity below, in most cases $f$ and $g$ will be continuous functions.

\begin{definition}\label{def:continuityassumption}
A pair of functions $(f,g)$ defined on $[a,b]$ satisfies the continuity assumption if the following statements hold. First, $f:[a,b]\to [0,\infty)$ is upper semi-continuous and $g:[a,b]\to (0,\infty]$ is lower semi-continuous. Second, $f$ and $g$ are continuous at $a$ and $b$. Third, for all $t\in (a,b)$, $f$ and $g$ are one-sided continuous at $t$.
\end{definition}

We now start to formulate the squared $\alpha$-Bessel Gibbs property. 

\begin{definition}[Squared $\alpha$-Bessel bridge ensemble]\label{def:Bessel bridge LE}
Fix $k_1\leq k_2$ with $k_1,k_2 \in \mathbb{N}$, an interval $[a,b]\subset \mathbb{R}$ and two vectors $\vec{x},\vec{y}\in \mathbb{R}_+^{k_2-k_1+1}$. A $[k_1,k_2]_{\Z}\times [a,b]$-indexed line ensemble $\mathcal{L} = (\mathcal{L}_{k_1},\ldots,\mathcal{L}_{k_2})$ is called a free squared $\alpha$-Bessel bridge ensemble with entrance data $\vec{x}$ and exit data $\vec{y}$ if its law $\mathbb{P}^{k_1,k_2,[a,b],\vec{x},\vec{y}}_{\free}$ is that of $k_2-k_1+1$ independent squared $\alpha$-Bessel bridges defined on $[a,b]$ with entrance and exit data $(\vec{x},\vec{y})$. 

Let $(f,g)$ be a pair of functions defined on $[a,b]$ which satisfies the continuity assumption. The normalizing constant is the following non-intersecting probability 
\begin{equation}\label{eqn:normalcont_Bessel}
Z^{k_1,k_2,(a,b),\vec{x},\vec{y},f,g} :=\mathbb{E}^{k_1,k_2,(a,b),\vec{x},\vec{y}}_{\free}\Big[\mathbbm{1} \{f<\mathcal{J}_{k_1}<\cdots<\mathcal{J}_{k_2}<g \ \textup{on}\ [a,b]\}\Big],
\end{equation}
where $\mathcal{J}$ in the above expectation is distributed according to the measure $\mathbb{P}^{k_1,k_2,(a,b),\vec{x},\vec{y}}_{\free}$. If the normalizing constant is positive, we define the non-intersecting squared $\alpha$-Bessel bridge ensemble with entrance data $\vec{x}$, exit data $\vec{y}$ and boundary data $(f,g)$ to be a $[k_1,k_2]_{\Z}\times [a,b]$-indexed line ensemble with law $\mathbb{P}^{k_1,k_2,(a,b),\vec{x},\vec{y},f,g}$ given according to the following Radon-Nikodym derivative relation:
\begin{equation}\label{eqn:RN}
\frac{\textup{d}\mathbb{P}^{k_1,k_2,(a,b),\vec{x},\vec{y},f,g}}{\textup{d}\mathbb{P}_{\free}^{k_1,k_2,(a,b),\vec{x},\vec{y}}}(\mathcal{J}) := \frac{\mathbbm{1} {\{f<\mathcal{J}_{k_1}<\cdots<\mathcal{J}_{k_2}<g\ \textup{on}\ [a,b]\}}}{Z^{k_1,k_2,(a,b),\vec{x},\vec{y},f,g}}.
\end{equation}

Moreover, given $a'<b'$ contained in $(a,b)$, we define
\begin{equation}\label{def:normalcont_Brownian_t}
Z_{(a',b')}^{k_1,k_2,(a,b),\vec{x},\vec{y},f,g} :=\mathbb{E}^{k_1,k_2,(a,b),\vec{x},\vec{y}}_{\free}\Big[\mathbbm{1} \{f<\mathcal{J}_{k_1}<\cdots<\mathcal{J}_{k_2}<g \ \textup{on}\ [a,a']\cup [b',b]\}\big],
\end{equation}
where $\mathcal{J}$ in the above expectation is distributed according to the measure $\mathbb{P}^{k_1,k_2,(a,b),\vec{x},\vec{y}}_{\free}$.
That is, the non-intersecting property is only required on $[a,a']\cup [b',b]$ but not on $(a',b')$. We similarly define 
\begin{equation*}
\frac{\textup{d}\mathbb{P}^{k_1,k_2,(a,b),\vec{x},\vec{y},f,g}_{(a',b')}}{\textup{d}\mathbb{P}_{\free}^{k_1,k_2,(a,b),\vec{x},\vec{y}}}(\mathcal{J}) := \frac{\mathbbm{1} {\{f<\mathcal{J}_{k_1}<\cdots<\mathcal{J}_{k_2}<g\ \textup{on}\ [a,a']\cup [b',b]\}}}{Z_{(a',b')}^{k_1,k_2,(a,b),\vec{x},\vec{y},f,g}}.
\end{equation*}
Note that ${\textup{d}\mathbb{P}_{(a',b')}^{k_1,k_2,(a,b),\vec{x},\vec{y},f,g}}\big/{\textup{d}\mathbb{P}^{k_1,k_2,(a,b),\vec{x},\vec{y},f,g}}(\mathcal{J})$ is proportional to $\mathbbm{1} {\{f<\mathcal{J}_{k_1}<\cdots<\mathcal{J}_{k_2}<g\ \textup{on}\ [a',b']\}}.$ This feature will be used in Section~\ref{sec:proofofZ}. 
\end{definition}

\begin{remark}
The normalizing constant is positive when the boundary values are ordered. That is, $f(a)< {x}_{k_1}<\cdots <  {x}_{k_2}<g(a)$, $f(b)< {y}_{k_1} <\cdots<  {y}_{k_2}<g(b)$ and $f<g$ on $[a,b]$ implies $Z^{k_1,k_2,(a,b),\vec{x},\vec{y},f,g}>0.$
\end{remark}

The squared $\alpha$-Bessel Gibbs property could be viewed as a spatial Markov property. More specifically, it provides a description of the conditional law inside a compact set.
\begin{definition}[Squared $\alpha$-Bessel Gibbs property]\label{def:BesselGP}
Fix $\alpha\geq 0$. A $\Sigma \times \Lambda$-indexed line ensemble $\mathcal{L}$ satisfies the  squared $\alpha$-Bessel Gibbs property if for all $[k_1,k_2]_{\mathbb{Z}} \subset \Sigma$ and $[a,b]\subset \Lambda$, its conditional law inside $[k_1,k_2]_{\mathbb{Z}}\times [a,b]$ takes the following form,
\begin{equation}\label{Gibbs-condition}
\textrm{Law\ of}\left(\mathcal{L} \left\vert_{[k_1,k_2]_{\mathbb{Z}} \times [a,b]} \textrm{conditional on } \mathcal{L}\right\vert_{(\Sigma \times \Lambda) \setminus ( [k_1,k_2]_{\mathbb{Z}} \times (a,b) )} \right) =\bP^{k_1,k_2,(a,b),\vec{x},\vec{y},f,g}
\end{equation}
Here $f =\mathcal{L}_{k_1-1}$ and $g =\mathcal{L}_{k_2+1}$ with the convention that if $k_1-1\notin\Sigma$ then $f\equiv 0$ and likewise if $k_2+1\notin \Sigma$ then $g\equiv +\infty$; we have also set $\vec{x}=\big(\mathcal{L}_{k_1}(a),\cdots ,\mathcal{L}_{k_2}(a)\big)$ and $\vec{y}=\big(\mathcal{L}_{k_1}(b),\ldots ,\mathcal{L}_{k_2}(b)\big)$. 
\end{definition}
The following description of Gibbs property using conditional expectation is equivalent. We will make use of it as it is convenient for writing the arguments. A $\Sigma \times \Lambda$-indexed line ensemble $\mathcal{L}$ enjoys the squared $\alpha$-Bessel Gibbs property if and only if for any $[k_1,k_2]_{\mathbb{Z}}\subset \Sigma$, $(a,b)\subset \Lambda$, and bounded Borel function $F:C\left([k_1,k_2]_{\mathbb{Z}}\times[a,b], \R\right)\to \R$, it holds $\mathbb{P}$-almost surely that
\begin{equation}\label{Gibbs-algebra}
\E\left[F\left(\mathcal{L}|_{[k_1,k_2]_{\mathbb{Z}}\times [a,b]} \right) \right\vert \Fext\big([k_1,k_2]_{\mathbb{Z}}\times (a,b)\big)\Big] =  \mathbb{E}^{k_1,k_2,(a,b),\vec{x},\vec{y},f,g}\left[F(\mathcal{J}_{k_1},\ldots, \mathcal{J}_{k_2})\right],
\end{equation}
where $\vec{x}$,$\vec{y}$,$f$ and $g$ are defined as in Definition \ref{def:BesselGP} and where \glossary{$\Fext\left(K\times (a,b)\right)$, Sigma-field generated by a line ensemble outside $[k_1,k_2]_{\mathbb{Z}}\times (a,b)$}
\begin{equation}\label{eqn:exterior-signma}
\Fext\left([k_1,k_2]_{\mathbb{Z}}\times (a,b)\right) := \sigma\left(\mathcal{L}_{i}(s): (i,s)\in \Sigma\times \Lambda \setminus [k_1,k_2]_{\mathbb{Z}}\times (a,b)\right)
\end{equation}
is the exterior sigma-field generated by the line ensemble outside $[k_1,k_2]_{\mathbb{Z}}\times (a,b)$. On the left-hand side of the above equality  $\mathcal{L}|_{[k_1,k_2]_{\mathbb{Z}}\times [a,b]}$ is the restriction to $[k_1,k_2]_{\mathbb{Z}}\times[a,b]$ of curves distributed according to $\mathbb{P}$, while on the right-hand side $\mathcal{J}_{k_1},\ldots, \mathcal{J}_{k_2}$ are curves on $[a,b]$ distributed according to $\mathbb{P}^{k_1,k_2,(a,b),\vec{x},\vec{y},f,g}$. 

We finish this subsection by the squared $\alpha$-Bessel process of $\mathcal{L}^{N,\alpha}$.
\begin{proposition}\label{pro:GibbsforL}
For any $\alpha\geq 0$ and $N\in\mathbb{N}$, the line ensemble $\mathcal{L}^{N,\alpha}$ defined in \eqref{eqn:edgeLE2} satisfies the squared $\alpha$-Bessel Gibbs property.
\end{proposition}
\begin{proof}
Recall that $Y^{N,\alpha}=(Y^{N,\alpha}_1(t),\dots, Y^{N,\alpha}_N(t))$ is the non-intersecting squared $\alpha$-Bessel process. This implies $Y^{N,\alpha}$ satisfies the squared $\alpha$-Bessel Gibbs property. The assertion then follows by combining \eqref{eqn:edgeLE1} and the scaling property of Bessel processes.
\end{proof}

We finish this section with the strong Gibbs property, which enables us to resample the trajectory within a stopping domain as opposed to a deterministic interval. 

\begin{definition}\label{def:stopdm}
Let $\Sigma\subset\mathbb{N}$ and $\Lambda\subset\mathbb{R}$ be intervals that are closed. Consider a $\Sigma\times\Lambda$-indexed line ensemble $\mathcal{L}$. For $[k_1,k_2]_{\mathbb{Z}}\subset \Sigma$, the random variable $(\mathfrak{l},\mathfrak{r})$ \glossary{$(\mathfrak{l},\mathfrak{r})$, Stopping domain} is called a {\it $[k_1,k_2]_{\mathbb{Z}}$-stopping domain} if for all $\ell<r$,
\begin{equation*}
\big\{\mathfrak{l} \leq \ell , \mathfrak{r}\geq r\big\} \in \Fext\big([k_1,k_2]_{\mathbb{Z}}\times (\ell,r)\big).
\end{equation*}
\end{definition}

The strong Gibbs property for squared $\alpha$-Bessel Gibbsian line ensembles follows from the same argument of the strong Gibbs property for Brownian Gibbsian line ensemble. We omit the proof here and refer the readers to \cite[Lemma 2.5]{CH16}. Define
\begin{align*}
C(k_1,k_2):=\left\{ (\ell,r,f_{k_1},\dots,f_{k_2})\,:\, \ell<r,\ (f_{k_1},\dots,f_{k_2})\in C([k_1,k_2]_{\mathbb{Z}}\times [\ell,r],\mathbb{R}) \right\}.
\end{align*}
We equip $C(k_1,k_2)$ with the topology induced by the restriction map $\mathbb{R}\times\mathbb{R}\times C([k_1,k_2]_{\mathbb{Z}}\times\mathbb{R},\mathbb{R})$.
\begin{lemma}[{\bf Strong Gibbs property}]\label{lem:stronggibbs}
Consider a $\Sigma\times\Lambda$-indexed line ensemble $\mathcal{L}$ which has the squared $\alpha$-Bessel Gibbs property. Fix $[k_1,k_2]_{\mathbb{Z}}\subset \Sigma$. For all random variables $(\mathfrak{l},\mathfrak{r})$ which are $[k_1,k_2]_{\mathbb{Z}}$-stopping domains for $\mathcal{L}$, the following strong squared $\alpha$-Bessel Gibbs property holds: for any bounded Borel function $F: C(k_1,k_2)\to\mathbb{R}$, it holds $\bP$-almost surely that
\begin{equation}\label{eqn:stronggibbs}
\E\bigg[ F\Big(\mathfrak{l},\mathfrak{r}, \mathcal{L}|_{[k_1,k_2]_{\mathbb{Z}}\times (\mathfrak{l},\mathfrak{r})} \Big) \Big\vert \Fext\big(K\times [\mathfrak{l},\mathfrak{r}]\big) \bigg]
=\bP^{k_1,k_2,(\mathfrak{l},\mathfrak{r}),\vec{x},\vec{y},f,g}\Big[F\big(\mathfrak{l} ,\mathfrak{r}, \mathcal{J}_{k_1},\ldots, \mathcal{J}_{k_2}\big) \Big].
\end{equation}
Here $\vec{x} = \{\mathcal{L}_i(\mathfrak{l})\}_{i=k_1}^{k_2}$, $\vec{y} = \{\mathcal{L}_i(\mathfrak{r})\}_{i=k_1}^{k_2}$, $f =\mathcal{L}_{k_1-1} $ (or $0$ if $k_1-1\notin \Sigma$), $g =\mathcal{L}_{k_2+1} $ (or $+\infty$ if $k_2+1\notin \Sigma$). On the left-hand side $\ \mathcal{L}|_{[k_1,k_2]_{\mathbb{Z}}\times [\mathfrak{l},\mathfrak{r}]}$ is the restriction of curves distributed according to $\bP$ and on the right-hand side $\mathcal{J}_{k_1},\ldots, \mathcal{J}_{k_2}$ is distributed according to $\bP^{k_1,k_2,(\mathfrak{l},\mathfrak{r}),\vec{x},\vec{y},f,g}$.
\end{lemma}

\section{Stochastic monotonicity}\label{sec:Stochastic monotonicity}
In this section we prove the stochastic monotonicity, Proposition~\ref{lem:monotonicity}. Stochastic monotonicity will play a crucial role in Sections~\ref{sec:uniformbounds} and \ref{sec:proofofZ} because it provides a method to bound the probability of certain events for Gibbsian line ensembles without the knowledge of the normalizing constant.
\begin{proposition}[{\bf Stochastic monotonicity}] \label{lem:monotonicity}
Fix $\alpha\geq 0$, two positive integers $k_1\leq k_2$ and $(a, b)\subset \R$. For $i\in \{1,2\}$, let $(\vec{x}^{(i)},\vec{y}^{(i)})\in \R^{k_2-k_1+1}_+\times\R_+^{k_2-k_1+1}$ be vectors which satisfy  
\begin{align*}
x^{(1)}_j\geq x^{(2)}_j,\ y^{(1)}_j\geq y^{(2)}_j\ \textup{for all}\ j\in [k_1,k_2]_{\mathbb{Z}}. 
\end{align*}
For $i\in \{1,2\}$, let $(f^{(i)},g^{(i)})$ be functions on $[a,b]$ that satisfy the continuity assumption in Definition \ref{def:continuityassumption}. Assume that 
\begin{align*}
f^{(1)}(t) \geq f^{(2)}(t),\ g^{(1)}(t)\geq g^{(2)}(t)\ \textup{for all}\ t\in [a,b] 
\end{align*} 
Further assume that for $i\in \{1,2\}$,
\begin{align}\label{assump:Z} 
Z^{k_1,k_2,(a,b),\vec{x}^{(i)},\vec{y}^{(i)},f^{(i)},g^{(i)}}>0.
\end{align}
See \eqref{eqn:normalcont_Bessel} for the definition of $Z^{k_1,k_2,(a,b),\vec{x}^{(i)},\vec{y}^{(i)},f^{(i)},g^{(i)}}.$ \\[-0.25cm]

Let $\mathcal{Q}^{(i)}=\{\mathcal{Q}^{(i)}_j\}_{j=k_1}^{k_2}$ be a $[k_1,k_2]_{\mathbb{Z}} \times [a,b]$-indexed line ensemble whose law $\mathbb{P}^{(i)}$ is given by the non-intersecting squared $\alpha$-Bessel bridge line ensemble $\mathbb{P}^{k_1,k_2,(a,b),\vec{x}^{(i)},\vec{y}^{(i)},f^{(i)},g^{(i)}}$ (See Definition \ref{def:Bessel bridge LE}). Then there exists a coupling of the probability measures $\mathbb{P}^{(1)}$ and $\mathbb{P}^{(2)}$ such that almost surely $$\mathcal{Q}^{(1)}_j(t)\geq \mathcal{Q}^{(2)}_j(t)\ \textup{for all}\ j\in [k_1,k_2]_{\mathbb{Z}}\ \textup{and}\ t\in [a,b].$$
\end{proposition}

\begin{proof}
Fix $\alpha\geq 0$.  Without loss of generality, we assume $[a,b]=[0,1]$ and $k_1=1,\ k_2=k$. Take two sets of boundary data $(\vec{x}^{(i)},\vec{y}^{(i)}, f^{(i)}, g^{(i)}), i\in\{1,2\}$ as described in Proposition~\ref{lem:monotonicity}. For $i\in\{1,2\}$, let $\mathcal{Q}^{(i)} $ be a $[1,k]_{\mathbb{Z}} \times [0,1]$-indexed line ensemble of the law $\mathbb{P}^{(i)}=\mathbb{P}^{1,k,(0,1), \vec{x}^{(i)}, \vec{y}^{(i)},f^{(i)},g^{(i)}}.$ We seek to couple $\mathbb{P}^{(1)}$ and $\mathbb{P}^{(2)}$ in the same probability space such that almost surely, $\mathcal{Q}_j^{(1)}(t) \geq \mathcal{Q}_j^{(2)}(t)$ for all $j\in [1,k]_{\mathbb{Z}}$ and $t\in [0,1]$. For technical reasons, it is easier to couple the corresponding  \textbf{Bessel bridges}. We denote by $\mathbb{Q}^{(i)}$ the measures on $C([1,k]_{\mathbb{Z}}\times [0,1],\R)$ which are the push-forward measures of $\mathbb{P}^{(i)}$ under the map $h_j(t)\mapsto \sqrt{|h_j(t)|}$. We split the argument into four steps.\\[-0.3cm]

{\bf Step one}: we discretize the state space $C([1,k]_{\mathbb{Z}}\times [0,1],\mathbb{R})$. Fix $\ell\in\N$ and $M\geq 1$. Define 
\begin{align}\label{equ:OmegaMl}
\Omega_{M,\ell}\coloneqq \left\{  \mathbf{z}=(z_{j,n} )   \, \Big|\, j\in [1,k]_{\mathbb{Z}},\ n\in [1,2^{\ell}-1]_{\mathbb{Z}},\  z_{j,n}\in M^{-1}\mathbb{Z}\cap [0,M] \right\}.
\end{align}
We assign weights $W^{(i)}, i\in\{1,2\}$ to members of $\Omega_{M,\ell}$. Let $p(x,y)=p_{2^{-\ell} }(x,y)$ be the transition density function of the $\alpha$-{Bessel process} with time displacement $2^{-\ell} $. See \eqref{def:pF} for the explicit form of $p_t(x,y)$. We omit the subscript $ {2^{-\ell}}$ and set $K=2^{\ell}$ for simplicity.  For any $ \mathbf{z}^{(i)}=(z^{(i)}_{j,n})\in \Omega_{M,\ell}$, define 
\begin{align}\label{def:weightfree}
W_{\free}^{(i)}(\mathbf{z}^{(i)}):=\prod_{j=1}^{k}\prod_{n=1}^{K}p\left(   z^{(i)}_{j,n-1}, z^{(i)}_{j,n}\right),\ W^{(i)}(\mathbf{z}^{(i)}):=W_{\free}^{(i)}(\mathbf{z}^{(i)}) G^{(i)}(\mathbf{z}^{(i)}), 
\end{align}
where
\begin{align*}
G^{(i)}(\mathbf{z}^{(i)})\coloneqq  \mathbbm{1}\left\{\sqrt{f^{(i)}}(n/K)<z_{1,n}^{(i)}<z_{2,n}^{(i)}<\dots <z_{k,n}^{(i)}<\sqrt{g^{(i)}}(n/K)\ \textup{for all}\ n\in [1,K-1]_{\mathbb{Z}} \right\}.
\end{align*}
Here we adopt the convention $z^{(i)}_{j,0}=\sqrt{x^{(i)}_j}$ and $z^{(i)}_{j,K}=\sqrt{y^{(i)}_j}$.

Define $\mathbf{z}^{(i),\min} =({z}^{(i),\min}_{j,n}) $, $j\in [1,k]_{\mathbb{Z}}$ and $n\in [1,K-1]_{\mathbb{Z}}$ by
\begin{equation}\label{def:07080819}
\begin{split}
z^{(i),\min}_{1,n}\coloneqq &\min\{ mM^{-1}\in M^{-1}\mathbb{Z}\, |\, mM^{-1}>\sqrt{f^{(i)}}(n/K) \},\\
z^{(i),\min}_{j,n}\coloneqq &z^{(i),\min}_{1,n}+(j-1)M^{-1}
\end{split}
\end{equation}
From $f^{(1)}\geq f^{(2)}$, we have $z^{(1),\min}_{j,n}\geq z^{(2),\min}_{j,n}$. Because of the assumption \eqref{assump:Z}, there exists $M_0$ such that for all $M\geq M_0$, we have $\mathbf{z}^{(i),\min}\in \Omega_{M,\ell}$ and $W^{(i)}(\mathbf{z}^{(i),\min})>0$. From now on we always assume $M\geq M_0$. We write $\mathbb{Q}^{(i)}_{M,\ell}$, $i\in\{1,2\}$, for the probability measures on $\Omega_{M,\ell}$ such that
\begin{align*}
\mathbb{Q}^{(i)}_{M,\ell}(\mathbf{z} )\propto W^{(i)}(\mathbf{z}).
\end{align*} 
Note that $\Omega_{M,\ell}$ can be identified as a subset of $C([1,k]_{\mathbb{Z}}\times [0,1],\mathbb{R})$ through 
\begin{align}\label{OmegaMktoC}
h^{(i)}_j(t)=\left\{\begin{array}{cc}
z^{(i)}_{j,n}, & t=n/K, n\in [0,K]_{\mathbb{Z}}\\
\textup{linear interpolation}, &\ \textup{others}. 
\end{array} \right.
\end{align} 
Therefore, $\mathbb{Q}^{(i)}_{M,\ell}$ can be viewed as probability measures on $C([1,k]_{\mathbb{Z}}\times [0,1],\mathbb{R})$.\\

{\bf Step two}: we construct Markov chains to couple $\mathbb{Q}^{(1)}_{M,\ell}$ and $\mathbb{Q}^{(2)}_{M,\ell}$. We introduce the dynamics by assigning Poisson clocks at $[1,k]_{\mathbb{Z}}\times [1,K-1]_{\mathbb{Z}} \times\{+,-\}$. If the clock at $(j, n,+)$ rings, we attempt to increase $z^{(i)}_{j,n}$ to $\min\{z^{(i)}_{j,n}  + M^{-1},M\}$. Define $\mathbf{z}^{(i),+}$ by  replacing  $z^{(i)}_{j,n}$ with $\min\{z^{(i)}_{j,n}  + M^{-1},M\}$.
\begin{align}\label{def:z+}
z^{(i),+}_{j',n'}\coloneqq \left\{\begin{array}{cc}
\min\{z^{(i)}_{j,n}  + M^{-1},M\}, & (j',n')=(j,n),\\
z^{(i)}_{j',n'}, & \textup{others}.
\end{array} \right.
\end{align} 
Take a family of independent uniform random variables $U^{j,n,s}$ with support $(0,1)$. Here $(j,n,s)\in [1,k]_{\mathbb{Z}}\times [1,K-1]_{\mathbb{Z}}\times\{-,+\}$. If $ {W^{(i)}(\mathbf{z}^{(i),+})}/{W^{(i)}(\mathbf{z}^{(i)})} \geq U^{j,n,+}$, we update $\mathbf{z}^{(i)}$ to $\mathbf{z}^{(i),+}$; otherwise we make no change. Similarly, if the clock at $(j,n, -)$ rings, we attempt to decrease $z^{(i)}_{j,n}$ to $\max\{z^{(i)}_{j,n}  - M^{-1},0\}$. Define $\mathbf{z}^{(i),-}$ by replacing $z^{(i)}_{j,n}$ with $\max\{z^{(i)}_{j,n}  - M^{-1},0\}$. It follows that
\begin{align}\label{def:z-}
z^{(i),-}_{j',n'}\coloneqq \left\{\begin{array}{cc}
\max\{z^{(i)}_{j,n}  - M^{-1},0\}, & (j',n')=(j,n),\\
z^{(i)}_{j',n'}, & \textup{others}.
\end{array} \right.
\end{align}  
Such an update is accepted if ${W^{(i)}(\mathbf{z}^{(i),-})}/{W^{(i)}(\mathbf{z}^{(i)})}\geq U^{j,n,-}$.\\[-0.3cm]

We run two Markov chains ${\bf z}^{(1)}, {\bf z}^{(2)}$ on $\Omega_{M,\ell}$ as described above with initial configurations $z^{(i),\textup{int}}$, $i\in\{1,2\}$, respectively. These two Markov chains ${\bf z}^{(1)}, {\bf z}^{(2)}$ are coupled through the same Poisson clocks and the same collection of uniform random variables. It could be directly checked that $ \mathbb{Q}^{(1)}_{M,\ell}$ and $\mathbb{Q}^{(2)}_{M,\ell}$ are invariant measures for ${\bf z}^{(1)}, {\bf z}^{(2)}$ respectively. Moreover, $\mathbf{z}^{(i),\min}$ defined in \eqref{def:07080819} belongs to every communication class of $\mathbf{z}^{(i)}$. This implies ${\bf z}^{(1)}$ and ${\bf z}^{(2)}$ are irreducible. As a result, $\mathbb{Q}^{(1)}_{M,\ell}$ and $\mathbb{Q}^{(2)}_{M,\ell}$ are coupled by taking the limits of these two Markov processes.\\

{\bf Step three}: we show that $\mathbb{Q}^{(1)}_{M,\ell}$ dominates $\mathbb{Q}^{(2)}_{M,\ell}$ almost surely under the above coupling. We make the following Claim \ref{clm:q-convex}. In view of Claim \ref{clm:q-convex}, ${\bf z}^{(1)}$ always dominates ${\bf z}^{(2)}$. As a result, $\mathbb{Q}^{(1)}_{M,\ell}$ dominates $\mathbb{Q}^{(2)}_{M,\ell}$. 
\begin{claim}\label{clm:q-convex}
Fix $j\in [1,K-1]_{\mathbb{Z}}$. Let $\mathbf{z}^{(1)}$, $\mathbf{z}^{(2)}\in\Omega_{M,k}$ be two configurations which satisfy $W^{(i)}(\mathbf{z}^{(i)})>0$. We further assume that for all $(j',n')\in [1,k]_{\mathbb{Z}}\times [1,K-1]_{\mathbb{Z}}$, $z^{(1)}_{j',n'}\geq z^{(2)}_{j',n'}$ and that $z^{(1)}_{j,n}= z^{(2)}_{j,n}$. Let $\mathbf{z}^{(i),\pm}$ be defined as in \eqref{def:z+} and \eqref{def:z-}. Then it holds that
\begin{align}\label{equ:W+}
{W^{(1)}(\mathbf{z}^{(1),+})}/{W^{(1)}(\mathbf{z}^{(1)})}\geq {W^{(2)}(\mathbf{z}^{(2),+})}/{W^{(2)}(\mathbf{z}^{(2)})}
\end{align}
and that
\begin{align}\label{equ:W-}
{W^{(1)}(\mathbf{z}^{(1),-})}/{W^{(1)}(\mathbf{z}^{(1)})}\leq {W^{(2)}(\mathbf{z}^{(2),-})}/{W^{(2)}(\mathbf{z}^{(2)})}.
\end{align}
\end{claim}
Now it suffices to prove Claim~\ref{clm:q-convex}. We provide the argument for \eqref{equ:W+} and the one for \eqref{equ:W-} is similar. Let $\mathbf{z}^{(1)},\mathbf{z}^{(2)}\in\Omega_{M,\ell}$ be given as in Claim~\ref{clm:q-convex}. In particular $z^{(1)}_{j,n}=z^{(2)}_{j,n}=: z_{j,n}$. If $z_{j,n}+M^{-1}> M$, $\mathbf{z}^{(i),+}=\mathbf{z}^{(i)}$ and both sides of \eqref{equ:W+} equals$1$. If $z_{j,n}+M^{-1} \geq z^{(2)}_{j+1,n}$, then $W^{(2)}(\mathbf{z}^{(2),+})=0$ and the assertion holds. Hence we may assume $z_{j,n}+M^{-1}\leq M$ and $z_{j,n}+M^{-1} <z^{(2)}_{j+1,n} $. By a direct computation,
\begin{align*}
 {W^{(i)}(\mathbf{z}^{(i),+})}/{W^{(i)}(\mathbf{z}^{(i)})}=\frac{p(z^{(i)}_{j,n-1},z_{j,n}+M^{-1})}{p(z^{(i)}_{j,n-1},z_j)}\times \frac{p(z_{j,n}+M^{-1},z^{(i)}_{j,n+1})}{p(z_{j,n},z^{(i)}_{j,n+1})}.
\end{align*}
Apply twice the mean value theorem,
\begin{align*}
&\log\left( \frac{p(z^{(1)}_{j,n-1},z_{j,n}+M^{-1})}{p(z^{(1)}_{j,n-1},z_j)} \bigg/\frac{p(z^{(2)}_{j,n-1},z_{j,n}+M^{-1})}{p(z^{(2)}_{j,n-1},z_j)} \right)=\frac{\partial^2 }{\partial x\partial y} \log p( x,y) (z^{(1)}_{j,n-1}-z^{(2)}_{j,n-1})M^{-1}
\end{align*}
for some $(x,y)\in [z^{(2)}_{j,n-1},z^{(1)}_{j,n-1}]\times [z_{j,n},z_{j,n}+M^{-1}]$.
By Corollary~\ref{cor:q-convex}, $\frac{\partial^2 }{\partial x\partial y} \log p( x,y)\geq 0$. Therefore, the above is non-negative. Similarly, 
\begin{align*}
\log\left( \frac{p(z_{j,n}+M^{-1},z^{(1)}_{j,n+1})}{p(z_{j,n},z^{(1)}_{j,n+1})} \bigg/\frac{p(z_{j,n}+M^{-1},z^{(2)}_{j,n+1})}{p(z_{j,n},z^{(2)}_{j,n+1})}\right)\geq 0.
\end{align*}
The desired result \eqref{equ:W+} then follows.\\
 
{\bf Step four}, a coupling of $\mathbb{Q}^{(1)}$ and $\mathbb{Q}^{(2)}$ can be obtained by letting $M$ and $\ell$ go to infinity. Recall that we identify $\Omega_{M,\ell}$ as a subset of $C([1,k]_{\Z} \times [0,1],\mathbb{R})$ through linear interpolation \eqref{OmegaMktoC} hence $\mathbb{Q}^{(1)}$ and $\mathbb{Q}^{(2)}$ are probability measures on $C([1,k]_{\Z} \times [0,1],\mathbb{R})$.  From Proposition~\ref{lem:discrete} in Appendix~\ref{sec:discrete}, $\mathbb{Q}^{(i)}_{M,\ell}$ converges weakly to $\mathbb{Q}^{(i)}$. Since $\mathbb{Q}^{(1)}_{M,\ell}$ dominates $\mathbb{Q}^{(2)}_{M,\ell}$, we conclude that $\mathbb{Q}^{(1)}$ dominates $\mathbb{Q}^{(2)}$. The proof is completed.
\end{proof}

\section{Uniform upper and Lower bounds}\label{sec:uniformbounds}
The main goal of this section is to prove Propositions \ref{pro:sup} and \ref{pro:inf}, which establish uniform upper and lower bounds of the curves in $\mathcal{L}^{N,\alpha}$ over bounded intervals. The definition of the line ensemble $\mathcal{L}^{N,\alpha}(t)$ is in \eqref{eqn:edgeLE1}. 

\begin{proposition}\label{pro:sup}
Fix $\alpha\geq 0$. For any $a<b$, $\varepsilon>0$ and $k\in\mathbb{N}$, there exist  $R_1=R_1(\alpha,\varepsilon,k,b-a)$ and $N_1=N_1(\alpha,\varepsilon,k,\max\{|a|,|b|\})$ such that the following statement holds. For any $N\geq N_1$, it holds that
\begin{align}\label{equ:U}
\mathbb{P}\bigg( \sup_{t\in [a,b]}\mathcal{L}^{N,\alpha}_k(t)>R_1 \bigg)<\varepsilon.
\end{align}
\end{proposition}
\begin{proposition}\label{pro:inf}
Fix $\alpha\geq 0$. For any $a<b$, $\varepsilon>0$, there exist  $r_1=r_1(\alpha,\varepsilon,b-a)$ and $N_2=N_2(\alpha,\varepsilon,\max\{|a|,|b|\})$ such that the following statement holds. For any $N\geq N_2$, it holds that
\begin{align}\label{equ:LL}
\mathbb{P}\left( \inf_{t\in [a,b]}\mathcal{L}^{N,\alpha}_2(t)<r_1 \right)<\varepsilon.
\end{align} 
\end{proposition}
The starting point of showing Propositions~\ref{pro:sup} and \ref{pro:inf} is the finite dimensional convergence of $\mathcal{L}^{N,\alpha}$ to the extended $\alpha$-Bessel point process $\{\xi^\alpha_i, i\in\mathbb{N}\}$ \cite[Section 11.7.3]{For10}. See Appendix~\ref{sec:BEkernel} for the convergence of the correlation kernel. 
\begin{theorem}\label{thm:1}
Fix $\alpha\geq 0$, $n,m\in\mathbb{N}$ and $t_1<t_2<\dots <t_m $. Then $\{\mathcal{L}^{N,\alpha}_i(t_j)\},\ i\in [1,n]_{\mathbb{Z}}, j\in[1,m]_{\mathbb{Z}} $ converges in probability as $N$ goes to infinity. Moreover, for any $t\in \mathbb{R}$, $\{\mathcal{L}^{N,\alpha}_i(t)\},\ i\in [1,n]_{\mathbb{Z}} $ converges weakly to $\{\xi^\alpha_i\},\ i\in [1,n]_{\mathbb{Z}}$.
\end{theorem}

\begin{lemma}\label{thm:onept}
Fix $\alpha\geq 0$. For any $t\in\mathbb{R}$, $i\in\mathbb{N}$, $\mathcal{L}^{N,\alpha}_i(t)$ converges to $\xi^\alpha_i$ in distribution. Moreover, the convergence is locally uniformly in $t$ in the following sense. For any $L>0$, $r\geq 0$ and $i\in\mathbb{N}$, it holds that
\begin{equation}\label{equ:uniformconvergence}
\begin{split}
\limsup_{N\to\infty}\sup_{t\in [-L,L]} \bP(\mathcal{L}^{N,\alpha}_i(t)\leq r)\leq  \bP(\xi^{\alpha}_i \leq r),\\
\limsup_{N\to\infty}\sup_{t\in [-L,L]} \bP(\mathcal{L}^{N,\alpha}_i(t)\geq r)\leq  \bP(\xi^{\alpha}_i \geq r).
\end{split}
\end{equation}
\end{lemma} 
 
\begin{proof}
For fixed $t$, the convergence follows from Theorem \ref{thm:1}. From \eqref{equ:06261716} and \eqref{eqn:edgeLE1}, $\mathcal{L}^{N,\alpha}_i(t)$ has the same distribution as $\left(1+ {t}/{(4N)}\right)\mathcal{L}^{N,\alpha}_i(0).$ Locally uniform convergence \eqref{equ:uniformconvergence} then follows.
\end{proof}

Propositions~\ref{pro:sup} and \ref{pro:inf} are proved in Sections~\ref{sec:sup} and \ref{sec:inf} respectively. For brevity, we will denote $\mathcal{L}^{N,\alpha}$ by $\mathcal{L}^N$ in the rest of the section.

\subsection{Uniform upper bound}\label{sec:sup} We begin with the following one point upper bound for $\mathcal{L}^N$.
\begin{lemma}\label{lem:oneptU}
Fix $\alpha\geq 0$. For any $L>0,\varepsilon>0$ and $k\in\mathbb{N}$, there exist  $R_2=R_2(\alpha, \varepsilon,k)$ and $N_3=N_3(\alpha,\varepsilon,k,L)$ such that the following statement holds. For any $t\in [-L,L]$ and $N\geq N_3$, it holds that
\begin{align*}
\bP(\mathcal{L}^N_k(t)>R_2)<\varepsilon.
\end{align*}
\end{lemma}

\begin{proof}
Because $\mathbb{P}(\xi^{\alpha}_k<\infty)=1$, there exists $R_2=R_2(\alpha,\varepsilon,k)$ such that $$\bP(\xi^\alpha_k>R_2)<2^{-1}\varepsilon.$$
Apply Lemma \ref{thm:onept} with $r=R_2$, the existence of $N_3$ follows from \eqref{equ:uniformconvergence}.
\end{proof}

We are now ready to prove Proposition \ref{pro:sup}. The argument propagates the one-point upper bound in Lemma \ref{lem:oneptU} to uniform upper bound estimates through the Bessel Gibbs property of $\mathcal{L}^N$.
\begin{proof}[Proof of Proposition \ref{pro:sup}]
Let $L =\max\{|a|,|b|\}$.  
We argue by induction on $k$ and start with $k=1$. Take $R_2=R_2(\alpha,4^{-1}\varepsilon,1)$ as in Lemma \ref{lem:oneptU}. For $N\in\mathbb{N}$, define the event
\begin{align*}
\mathsf{A}_N\coloneqq \{ \mathcal{L}^N_1(a)\leq R_2, \mathcal{L}^N_1(b)\leq R_2 \}.
\end{align*}
From Lemma~\ref{lem:oneptU}, we have for $N\geq N_3(\alpha,4^{-1}\varepsilon,1,L)$, 
\begin{equation}\label{equ:07051557}
\mathbb{P}(\mathsf{A}_N^{\textup{c}})<2^{-1}\varepsilon.
\end{equation}
We proceed to bound $\bP \left(\left\{ \sup_{t\in [a,b]}\mathcal{L}_1^N(t)>R_1\right\}\cap\mathsf{A}_N \right)$ for suitable $R_1$. Let $R_1$ be large enough such that 
\begin{align*}
\mathbb{P}^{1,1,(a,b),R_2,R_2}_{\free} \left(\sup_{t\in [a,b]} \mathcal{J}(t)<R_1\right)<2^{-1}\varepsilon,
\end{align*}
where $\mathcal{J}(t)$ is distributed according to $\mathbb{P}^{1,1,(a,b),R_2,R_2}_{\free} $. By the tower property of conditional expectation, we have
\begin{align*}
\bP \left(\left\{ \sup_{t\in [a,b]}\mathcal{L}_1^N(t)>R_1\right\}\cap\mathsf{A}_N \right)=&\mathbb{E}\left[\mathbbm{1}_{\mathsf{A}_N}\cdot\mathbb{E}\left[ \mathbbm{1}\left\{\sup_{t\in [a,b]} \mathcal{L}^N_1(t)>R_1\right\} \,\bigg|\,\Fext (\{1\}\times (a,b))\right]\right]
\end{align*}
By the Gibbs property Proposition~\ref{pro:GibbsforL}, we have
\begin{align*}
\mathbb{E}\left[ \mathbbm{1}\left\{\sup_{t\in [a,b]} \mathcal{L}^N_1(t)>R_1\right\} \,\bigg|\,\Fext(\{1\}\times (a,b))\right]=  \mathbb{P}^{1,1,(a,b),\mathcal{L}^N_1(a),\mathcal{L}^N_1(b),0,\mathcal{L}^N_2 } \left(\sup_{t\in [a,b]} \mathcal{J}(t)>R_1 \right).
\end{align*}
When $\mathsf{A}_N$ occurs, $\mathcal{L}^N_1(a)\leq R_2$ and $\mathcal{L}^N_1(b)\leq R_2$. By the stochastic monotonicity Proposition~\ref{lem:monotonicity},  
\begin{align*}
\mathbbm{1}_{\mathsf{A}_N}\cdot\mathbb{E}\left[ \mathbbm{1}\left\{\sup_{t\in [a,b]} \mathcal{L}^N_1(t)>R_1\right\} \,\bigg|\,\Fext(\{1\}\times (a,b))\right]\leq & \mathbbm{1}_{\mathsf{A}_N}\cdot \mathbb{P}^{1,1,(a,b),R_2,R_2}_{\free} \left(\sup_{t\in [a,b]} \mathcal{J}(t)>R_1 \right)\\
<&\mathbbm{1}_{\mathsf{A}_N}\cdot 2^{-1}\varepsilon.
\end{align*}
Therefore,
\begin{equation}\label{equ:07051600}
\bP \left(\left\{ \sup_{t\in [a,b]}\mathcal{L}_1^N(t)>R_1\right\}\cap\mathsf{A}_N \right)<2^{-1}\varepsilon.
\end{equation} 
Combining \eqref{equ:07051557} and \eqref{equ:07051600}, we conclude \begin{align*}
\bP\left(\sup_{t\in [a,b]}\mathcal{L}_1^N(t)>R_1\right)<\varepsilon.
\end{align*}
We have established \eqref{equ:U} for $k=1$.\\[-0.3cm]

Next, we assume that $k\geq 2$ and that \eqref{equ:U} holds for $k-1$. Define
\begin{align*}
\bar{R}\coloneqq \max\{2^{-1}R_2(\alpha,4^{-1}\varepsilon,k),R_1(\alpha,4^{-1}\varepsilon,k-1,b-a)\}.
\end{align*} 
Here $R_2(\alpha, 4^{-1}\varepsilon,k)$ is the constant in Lemma \ref{lem:oneptU}. For $N\in\mathbb{N}$, consider the event
\begin{align*}
\mathsf{E}_N\coloneqq \left\{\mathcal{L}^N_k(a)\leq 2\bar{R},\mathcal{L}^N_k(b)\leq 2\bar{R}, \, \textup{and}\, \sup_{t\in[a,b]}\mathcal{L}^N_{k-1}(t)\leq \bar{R}\right\}.
\end{align*}
By Lemma \ref{lem:oneptU} and the induction hypothesis, there exists $N_1=N_1(\alpha,\varepsilon,k,L)$ such that for all $N\geq N_1$, 
\begin{equation}\label{equ:07051605}
\bP\left(\mathsf{E}_N^{\textup{c}}\right)<\frac{3}{4} \varepsilon.
\end{equation}

Take $R_1=R_1(\alpha,\varepsilon,k,b-a)$ such that
\begin{align*}
 \mathbb{P}^{1,1,(a,b),2\bar{R} ,2\bar{R},\bar{R},\infty  }\left(\sup_{t\in [a,b]} \mathcal{J}(t)>R_1 \right)< 4^{-1}\varepsilon.
\end{align*}
Here $\mathcal{J}$ is distributed according to $\mathbb{P}^{1,1,(a,b),2\bar{R} ,2\bar{R},\bar{R},\infty  }$. By the tower property of conditional expectation, we have
\begin{align*}
\bP \left(\left\{ \sup_{t\in [a,b]}\mathcal{L}_k^N(t)>R_1\right\}\cap\mathsf{E}_N \right)=&\mathbb{E}\left[\mathbbm{1}_{\mathsf{E}_N}\cdot\mathbb{E}\left[ \mathbbm{1}\left\{\sup_{t\in [a,b]} \mathcal{L}^N_k(t)>R_1\right\} \,\bigg|\,\Fext(\{k\}\times (a,b))\right]\right].
\end{align*}
By the Gibbs property Proposition~\ref{pro:GibbsforL}, we have
\begin{align*}
 \mathbb{E}\left[ \mathbbm{1}\left\{\sup_{t\in [a,b]} \mathcal{L}^N_k(t)>R_1\right\} \,\bigg|\,\Fext(\{k\}\times (a,b))\right]= \mathbb{P}^{1,1,(a,b),\mathcal{L}^N_k(a),\mathcal{L}^N_k(b),\mathcal{L}^N_{k-1},\mathcal{L}^N_{k+1} } \left(\sup_{t\in [a,b]} \mathcal{J}(t)>R_1 \right).
\end{align*}
When $\mathsf{E}_N$ occurs, $\mathcal{L}^N_k(a), \mathcal{L}^N_k(b)\leq 2\bar{R}$ and $\mathcal{L}^N_{k-1}\leq \bar{R} $. By the stochastic monotonicity Proposition~\ref{lem:monotonicity}, 
\begin{align*}
\mathbbm{1}_{\mathsf{E}_N}\cdot\mathbb{E}\left[ \mathbbm{1}\left\{\sup_{t\in [a,b]} \mathcal{L}^N_k(t)>R_1\right\} \,\bigg|\,\Fext(\{k\}\times (a,b))\right]\leq & \mathbbm{1}_{\mathsf{E}_N}\cdot \mathbb{P}^{1,1,(a,b),2\bar{R} ,2\bar{R}, \bar{R},\infty}\left(\sup_{t\in [a,b]} \mathcal{J}(t)>R_1 \right)\\
 <&\mathbbm{1}_{\mathsf{E}_N}\cdot  4^{-1}\varepsilon.
\end{align*}
Combining the above, we have 
\begin{equation}\label{equ:07051606}
\bP \left(\left\{ \sup_{t\in [a,b]}\mathcal{L}_k^N(t)>R_1\right\}\cap\mathsf{E}_N \right)<4^{-1}\varepsilon. 
\end{equation}
Combining \eqref{equ:07051605} and \eqref{equ:07051606}, we conclude
\begin{align*}
\bP\left(\sup_{t\in [a,b]}\mathcal{L}_k^N(t)>R_1\right)\leq \bP \left(\left\{ \sup_{t\in [a,b]}\mathcal{L}_k^N(t)>R_1\right\}\cap\mathsf{E}_N \right)+\bP\left(\mathsf{E}_N^{\textup{c}}\right)<\varepsilon.
\end{align*}
We have established \eqref{equ:U}. The proof is finished.
\end{proof}
\subsection{Uniform lower bound}\label{sec:inf}
In this section, we prove Proposition~\ref{pro:inf}. The main technical input is Proposition \ref{pro:epsilond}, which shows a uniform lower bound over a small interval.
\begin{proposition}\label{pro:epsilond}
Fix $\alpha\geq 0$. For any $\varepsilon\in (0,1]$ and $L>0$, there exist $N_4=N_4(\alpha,\varepsilon,L)$, $C=C(\alpha)\geq 1$ and $\nu=\nu(\alpha)>0$ such that the following statement holds. For all $N\geq N_4$, $|t_0|\leq L$ and  $d=C^{-1}\varepsilon^{1/(1+\nu)}$, it holds that
\begin{align}\label{equ:L}
\mathbb{P}\left( \inf_{t\in [t_0-2d,t_0-d]}\mathcal{L}^N_2(t)\leq d^2 \right)<\varepsilon.
\end{align}
\end{proposition}

By a union bound argument, we prove the uniform lower bound Proposition~\ref{pro:inf}.
\begin{proof}[Proof of Proposition \ref{pro:inf}]
Let $[a,b]\subset \mathbb{R}$. Take $d=C^{-1}\varepsilon^{1/(1+\nu)}$ as in Proposition \ref{pro:epsilond} and $t_0\in [a,b].$ Denote $L=\max\{|a|,|b|\}$. By Proposition \ref{pro:epsilond}, there exists $N_4=N_4(\alpha,\varepsilon,L)$ such that for all $N\geq N_4$,
\begin{align*}
\bP\left(\inf_{t\in [t_0-2d,t_0-d]}\mathcal{L}^N_2(t)\leq d^2 \right)<\varepsilon.
\end{align*}
Covering $[a,b]$ with $\lceil d^{-1}(b-a)\rceil$ intervals with length $d$ and applying the union bound estimate, we obtain that
\begin{align*}
\bP\left(\inf_{t\in [a,b]}\mathcal{L}^N_2(t)\leq d^2 \right)<\lceil d^{-1}(b-a)\rceil\varepsilon\leq \left(C (b-a)\varepsilon^{-1/(1+\nu)}+1\right)\varepsilon.
\end{align*} 
Because $\nu>0$, the right hand side could be made as small as possible. The assertion then follows by re-choosing $\varepsilon$.
\end{proof}

The rest of this section is devoted to proving Proposition~\ref{pro:epsilond}. We first set up a lemma which controls the lower tail of the second smallest particle in the $\alpha$-Bessel point process, i.e. $\xi^\alpha_2$. This can be found in \cite{TW}. We provide the proof for completeness.

\begin{lemma}\label{lem:ppL}
Fix $\alpha\geq 0$. There exists positive constant $c=c(\alpha)$ and $\nu=\nu(\alpha)$ such that 
\begin{align*}
\bP (\xi^\alpha_1\leq r)\leq  cr^{1+\alpha}+o_\alpha(r^{2}),
\end{align*} 
and
\begin{align*}
\bP (\xi^\alpha_2\leq r)\leq  cr^{1+\nu}+o_\alpha(r^{2}).
\end{align*}
\end{lemma}
 
\begin{proof}
First consider the case when $\alpha>0$. For $n\in\mathbb{N}_0$ and $s>0$, let $E(n,r)=\mathbb{P}(\xi^\alpha_n<r, \xi^\alpha_{n+1}\geq r).$ be the probability that the number of $\xi^\alpha_i$ in $(0,r)$ is $n$. We adopt the convention that $\xi_0=0$. From \cite[(1.22)]{TW}, one has the following asymptotic expression when $r$ goes to zero.
\begin{align*}
-r\frac{d}{dr}\log E(0,r)= c_\alpha r^{\alpha+1}+O(r^{\alpha+2}).
\end{align*}
Here $c_\alpha=\frac{1}{2^{2\alpha+2}}\frac{1}{\Gamma(\alpha+1)\Gamma(\alpha+2)} $. Then we deduce
\begin{align*}
E(0,r)=1- (\alpha+1)^{-1}{c_\alpha}r^{\alpha+1}+O(r^{\alpha+2}).
\end{align*}
Hence, 
\begin{align*}
\bP (\xi^\alpha_2\leq r)\leq  \bP (\xi^\alpha_1\leq r)=1-E(0,r)=(\alpha+1)^{-1}{c_\alpha}r^{\alpha+1}+O(r^{\alpha+2}).
\end{align*}
To conclude, for $\alpha>0$, the desired result holds with $c=\frac{c_\alpha}{\alpha+1}$ and $\nu=\alpha$.\\[-0.2cm]

Next, we consider the case $\alpha=0$. Evaluating $c_0$, we obtain $E(0,r)=1-4^{-1} r  +O(r^{2}).$ From \cite[(2.30)]{TW}, 
\begin{align*}
 \frac{E(1,r)}{E(0,r)} = 2^{-1} {r} \left(I^2_0(r^{1/2})-r^{-1/2}I_0(r^{1/2})I_1(r^{1/2})-I^2_1(r^{1/2})\right).
\end{align*} 
Here $I_0$ and $I_1$ are modified Bessel functions of the first kind. As $z$ goes to zero, we have \cite[(9.6.10)]{AS}
\begin{align*}
I_0(z)=1+O(z^2),\ I_1= 2^{-1}{z} +O(z^3).
\end{align*}
Hence 
\begin{align*}
  {E(1,r)} ={E(0,r)}\times  2^{-1} {r} \left(I^2_0(r^{1/2})-r^{-1/2}I_0(r^{1/2})I_1(r^{1/2})-I^2_1(r^{1/2})\right)= 4^{-1} {r} +O(r^2).
\end{align*}
As a result,
\begin{align*}
\bP (\xi^\alpha_2\leq r)=1-{E(0,r)}-{E(1,r)}=O(r^2).
\end{align*}
\end{proof}

From Lemma~\ref{lem:ppL}, we may derive the following one-point lower bound for $\mathcal{L}^N_1$ and $\mathcal{L}^N_2$.
\begin{lemma}\label{lem:oneptL}
Fix $\alpha\geq 0$ and let $c=c(\alpha)$, $\nu=\nu(\alpha)$ be the constants in Lemma~\ref{lem:ppL}. There exists $\varepsilon_0=\varepsilon_0(\alpha)$ such that for all $0<\varepsilon\leq \varepsilon_0$ and $L>0$, there exists $N_5=N_5(\alpha,\varepsilon,L)$ which makes the following statement true. Let $r_2=(\varepsilon/(2c))^{1/(1+\alpha)}$ and $r_3=  (\varepsilon/(2c))^{1/(1+\nu) } $. For all $N\geq N_5$ and $|t|\leq L$, it holds that
\begin{align*}
 \bP(\mathcal{L}^N_1(t)<r_2)<\varepsilon,  
\end{align*} 
and
\begin{align*}
\bP(\mathcal{L}^N_2(t)<r_3)<\varepsilon.
\end{align*}
\end{lemma}
\begin{proof}
We present only the proof for the latter bound. The argument for the first one is similar.  From Lemma \ref{lem:ppL}, $\bP(\xi_2<r_3)=2^{-1}\varepsilon+o_\alpha(\varepsilon).$ Take $\varepsilon_0$ small enough such that for all $0<\varepsilon\leq \varepsilon_0$, $\bP(\xi_2<r_3)<\frac{2}{3}\varepsilon.$ Apply Lemma \ref{thm:onept} with $r=r_3$, the existence of $N_5$ follows from \eqref{equ:uniformconvergence}.
\end{proof}

We finish this section with the proof of Proposition~\ref{pro:epsilond}. The proof relies on estimates of transition densities for non-intersecting Bessel bridges.
\begin{proof}[Proof of Proposition~\ref{pro:epsilond}]
Let $c$, $\nu$ and $\varepsilon_0$ be the constants in Lemma \ref{lem:oneptL}. Fix $\varepsilon\leq \varepsilon_0$ and set $d=(\varepsilon/(2c))^{1/(1+\nu) }$. We may assume that $d\leq 2^{-1}$. For $|t_0|\leq L$, denote 
\begin{align*}
\mathsf{F}_N\coloneqq \left\{ \inf_{t\in [t_0-2d,t_0-d]}\mathcal{L}^N_2(t)\leq d^2 \right\}.
\end{align*}
It is sufficient to show the following.  For any $0<\varepsilon\leq \varepsilon_0$, there exists $N_4=N_4(\alpha, \varepsilon,L)$ and $C_0$ such that for all $N\geq N_4$ and $|t_0|\leq L$,
\begin{align}\label{equ:Lmid}
\mathbb{P}(\mathsf{F}_N)<C_0\varepsilon.
\end{align}

Let $R=R_2(\alpha, \varepsilon,2)$ be the constant in Lemma \ref{lem:oneptU} and $T=R^{1/2}$. Consider the event
\begin{align*}
\mathsf{G}_N\coloneqq \left\{ \mathcal{L}^N_2(t_0+T)\leq R  \right\}.
\end{align*}
From Lemma \ref{lem:oneptU}, there exists $N=N_3(\alpha, \varepsilon,L)$ such that for all $N\geq N_3$, $\bP(\mathsf{G}_N^{\textup{c}})<\varepsilon$. We aim to show that there exists $N_5(\alpha,\varepsilon,L)\geq N_3$ such that for all $N\geq N_5$,
\begin{align}\label{equ:FcapG}
\bP(\mathsf{F}_N\cap \mathsf{G}_N)<C_1\varepsilon.
\end{align}
This implies $\bP(\mathsf{F}_N)\leq \bP(\mathsf{F}_N\cap \mathsf{G}^N)+\bP(\mathsf{G}_N^{\textup{c}})<(C_1+1)\varepsilon$, which is sufficient for \eqref{equ:Lmid}.\\[-0.3cm] 

Now we turn to \eqref{equ:FcapG}. Consider the random variable
\begin{align*}
\mathfrak{l}\coloneqq \left\{ \begin{array}{cc}
\inf\left\{ t\in [t_0-2d,t_0-d]\,|\, \mathcal{L}^N_2(t)\leq d^2 \right\}, &\left\{ t\in [t_0-2d,t_0-d]\,|\, \mathcal{L}^N_2(t)\leq d^2 \right\}\neq \phi,\\
t_0-d, & \left\{ t\in [t_0-2d,t_0-d]\,|\, \mathcal{L}^N_2(t)\leq d^2 \right\}=\phi. 
\end{array} \right.
\end{align*}
This is the smallest time when $\mathcal{L}^N_2$ goes below $d^2$ if it does. By the tower property for conditional expectation, we have
\begin{align*}
\mathbb{P}\left(\left\{ \mathcal{L}^N_2(t_0)<d \right\}\cap \mathsf{F}_N\cap \mathsf{G}_N \right)
=\mathbb{E}\left[ \mathbbm{1}_{\mathsf{F}_N}\cdot \mathbbm{1}_{\mathsf{G}_N}\cdot \mathbb{E}\left[ \mathbbm{1}\left\{ \mathcal{L}^N_2(t_0)<d \right\}\,|\,\Fext(\{1,2\}\times (\mathfrak{l},t_0+T))  \right] \right]\\[-0.4cm]
\end{align*}

Denote $\vec{x}=(\mathcal{L}^N_1(\mathfrak{l}),\mathcal{L}^N_2(\mathfrak{l}))$, $\vec{y}=(\mathcal{L}^N_1(t_0+T),\mathcal{L}^N_2(t_0+T))$ and $g=\mathcal{L}^N_3|_{[\mathfrak{l},t_0+T]}.$ By the strong Gibbs property Lemma~\ref{lem:stronggibbs},
\begin{equation}\label{equ:GibbsFcapG}
\begin{split}
& \mathbb{E}\left[ \mathbbm{1}\left\{ \mathcal{L}^N_2(t_0)<d  \right\}\,|\,\Fext(\{1,2\}\times (\mathfrak{l},t_0+T))  \right]= \mathbb{P}^{1,2,(\mathfrak{l},t_0+T),\vec{x},\vec{y},0,g }( \mathcal{J}_2(t_0)<d ).
\end{split}
\end{equation}
Here $(\mathcal{J}_1,\mathcal{J}_2)$ is distributed according to $\mathbb{P}^{1,2,(\mathfrak{l},t_0+T),\vec{x},\vec{y},0,g }$. We claim there exists $C_1$ depending only on $\alpha$ such that
\begin{align}\label{equ:Be2_t_0}
\mathbbm{1}_{\mathsf{F}_N}\cdot \mathbbm{1}_{\mathsf{G}_N}\cdot\mathbb{P}^{1,2,(\mathfrak{l},t_0+T),\vec{x},\vec{y},0,g }( \mathcal{J}_2(t_0)<d )\geq 1/C_1\cdot \mathbbm{1}_{\mathsf{F}_N}\cdot \mathbbm{1}_{\mathsf{G}_N}.
\end{align}
Assuming for a moment this claim holds. Combining \eqref{equ:GibbsFcapG} and \eqref{equ:Be2_t_0}, we have
\begin{align*}
\mathbb{P}(\mathsf{F}_N\cap\mathsf{G}_N)&=\mathbb{E}[\mathbbm{1}_{\mathsf{F}_N}\cdot \mathbbm{1}_{\mathsf{G}_N}]\\
&\leq C_1 \mathbb{E}\left[\mathbbm{1}_{\mathsf{F}_N}\cdot \mathbbm{1}_{\mathsf{G}_N}\cdot \mathbb{E}\left[ \mathbbm{1}\left\{ \mathcal{L}^N_2(t_0)<d  \right\}\,|\,\Fext(\{1,2\}\times (\mathfrak{l},t_0+T))  \right]\right]\\
&= C_1  \mathbb{P}\left(\left\{ \mathcal{L}^N_2(t_0)<d \right\}\cap \mathsf{F}_N\cap \mathsf{G}_N \right)\\
&\leq C_1  	\mathbb{P}\left( \mathcal{L}^N_2(t_0)<d  \right).
\end{align*}
From Lemma \ref{lem:oneptL}, For $N\geq N_5(\alpha, \varepsilon,L)$, we have
$\mathbb{P}\left( \mathcal{L}^N_2(t_0)<d  \right)<\varepsilon$. This proves \eqref{equ:FcapG}.\\[-0.3cm]

It remains to prove \eqref{equ:Be2_t_0}. We use $C$ to denote a constant which depends only $\alpha$ and its value may vary from line to line.  Note that $\mathsf{F}_N=\left\{ \mathcal{L}^N_2(\mathfrak{l})\leq d^2 \right\}.$ By the stochastic monotonicity Lemma~\ref{lem:monotonicity}, the left hand side of \eqref{equ:Be2_t_0} is bounded from below by
\begin{align*}
\mathbbm{1}_{\mathsf{F}_N}\cdot \mathbbm{1}_{\mathsf{G}_N}\cdot\mathbb{P}^{1,2,(\mathfrak{l},t_0+T),\vec{x}_*,\vec{y}_*,0,+\infty }( \mathcal{J}_2(t_0)<d ).
\end{align*}
with $\vec{x}_*=(d^2,d^2)$ and $\vec{y}_*=(R,R)$. \\

We write $x_* =d^2$ and $y_*=R$. Let $\tau\coloneqq t_0-\mathfrak{l}$. Note that $d\leq  \tau\leq 2d$. For any $0\leq z_1\leq z_2\leq d$, we have
\begin{equation}\label{equ:xzy}
x_*  z_2\leq  d^2\leq \tau^2,\ y_*z_2  \leq  R=T^2\ \textup{and}\ x_* y_*\leq  R\leq (\tau+T)^2. 
\end{equation}
This verifies the assumption of Lemma~\ref{lem:2sample} with $L=1$. Let $( \mathcal{J}_1(t),\mathcal{J}_2(t))$ be distributed according to $\mathbb{P}^{1,2,(\mathfrak{l},t_0+T),\vec{x}_*,\vec{y}_*,0,+\infty }$. From Lemma~\ref{lem:2sample}, the p.d.f at $( \mathcal{J}_1(t_0),\mathcal{J}_2(t_0))=(z_1,z_2)$ has a lower bound
\begin{align*}
C^{-1} \left({\tau}^{-1}+ {T}^{-1} \right)^2 \frac{q_\tau(x_*,z_2)q_\tau(x_*,z_1)q_T(z_2,y_*)q_T(z_1,y_*)}{q_{\tau+T}(x_*,y_*)^2}(z_1-z_2)^2.
\end{align*}
Using \eqref{def:qF}, the above equals
\begin{equation}\label{equ:mid3}
\begin{split}
&2^{-2}C^{-1}(\tau^{-1}+T^{-1})^{2\alpha+4} z_1^\alpha z_2^\alpha \cdot (z_1-z_2)^2\\
\times &h_{\alpha}(\sqrt{x_*z_1}/\tau)h_{\alpha}(\sqrt{x_*z_2}/\tau)h_{\alpha}(\sqrt{z_1y_*}/T)h_{\alpha}(\sqrt{z_2y_*}/T)\times   h_{\alpha}(\sqrt{x_*y_*}/(\tau+T))^{-2} \\
\times &\exp\bigg( -x_*(\tau^{-1}-(\tau+T)^{-1})-y_*(T^{-1}-(\tau+T)^{-1})-2^{-1}(z_1+z_2)(\tau^{-1}+T^{-1}) \bigg).
\end{split}
\end{equation}
Here $h_\alpha$ is an entire function. See \eqref{equ:hexpansion} for the Taylor expansion of $h_\alpha$. From \eqref{equ:xzy}, the second line in \eqref{equ:mid3} has a lower bound. To bound the third line in \eqref{equ:mid3}, we use
\begin{align*}
 &x_*(\tau^{-1}-(\tau+T)^{-1})\leq \tau^{-1}x_*\leq d\leq 1,\  y_*(T^{-1}-(\tau+T)^{-1})\leq T^{-2}y_*\leq 1,\\
&2^{-1}(z_1+z_2)(\tau^{-1}+T^{-1})\leq (z_1+z_2)/\tau .
\end{align*}  
Combining the above, \eqref{equ:mid3} is bounded from below by
\begin{align*}
C^{-1} \tau^{-2\alpha-4} z_1^\alpha  z_2^\alpha e^{-(z_1+z_2)/ \tau  }(z_1 -z_2 )^2.
\end{align*}
It follows that $$\mathbb{P}^{1,2,(\mathfrak{l},t_0+T),\vec{x}_0,\vec{y}_0,0,+\infty }( \mathcal{J}_2(t_0)<d ) 
 \geq C^{-1} \tau^{-2\alpha-4} \int_0^{d} dz_2\int_0^{z_2} dz_1\,  z_1^\alpha  z_2^\alpha e^{-(z_1+z_2)/ \tau  }(z_1 -z_2 )^2.$$ 
Using the change of variables $(w_1,w_2)=(z_1/\tau,z_2\tau)$ and $d/\tau\geq 1$, the above is bounded from below by 
\begin{align*}
C^{-1} \int_0^{1/2} dw_2\int_0^{w_2} dw_1\,  w_1^\alpha w_2^\alpha e^{-(w_1+w_2)   }(w_1 -w_2)^2 =:C_1^{-1}.
\end{align*}
This establishes \eqref{equ:Be2_t_0}. The proof is finished.
\end{proof}

\begin{remark}
A similar and actually simpler argument based on the lower bound of $\mathcal{L}^N_1$ in Lemma~\ref{lem:oneptL} can show that $\mathbb{P}\left( \inf_{t\in [0,d]}\mathcal{L}^N_1\leq d^2\right)<\varepsilon$ for $d\sim \varepsilon^{1/(1+\alpha)}$. However, if $\alpha=0$, this estimate is not enough to extend to a unit interval as in Proposition~\ref{pro:inf}.  
\end{remark}

\section{Normalizing constant}\label{sec:proofofZ}

In this section, we give a lower bound for the normalizing constant. Let us recall its definition. Given $k\in\mathbb{N}$, $(a,b)\subset \mathbb{R}$, two vectors $\vec{x},\vec{y}\in\mathbb{R}_+^k$ and two semi-continuous functions $f,g:[a,b]\to\mathbb{R}$, the normalizing constants $Z^{1,k,(a,b),\vec{x},\vec{y},f,g}$ is defined in Definition~\ref{def:Bessel bridge LE}. It represents the probability that $k$ independent $\alpha$-Bessel bridges, starting at $\vec{x}$ at $t=a$ and terminating at $\vec{y}$ at $t=b$, avoid each other as well as $f(t)$ and $g(t)$. We are mainly concerned with the case in which $(\vec{x},\vec{y},f,g)$ are random variables. To be concrete, let $\mathcal{L}^{N,\alpha}$ be the scaled line ensembles defined in \eqref{eqn:edgeLE2}. For $N$ large enough such that $(a,b)\subset [-4N,\infty)$ and $N\geq k+1$, we consider the random variable 
$$Z^{1,k,(a,b),\vec{x},\vec{y},0,g},$$
where ${x}_i=\cL^{N,\alpha}_i(a), {y}_i=\cL^{N,\alpha}_i(b)$ and $g=\cL^{N,\alpha}_{k+1}\big|_{[-L,L]}$. The randomness of $Z^{1,k,(-L,L),\vec{x},\vec{y},0,g}$ inherits from the boundary data, $\vec{x},\vec{y}$ and $g$. This random normalizing constant appears naturally from the squared $\alpha$-Bessel Gibbs property of $\mathcal{L}^{N,\alpha}$. As long as $Z^{1,k,(a,b),\vec{x},\vec{y},0,g}$ has a lower bound, $\mathcal{L}^{N,\alpha}$ on $ [1,k]_{\mathbb{Z}}\times [a,b]$ behaves similarly to independent squared $\alpha$-Bessel bridges. See Definition~\ref{def:Bessel bridge LE}. The main goal of this section is to show that, under the law of $\mathcal{L}^{N,\alpha}$, the random normalizing constants are not small with a high probability. This is the content of Proposition~\ref{pro:Z} and is a key ingredient for proving the tightness of $\mathcal{L}^{N,\alpha}$ in the next section. 
 
\begin{proposition}\label{pro:Z}
Fix $\alpha\geq 0$, $k\in\mathbb{N}$ and $L\geq 1$. For any $\varepsilon>0$, there exists $\delta_1=\delta_1(\alpha,k,L,\varepsilon)>0$ and $N_6=N_6(\alpha,k,L,\varepsilon)>0$ such that for all $N\geq N_6$, it holds that
\begin{align*}
\mathbb{P}(Z^{1,k,(-L,L),\vec{x},\vec{y},0,g}<\delta_1) <\varepsilon.
\end{align*}
Here $x_i=\mathcal{L}^{N,\alpha}_i(-L)$, $y_i=\mathcal{L}^{N,\alpha}_i(L)$ and $g=\mathcal{L}^{N,\alpha}_{k+1}\big|_{[-L,L]}$.
\end{proposition}

In Section~\ref{sec:proofZ} we set up a resampling framework, state Proposition~\ref{pro:spacing-k} and discuss some of its consequences. These are used in Section~\ref{sec:proofZZ} to prove Proposition \ref{pro:Z}. Section~\ref{sec:spacing-k} adapts the two-step resampling method in \cite{Wu21a} to prove Proposition~\ref{pro:spacing-k}. For brevity, we will denote $\mathcal{L}^{N,\alpha}$ by $\mathcal{L}^N$ in the rest of the section. 


\subsection{A resampling framework}\label{sec:proofZ} We start by setting the resampling domain. We are interested in the behavior of $\mathcal{L}^N$ restricted on $[-L,L]$, but we are in need of buffer regions for the curves in $\mathcal{L}^N$ to configure themselves.  Fix $k\in\mathbb{N}, L\geq 1$.  For $j\in [1,k]_\mathbb{Z}$, let 
$$\ell_j\coloneqq -(k+2-j)L,\ r_j\coloneqq (k+2-j)L,$$
and
\begin{equation}\label{def:Pcal}
\begin{split}
\mathcal{P}\coloneqq &\cup_{j=1}^k \{j\}\times [\ell_j,r_j],\\ 
b\mathcal{P}\coloneqq &\big(\cup_{j=1}^k \{j\}\times ([\ell_1, \ell_j] \cup  [r_j,r_1])\big)\cup \{k+1\}\times [\ell_1,r_1] .
\end{split}
\end{equation}
We will run resampling arguments on $\mathcal{P}$ and $b\mathcal{P}$ will be viewed as the boundary of $\mathcal{P}$. In other words, we will sample $\mathcal{L}^N_j$ on $[\ell_i, r_j]$. \\[-0.2cm]

Now we introduce the favorable boundary conditions, which are typical under the law of $\mathcal{L}^N$. Let $f=( f_1, \cdots, f_{k+1})\in C(b\mathcal{P},\mathbb{R})$ be functions. For $R\geq 1$, we refer to the following conditions as the $R$-{\bf Good} conditions. Denote the collection of continuous functions $f\in C(b\mathcal{P},\mathbb{R})$ as $\mathcal{G}(k,L,R)$.
\begin{enumerate}
\item For all $1\leq j\leq k$, $f_j(t)\leq R$ for $t\in [\ell_{1}, \ell_j] \cup [r_j, r_1]$.\\[-0.4cm]
\item For all $2\leq j \leq k$, $f_j(t)\geq R^{-1}$ for $t\in [\ell_1, \ell_j] \cup [r_j, r_1]$. $f_{k+1}(t)\geq R^{-1}$ for $t\in [\ell_1,r_1]$.\\[-0.4cm]
\item For all $1\leq j\leq k$, $f_j(t)< f_{j+1}(x)$ for $t\in[\ell_1,\ell_j] \cup [r_j,r_1]$.
\end{enumerate}

Fix an element $f\in \mathcal{G}(k,L,R)$. We denote by $\mathfrak{X}_f\subset C(\mathcal{P},\mathbb{R})$ the collection of continuous functions $\mathcal{J}=(\mathcal{J}_1,\dots,\mathcal{J}_k)$ with $\mathcal{J}_{i}(\ell_i)=f_i(\ell_i)$ and $\mathcal{J}_{i}(r_i)=f_i(r_i)$ for all $1\leq i\leq k$. For any continuous function $\mathcal{J}\in \mathfrak{X}_f$, we may combine $\mathcal{J}$ and $f$ to form a continuous function in $[1,k]_{\mathbb{Z}}\times [\ell_1,r_1]$ as
\begin{align}\label{def:J_ext}
\mathcal{J}_{f,i}(t)\coloneqq \left\{ \begin{array}{cc}
 \mathcal{J}_i(t), & t\in [\ell_i,r_i],\\
 f_i(t), & t\in [\ell_1,\ell_i)\cup (r_i,r_1].
\end{array} \right.
\end{align} 

We start to set up the resampling argument. We consider line ensembles which take values in $\mathfrak{X}_f$. Let $\{\mathcal{Q}_j, 1\leq j\leq k\}$ be independent squared Bessel bridges defined on $[\ell_j,r_j]$ with $\mathcal{Q}_j(\ell_j)=f_j(\ell_j)$ and $\mathcal{Q}_j(r_j)=f_j(r_j)$. The law of $\mathcal{Q}_j$ is given by $\mathbb{P}_{\free}^{j,j,(\ell_j,r_j),f_j(\ell_j),f_j(r_j)}$. We denote by $\mathbb{P}_{\free}$ the joint law of $ (\mathcal{Q}_1,\mathcal{Q}_2,\dots ,\mathcal{Q}_k)$ and view it as a measure on $\mathfrak{X}_f$. Consider the events
\begin{align*}
\mathsf{NoInt}:=&\{\mathcal{J}_{f,1}(t)<\mathcal{J}_{f,2}(t)<\dots <\mathcal{J}_{f,k}(t)<f_{k+1}(t)\ \textup{for}\ t\in [\ell_1,r_1]\},\\
\widetilde{\mathsf{NoInt}}:=&\{\mathcal{J}_{f,1}(t)<\mathcal{J}_{f,2}(t)<\dots <\mathcal{J}_{f,k}(t)<f_{k+1}(t)\ \textup{for}\ t\in [\ell_1,-L]\cup[L ,r_1]\}.
\end{align*}
Denote by $\mathbb{P}$ and $ \tilde{\mathbb{P}}$ the laws of $\mathbb{P}_{\free}$ conditioned on $ \mathsf{NoInt} $ and $\widetilde{\mathsf{NoInt}}$ respectively. Equivalently, the Radon-Nikodym derivatives of $\mathbb{P}$ and $ \tilde{\mathbb{P}}$ with respect to $\mathbb{P}_\free$ are given by
\begin{align}
\frac{\textup{d}\mathbb{P}}{\textup{d}\mathbb{P}_\free}(\mathcal{J})&=\frac{\mathbbm{1}_{\mathsf{NoInt}}(\mathcal{J} ) }{\mathbb{P}_\free(\mathsf{NoInt})}, \label{def:J}\\ 
\frac{\textup{d} \tilde{\mathbb{P}} }{\textup{d}\mathbb{P}_\free}(\mathcal{J})&=\frac{\mathbbm{1}_{\widetilde{\mathsf{NoInt}}}(\mathcal{J}) }{\mathbb{P}_\free(\widetilde{\mathsf{NoInt}})}.\label{def:J'}
\end{align}
Note that
\begin{align}\label{equ:J/J'}
\frac{\textup{d}\mathbb{P} }{\textup{d}\tilde{\mathbb{P}} }(\mathcal{J} )\propto \mathbbm{1}\{\mathcal{J}_1(t)<\mathcal{J}_2(t)<\dots <\mathcal{J}_k(t)<f_{k+1}(t)\ \textup{for}\ t\in[-L,L]\}. 
\end{align}
This relation will be used in the proof of Proposition~\ref{pro:Z_small-k}.

\begin{remark}\label{rmk:P_J'}
For $k=1$, the measure $\mathbb{P}$ and $\tilde{\mathbb{P}}$ are the same as $\mathbb{P}^{1,1,(\ell_1,r_1),f_1(\ell_1),f_1(r_1),0, f_2}$ and $\mathbb{P}_{ (-L,L)}^{1,1,(\ell_1,r_1),f_1(\ell_1),f_1(r_1),0, f_2}$ defined in Definition~\ref{def:Bessel bridge LE} respectively. 
\end{remark}

Consider the event that curves are separated at the endpoints of the interval $[-L,L]$,
\begin{align}\label{eq:E}
\mathsf{E}:=\cap_{j=1}^k \left\{  {\mathcal{J}}_j(\pm L)  \in \left[ (2j-2)( 2k)^{-1}  R^{-1} ,(2j-1)( 2k)^{-1} R^{-1}\right] \right\}.
\end{align}

Proposition \ref{pro:spacing-k} below provides a lower bound of $\tilde{\mathbb{P}}(\mathsf{E})$ and will be proved in Section~\ref{sec:spacing-k}.

\begin{proposition}\label{pro:spacing-k}
Fix $\alpha\geq 0$, $k\in\mathbb{N}$, $L\geq 1$ and $R\geq 1$. There exists a constant $\delta_2=\delta_2(\alpha,k,L,R)>0$ such that the following holds. For any $(f_1,f_2,\dots f_{k+1})\in \mathcal{G}(k,L,R)$, let $\tilde{\mathbb{P}}$ be the probability measure given in \eqref{def:J'}. Then we have  
\begin{align*}
\tilde{\mathbb{P}}(\mathsf{E})\geq \delta_2.
\end{align*}
\end{proposition}

\begin{corollary}\label{cor:Z_lowerbound-k}
Fix $\alpha\geq 0$, $k\in\mathbb{N}$, $L\geq 1$ and $R\geq 1$. There exists a constant  $\delta_3=\delta_3(\alpha, k,L,R)>0$ such that the following holds. For any $(f_1,f_2,\dots f_{k+1})\in \mathcal{G}(k,L,R)$, let $\tilde{\mathbb{P}}$ be the probability measure given in \eqref{def:J'}. Then we have
\begin{align*}
\tilde{\mathbb{P}} \left(Z^{1,k,(-L,L),\vec{x} ,\vec{y} ,0,f_{k+1}} \geq \delta_3  \right)\geq \delta_2  .
\end{align*} 
Here $\vec{x}=( {\mathcal{J}}_j(-L))_{j=1}^k$, $\vec{y}=( {\mathcal{J}}_j(L))_{j=1}^k$ and $\delta_2=\delta_2(\alpha,k,L,R)$ is the small constant in Proposition~\ref{pro:spacing-k}.
\end{corollary}

\begin{proof}
Recall that $Z^{1,k,(-L,L),\vec{x},\vec{y},0,f_{k+1}}$ (defined in Definition~\ref{def:Bessel bridge LE}) is the following non-intersecting probability
\begin{align*}
Z^{1,k,(-L,L),\vec{x},\vec{y},0,f_{k+1}}:=\mathbb{E}_{\free}^{1,k,(-L,L),\vec{x} ,\vec{y}}\left[\mathbbm{1}\{\mathcal{J}_1<\mathcal{J}_2<\dots <\mathcal{J}_k<f_{k+1}\ \textup{on}\ [-L,L]\}   \right],
\end{align*}
where $(\mathcal{J}_1,\dots, \mathcal{J}_k)$ are distributed according to $\mathbb{P}_{\free}^{1,k,(-L,L),\vec{x} ,\vec{y}}$.

When $\mathsf{E}$ (defined in \eqref{eq:E}) happens, the boundary values $\vec{x}$ and $\vec{y}$ are separated at least by $(2k)^{-1}R^{-1}$. This leads to the observation that if the curves $ {\mathcal{J}}_j(t)$ stay close to their linear interpolation functions on $[-L,L]$, they will stay ordered over the full interval $[-L,L]$. Define the following event $\mathsf{Osc}$, where $ {\mathcal{J}}_j, j=1,2,\cdots,k$ do not deviate from their corresponding linear interpolation functions by $(4k)^{-1}R^{-1}$,
\begin{align*}
\mathsf{Osc}\coloneqq \left\{ \sup_{1\leq j\leq k}\sup_{t\in [-L,L]}\left|  { \mathcal{J}}_j(t)-2^{-1}L^{-1}\left((t+L) {\mathcal{J}}_j(L)+(L-t) {\mathcal{J}}_j(-L) \right) \right|< (4k)^{-1}R^{-1}  \right\}.
\end{align*}

Suppose $\mathsf{E}$ and $\mathsf{Osc}$ both occur, then $( {\mathcal{J} }_1, {\mathcal{J} }_2,\dots , {\mathcal{J} }_k,f_{k+1})$ remain ordered on $[-L,L]$. In view of Lemma~\ref{lem:compactcoupling}, there exists $\delta_3=\delta_3(\alpha, k,L,R)>0$ such that for all $ \vec{z},\vec{w}\in\mathbb{R}^k_+$ with $|z_i|\leq 1$ and $|w_i|\leq 1$, it holds that $\mathbb{P}_{\free}^{1,k,(-L,L),\vec{z} ,\vec{w}}(\mathsf{Osc}) >\delta_3.$ This implies that
\begin{align*}
Z^{1,k,(-L,L),\vec{x},\vec{y},0,f_{k+1}}\cdot\mathbbm{1}_{\mathsf{E}}\geq \delta_3\cdot\mathbbm{1}_{\mathsf{E}}. 
\end{align*}
In particular, $\mathsf{E}\subset \left\{ Z^{1,k,(-L,L),\vec{x},\vec{y},0,f_{k+1}}\geq  \delta_3 \right\}$. The desired result then follows from Proposition \ref{pro:spacing-k}.
\end{proof}

\begin{proposition}\label{pro:Z_small-k}
Fix $\alpha\geq 0$, $k\in\mathbb{N}$, $L\geq 1$ and $R\geq 1$. There exists a constant $\delta_4=\delta_4(\alpha, k,L,R)>0$ such that the following holds. For any $(f_1,f_2,\dots f_{k+1})\in \mathcal{G}(k,L,R)$, let $\mathbb{P}$ be the probability measure on $\mathfrak{X}_f$ given in \eqref{def:J}. Then for all $\varepsilon\in (0,1]$, we have
$$\mathbb{P}\left( Z^{1,k,(-L,L),\vec{x} ,\vec{y},0,f_{k+1}} \leq \varepsilon \delta_4   \right)\leq \varepsilon.$$
\end{proposition}

\begin{proof}
Let $\tilde{ \mathbb{P}}$ be the probability measure on $\mathfrak{X}_f$ given in \eqref{def:J'} and let $\mathfrak{X}'$ be the collection of functions $\mathcal{J}'=(\mathcal{J}'_1,\dots, \mathcal{J}'_k)$ with $\mathcal{J}'_i\in C([\ell_i,-L]\cup [L,r_i],\mathbb{R})$. We denote by $\mathbb{P}'$ and $\tilde{\mathbb{P}}'$ the marginal law of $\mathbb{P}$ and $\tilde{\mathbb{P}}$ on $\mathfrak{X}'$ respectively.

Recall that the Radon-Nikodym derivatives $\textup{d}\mathbb{P}/\textup{d}\tilde{\mathbb{P}}$ is given in \eqref{equ:J/J'}. This implies  
\begin{align}\label{eq:dJdJ'}
\frac{\textup{d}\mathbb{P}' }{\textup{d}\tilde{\mathbb{P}}' }(\mathcal{J}' ) =\frac{1}{Z'} Z^{1,k,(-L,L),\vec{x} ,\vec{y},0,f_{k+1}}.
\end{align}
Here $\vec{x}=(\mathcal{J}'_i(-L))_{i=1}^k$, $\vec{y}=(\mathcal{J}'_i(L))_{i=1}^k$ and $Z'$ is a normalizing constant. Denote by $\tilde{\mathbb{E}}$ and $\tilde{\mathbb{E}}'$ the expectation with respect to $\tilde{\mathbb{P}}$ and $\tilde{\mathbb{P}}'$ respectively. Then $$ Z'=\tilde{\mathbb{E}}'\left[  Z^{1,k,(-L,L),\vec{x} ,\vec{y},0,f_{k+1}} \right]=\tilde{\mathbb{E}}\left[  Z^{1,k,(-L,L),\vec{x} ,\vec{y},0,f_{k+1}} \right].$$
 
Let $\delta_2=\delta_2(\alpha, k,L,R)$ and $\delta_3=\delta_3(k,L,R)$ be the constants in Proposition~\ref{pro:spacing-k} and Corollary \ref{cor:Z_lowerbound-k}. It follows from Corollary~\ref{cor:Z_lowerbound-k} that $Z'\geq  \delta_2\delta_3.$
From \eqref{eq:dJdJ'}, we have
\begin{align*}
&=\mathbb{P}'\left( Z^{1,k,(-L,L),\vec{x} ,\vec{y},0,f_{k+1}} \leq \varepsilon \delta_2\delta_3  \right)\\ 
&=\tilde{\mathbb{E}}' \left[ \frac{1}{Z'} Z^{1,k,(-L,L),\vec{x} ,\vec{y},0,f_{k+1}} \cdot\mathbbm{1}\left\{Z^{1,k,(-L,L),\vec{x} ,\vec{y},0,f_{k+1}} \leq \varepsilon \delta_2\delta_3\right\}\right]\\
&\leq\ \frac{\varepsilon\delta_2\delta_3 }{Z'} \cdot \tilde{\mathbb{E}}'\left[\mathbbm{1}\left\{Z^{1,k,(-L,L),\vec{x} ,\vec{y},0,f_{k+1}} \leq \varepsilon \delta_2\delta_3\right\}  \right] \leq \varepsilon.
\end{align*}
Picking $\delta_4=\delta_2\delta_3$. Thus the assertion follows by noting that 
$$\mathbb{P}\left( Z^{1,k,(-L,L),\vec{x} ,\vec{y},0,f_{k+1}} \leq \varepsilon \delta_2\delta_3   \right)=\mathbb{P}' \left( Z^{1,k,(-L,L),\vec{x} ,\vec{y},0,f_{k+1}} \leq \varepsilon \delta_2\delta_3  \right).$$
\end{proof}
\subsection{Proof of Proposition~\ref{pro:Z}}\label{sec:proofZZ}
Recall that $\mathcal{P}$, $b\mathcal{P}$ and $\mathcal{G}(k,L,R)$ are defined at the beginning of Section~\ref{sec:proofZ}. Consider the good boundary event
$$\mathsf{GB}_N(k,L,R):=\left\{\left. \mathcal{L}^N \right|_{b\mathcal{P}}\in\mathcal{G}(k,L,R) \right\}.$$
The following lemma shows that $\mathsf{GB}_N(k,L,R)$ is a typical event under the law of $\mathcal{L}^N$.
\begin{lemma}\label{lem:GB}
Fix $\alpha\geq 0$, $k\in\mathbb{N}$, $L\geq 1$. For any $\varepsilon>0$, there exists $R_3=R_3(\alpha,k,L,\varepsilon)$ and $N_7=N_7(\alpha,k,L,\varepsilon)$ such that the following holds. For any $N\geq  N_7$, and $R\geq R_3$ we have
\begin{align*}
\mathbb{P}(\mathsf{GB}_N^{\textup{c}}(k,L,R))\leq \varepsilon.
\end{align*}
\end{lemma}
\begin{proof}
The assertion follows directly from Propositions \ref{pro:sup} and \ref{pro:inf}.
\end{proof}
Now we are ready to prove Proposition \ref{pro:Z}. 
\begin{proof}[{\bf Proof of Proposition \ref{pro:Z}}]
Fix $k\in\mathbb{N}$ and $L\geq 1$. Given $\varepsilon>0$, let $R=R_3(\alpha,k,L,2^{-1}\varepsilon)$ and $N_7=N_7(\alpha,k,L,2^{-1}\varepsilon)$ in Lemma \ref{lem:GB}. To simplify notation, we write $\mathsf{GB}$ for $\mathsf{GB}_N(k,L,R)$. From Lemma~\ref{lem:GB}, for all $N\geq N_7$, it holds that 
\begin{equation}\label{equ:07051634}
\mathbb{P}(\mathsf{GB}_N^{\textup{c}})\leq 2^{-1}\varepsilon. 
\end{equation}

Let $\Fext^*$ be the sigma-field generated by the restriction of $\mathcal{L}^N $ on $[1,N]_\mathbb{Z} \times [-4N,\infty) \setminus\mathcal{P}.$ Let $\delta_4=\delta_4(\alpha, k,L,R)$ be the constant in Proposition~\ref{pro:Z_small-k}. Applying the Gibbs property of $\mathcal{L}^N$ (see Definition~\ref{def:BesselGP}), we have
\begin{align}\label{equ:06290331}
   \mathbb{E}\left[\mathbbm{1}\{ Z^{1,k,(-L,L), \vec{x},\vec{y} ,+\infty,\mathcal{L}^N_{k+1}} \leq 2^{-1}\varepsilon  \delta_4 \} \,|\, \Fext^*\right]
= \mathbb{P}\left(Z^{1,k,(-L,L), \vec{z},\vec{w} ,0,f_{k+1}}\leq  2^{-1}\varepsilon  \delta_4  \right).
\end{align} 
On the left hand side of \eqref{equ:06290331}, $\vec{x}=(\mathcal{L}^N_j (-L))_{j=1}^k$ and $\vec{y}=(\mathcal{L}^N_j (L))_{j=1}^k$; on the right hand side of \eqref{equ:06290331}, $\mathbb{P}$ is defined in \eqref{def:J} with $f=\mathcal{L}^N|_{b\mathcal{P}}$, $\vec{z}=(\mathcal{J}_j (-L))_{j=1}^k$, $\vec{w}=(\mathcal{J}_j (L))_{j=1}^k$ and $f_{k+1}=\mathcal{L}^N_{k+1}\big|_{[-L,L]}$. Applying Proposition \ref{pro:Z_small-k},
\begin{align}\label{equ:07051636}
 \mathbbm{1}_{\mathsf{GB}}\cdot\mathbb{E}\left[\mathbbm{1}\{ Z^{1,k,(\ell,r), \vec{x},\vec{y} ,0,f_{k+1}} \leq 2^{-1} \varepsilon \delta_4 \} \,|\, \Fext^*\right] \leq  2^{-1} \varepsilon \cdot \mathbbm{1}_{\mathsf{GB}}.
\end{align}
Combining \eqref{equ:07051634} and \eqref{equ:07051636}, we conclude
\begin{align*}
&\mathbb{P}\left( Z^{1,k,(-L,L), \vec{x},\vec{y} ,+\infty,f_{k+1}} \leq 2^{-1}\varepsilon  \delta_4    \right) 
\leq   \mathbb{E}\left[\mathbbm{1}_{ \mathsf{GB}}\cdot\mathbbm{1}\left\{  Z^{1,k,(-L,L), \vec{x},\vec{y} ,+\infty,f_{k+1}} \leq 2^{-1}\varepsilon  \delta_4  \right\}   \right]+\mathbb{P}(\mathsf{GB}^{\textup{c}}) 
\leq  \varepsilon. 
\end{align*}
The assertion then follows by taking $\delta_1=2^{-1}\varepsilon \delta_4$ and $N_6=N_7.$ 
\end{proof} 

\subsection{Proof of Proposition \ref{pro:spacing-k}}\label{sec:spacing-k}
This section is devoted to proving Proposition \ref{pro:spacing-k}, which shows that the curves  are well separated at $\pm L$. Recall that $\tilde{\mathbb{P}}$ is the probability measure on $\mathfrak{X}_f$ defined in \eqref{def:J'}. We will adapt a {\bf two-step} inductive resampling in \cite{Wu21a}. The goal is to have the curves stay in the preferable region. In the {\bf first} step, we inductively lower curves to the desired heights. The {\bf second} step is carried out inductively to raise curves properly in order to separate them in the desired windows. These two steps are carried out in Lemma \ref{lem:lowerbound-k} and Lemma \ref{lem:spacing-k} respectively.\\[-0.3cm]

More precisely, for $j\in [1,k]_{\mathbb{Z}}$, denote
\begin{align*}
\mathsf{A}_j:=  \left\{ \sup_{t\in [\ell_{j+1},r_{j+1}]   }{\mathcal{J}}_j(t)\leq (2j-3/2)(2k)^{-1} R^{-1} \right\}.
\end{align*}
Also, consider
\begin{equation*}
\begin{split}
\mathsf{F}_j:=&\left\{  {\mathcal{J}}_j \in \left[ (2j-2)(2k)^{-1}  R^{-1} ,(2j-1)(2k)^{-1} R^{-1} \right]\ \textup{on}\ [\ell_{j+1} , r_{j+1}]  \right\} \\
 &\qquad\qquad\cap \left\{  {\mathcal{J}}_j \geq (2j-2)(2k)^{-1} R^{-1} \ \textup{on}\ [\ell_{j} ,r_{j}] \right\}.
\end{split}
\end{equation*}
Lemma \ref{lem:lowerbound-k} provides a lower bound of $\tilde{ \mathbb{P}} \big(\cap_{j=1}^k \mathsf{A}_j \big)$. Built on Lemma \ref{lem:lowerbound-k}, Lemma \ref{lem:spacing-k} gives a lower bound of $\tilde{\mathbb{P}}\big(\cap_{j=1}^k \mathsf{F}_j \big)$. Proposition \ref{pro:spacing-k} follows directly as $\cap_{j=1}^k\mathsf{F}_j \subset \mathsf{E}$. \\[-0.3cm]

We begin with Lemma \ref{lem:lowerbound-k}. Recall that given $\mathcal{J}\in C(\mathcal{P},\mathbb{R})$, we may extend the domain of $\mathcal{J}$ as $\mathcal{J}_f$ given in \eqref{def:J_ext}.

\begin{lemma}\label{lem:lowerbound-k}
Fix $\alpha\geq 0$, $k\in\mathbb{N}$, $L\geq 1$ and $R\geq 1$. There exists a constant $\delta_5=\delta_5(\alpha, k,L,R)>0$ such that the following holds. For any $(f_1,f_2,\dots f_{k+1})\in \mathcal{G}(k,L,R)$, let $\tilde{\mathbb{P}}$ be the probability measure given in \eqref{def:J'}. Then we have  
\begin{align*}
\tilde{\mathbb{P}} \left(\cap_{j=1}^k \mathsf{A}_j \right)\geq \delta_5. 
\end{align*}
\end{lemma}

\begin{proof}
For notational ease, we use $\delta$ to denote positive constants that depend only on $\alpha$, $k,L$ and $R$. The exact value of $\delta$ may change from line to line.\\

We start with the case $k=1$, i.e., showing the lower bound for $\tilde{\mathbb{P}}(\mathsf{A}_1)$. Let $\tilde{\mathbb{E}}$ be the expectation of $\tilde{\mathbb{P}}$. From the Gibbs property and Remark \ref{rmk:P_J'},  
\begin{align*}
\tilde{\mathbb{E}} [\mathbbm{1}_{\mathsf{A}_1} \,|\,\Fext(\{1\}\times (\ell_1,r_1))]= {\mathbb{P}_{(-L,L)}^{1,1,(\ell_1,r_1),f_1(\ell_1),f_1(r_1),0, \mathcal{{J}}_{f,2}} ( {\mathsf{A}_1})}
\end{align*}  
with $\mathcal{J}_{f,2}$ defined in \eqref{def:J_ext}. Note that equivalently we have $${\mathbb{P}_{ (-L,L)}^{1,1,(\ell_1,r_1),f_1(\ell_1),f_1(r_1),0, \mathcal{J}_{f,2}}}={\mathbb{P}^{1,1,(\ell_1,r_1),f_1(\ell_1),f_1(r_1),0, g_2}}$$
with
\begin{align*}
g_2(x)\coloneqq\left\{\begin{array}{cc}
\mathcal{{J}}_{2,f}(x) & x\in [\ell_1,-L]\cup[L,r_1],\\
+\infty & x\in (-L,L).
\end{array} \right.
\end{align*}
Note that $g_2$ satisfies the continuity assumption in Definition~\ref{def:continuityassumption}. This implies that the stochastic monotonicity, Proposition~\ref{lem:monotonicity}, applies. Therefore,
\begin{align*}
{\mathbb{P}^{1,1,(\ell_1,r_1),f_1(\ell_1),f_1(r_1),0, g_2} ({\mathsf{A}_1})}&\geq{\mathbb{P}^{1,1,(\ell_1,r_1),f_1(\ell_1),f_1(r_1),0, +\infty} ( {\mathsf{A}_1})}=\mathbb{P}_{\free}^{1,1,(\ell_1,r_1),f_1(\ell_1),f_1(r_1) }\left( {\mathsf{A}_1} \right). 
\end{align*}
Let
\begin{align*}
a_1\coloneqq \inf\left\{\mathbb{P}_{\free}^{1,1,(\ell_1,r_1),x,y }\left( {\mathsf{A}_1} \right)\,\Big|\, |x|\leq R,\ |y|\leq R \right\}.
\end{align*}
From Lemma~\ref{lem:compactcoupling}, $a_1$ is positive and depends only on $\alpha, k,L$ and $R$. As a result, it holds that
\begin{align*}
\tilde{\mathbb{P}}(\mathsf{A}_1)=  \tilde{\mathbb{E}}\left[ \tilde{\mathbb{E}}  [\mathbbm{1}_{\mathsf{A}_1} \,|\,\Fext(\{1\}\times (\ell_1,r_1))] \right]\geq \mathbb{P}_{\free}^{1,1,(\ell_1,r_1),f_1(\ell_1),f_1(r_1) }\left( {\mathsf{A}_1} \right)  \geq \delta.
\end{align*}
This proves the desired result for $k=1$.\\[-0.3cm]

Now we proceed by induction. Assume for some $2\leq j\leq k$, there exists $\delta>0$ such that $\tilde{\mathbb{P}}\left(\cap_{i=1}^{j-1} \mathsf{A}_{i} \right)\geq \delta .$ We aim to show that for a smaller $\delta>0$, it holds that $\tilde{\mathbb{P}} \left(\cap_{i=1}^{j}\mathsf{A}_{i} \right)\geq \delta.$ By the Gibbs property and Remark \ref{rmk:P_J'},
\begin{align*}
\tilde{\mathbb{E}} [\mathbbm{1}_{\mathsf{A}_j} \,|\,\Fext(\{j\}\times (\ell_j,r_j))]=      \mathbb{P}^{j,j,(\ell_j,r_j),f_j(\ell_j),f_j(r_j), {\mathcal{J}}_{j-1 }, {\mathcal{J}}_{f,j+1} }_{ (-L,L)}(\mathsf{A}_j),    
\end{align*}
with $ {\mathcal{J}}_{f,j+1} $ defined in \eqref{def:J_ext} and we adopt the convention that $ {\mathcal{J}}_{f,k+1}=f_{k+1}$. Note that
\begin{align*}
\mathbb{P}^{j,j,(\ell_j,r_j),f_j(\ell_j),f_j(r_j), {\mathcal{J}}_{j-1 }, {\mathcal{J}}_{j+1,f} }_{ (-L,L)}=\mathbb{P}^{j,j,(\ell_j,r_j),f_j(\ell_j),f_j(r_j),g_{j-1 },g_{j+1} } 
\end{align*}
with
\begin{align*}
g_{j-1}(x)\coloneqq\left\{\begin{array}{cc}
 {\mathcal{J}}_{j-1}(x) & x\in [\ell_j,-L]\cup[L,r_j],\\
0 & x\in (-L,L),\end{array} \right.
\end{align*}
and
\begin{align*}
g_{j+1}(x)\coloneqq\left\{\begin{array}{cc}
 {\mathcal{J}}_{f,j+1 }(x) & x\in [\ell_j,-L]\cup[L,r_j],\\
+\infty & x\in (-L,L).
\end{array} \right.
\end{align*}
We deduce that
\begin{align*}
\tilde{\mathbb{E}}[\mathbbm{1}_{\mathsf{A}_j}\,|\,\Fext(\{j\}\times (\ell_j,r_j)) ]&=\mathbb{P}^{j,j,(\ell_j,r_j),f_j(\ell_j),f_j(r_j),g_{j-1 },g_{j+1} } (\mathsf{A}_j)\\
&\geq \mathbb{P}^{j,j,(\ell_j,r_j),f_j(\ell_j),f_j(r_j),g_{j-1 },+\infty }  (\mathsf{A}_j)\\
&=\frac{1}{Z }\mathbb{E}_{\free}^{j,j,(\ell_j,r_j),f_j(\ell_j),f_j(r_j) }\left[\mathbbm{1}_{\mathsf{A}_j} \cdot \mathbbm{1}\{ g_{j-1}<\mathcal{J}_j\ \textup{on}\ [\ell_j,r_j] \}  \right].\\
&\geq \mathbb{E}_{\free}^{j,j,(\ell_j,r_j),f_j(\ell_j),f_j(r_j) }\left[\mathbbm{1}_{\mathsf{A}_j} \cdot \mathbbm{1}\{ g_{j-1}<\mathcal{J}_j\ \textup{on}\ [\ell_j,r_j] \}  \right]. 
\end{align*}
Here the in the second equality we use $Z$ to abbreviate $\mathbb{E}_{\free}^{j,j,(\ell_j,r_j),f_j(\ell_j),f_j(r_j) }\left[ \mathbbm{1}\{ g_{j-1}<\mathcal{J}_j\ \textup{on}\ [\ell_j,r_j] \} \right].$ In the first inequality we apply stochastic monotonicity and in the second inequality we use the fact that the normalizing constant is bounded above by 1.\\[-0.3cm]
 
Now we proceed to find a lower bound for
$$\mathbbm{1}_{\mathsf{A}_{j-1}}\cdot \mathbb{E}_{\free}^{j,j,(\ell_j,r_j),f_j(\ell_j),f_j(r_j) }\left[\mathbbm{1}_{\mathsf{A}_j} \cdot \mathbbm{1}\{ g_{j-1}<\mathcal{J}_j\ \textup{on}\ [\ell_j,r_j] \}  \right].$$
Consider the event
\begin{align*}
\mathsf{D}_j:= &\big\{ \sup_{t\in [\ell_{j+1},  r_{j+1}]  } {\mathcal{J}}_j(t)\leq (2j-3/2 )(2k)^{-1}R^{-1} \big\} \cap \big\{ \inf_{t\in [\ell_{j} ,r_{j}]  }  {\mathcal{J}}_j(t)\geq (2j-2 )(2k)^{-1}R^{-1}   \big\}.
\end{align*}
It is straightforward to check that $\mathsf{D}_j\subset \mathsf{A}_j$. As $\mathsf{D}_j$ and $\mathsf{A}_{j-1}$ occur, $ {\mathcal{J}}_{j-1}<  {\mathcal{J}}_{j}$ over $[\ell_{j} ,r_{j}]$. Hence
\begin{align*}
\mathbbm{1}_{\mathsf{A}_j} \cdot \mathbbm{1}\{ g_{j-1}<\mathcal{J}_j\ \textup{on}\ [\ell_j,r_j]\} \geq \mathbbm{1}_{\mathsf{D}_j}\cdot\mathbbm{1}_{\mathsf{A}_{j-1}}.
\end{align*}
Consequently,
\begin{align*}
\mathbbm{1}_{\mathsf{A}_{j-1}}\cdot \tilde{\mathbb{E}} [\mathbbm{1}_{\mathsf{A}_j}\,|\,\Fext(\{j\}\times (\ell_j,r_j)) ]&\geq \mathbbm{1}_{\mathsf{A}_{j-1}} \mathbb{E}_{\free}^{j,j,(\ell_j,r_j),f_j(\ell_j),f_j(r_j) }\left[\mathbbm{1}_{\mathsf{A}_j} \cdot \mathbbm{1}\{ g_{j-1}<\mathcal{J}_j\ \textup{on}\ (\ell_j,r_j) \}   \right]\\
&\geq \mathbbm{1}_{\mathsf{A}_{j-1}} \mathbb{P}_{\free}^{j,j,(\ell_j,r_j),f_j(\ell_j),f_j(r_j) } (\mathsf{D}_j).
\end{align*}
Let
\begin{align*}
a_j&:=  \inf \left\{\mathbb{P}_{\free}^{j,j,(\ell_j,r_j),x,y } (\mathsf{D}_j)\,|\,R^{-1} \leq |x|,|y|\leq R \right\}.
\end{align*}
From Lemma~\ref{lem:compactcoupling}, $a_j$ is positive and depends only on $\alpha, k,L,R$. Together with the induction hypothesis, we deduce
\begin{align*}
\tilde{\mathbb{P}}\big( \cap_{i=1}^{j } \mathsf{A}_{i} \big)=&\tilde{\mathbb{E}} \bigg[\prod_{i=1}^{j-1}\mathbbm{1}_{\mathsf{A}_i}\cdot \tilde{\mathbb{E}} [\mathbbm{1}_{\mathsf{A}_j}\,|\,\Fext(\{j\}\times (\ell_j,r_j)) ] \bigg] \geq a_j \cdot \tilde{\mathbb{P}} \big( \cap_{i=1}^{j-1 } \mathsf{A}_{i} \big)
\geq \delta.
\end{align*}
This completes the induction argument and hence proves the desired result.
\end{proof}

We proceed to prove Lemma \ref{lem:spacing-k}. For $j\in [1,k]_{\mathbb{Z}}$, recall that
\begin{equation*}
\begin{split}
\mathsf{F}_j:=&\left\{  {\mathcal{J}}_j \in \left[ (2j-2)(2k)^{-1}  R^{-1} ,(2j-1)(2k)^{-1} R^{-1} \right]\ \textup{on}\ [\ell_{j+1} , r_{j+1}]  \right\} \\
 &\qquad\qquad\cap \left\{  {\mathcal{J}}_j \geq (2j-2)(2k)^{-1} R^{-1} \ \textup{on}\ [\ell_{j} ,r_{j}] \right\}.
\end{split}
\end{equation*}
The condition ${\mathcal{J}}_j \in [(2j-2)(2k)^{-1}R^{-1} ,(2j-1)(2k)^{-1}R^{-1} ]$ is what we want to achieve while the other one is necessary along the induction argument we perform.

\begin{lemma}\label{lem:spacing-k}
Fix $\alpha\geq 0$, $k\in\mathbb{N}$, $L\geq 1$ and $R\geq 1$. There exist a constant $\delta_6=\delta_6(\alpha, k,L,R)>0$ such that the following holds. For any $(f_1,f_2,\dots f_{k+1})\in \mathcal{G}(k,L,R)$, let $\tilde{\mathbb{P}}$ be the probability measure given in \eqref{def:J'}. Then we have  
 
\begin{align*}
\tilde{\mathbb{P}}\left(\cap_{j=1}^k \mathsf{F}_j \right)\geq \delta_6.
\end{align*}
\end{lemma}

\begin{proof}
To simplify the notation, we use $\delta$ to denote positive constants that depend only on $\alpha, k,L$ and $R$. The exact value of $\delta$ may change from line to line.\\[-0.3cm]

We run a resampling in a reversed order starting from the $k$th-layer and argue inductively. We start by showing a lower bound for  $\tilde{\mathbb{P}} \big(\cap_{j=1}^{k-1} \mathsf{A}_j\cap \mathsf{F}_k \big)$. \\[-0.3cm]

Let $[\mathfrak{l}_k, \mathfrak{r}_k]$ be a $\{k\}$-stopping domain such that
\begin{align*}
\mathfrak{l}_k:=&\sup \big\{ t_0\in [\ell_k ,\ell_{k+1}] \ |\  {\mathcal{J}}_{k}(t)>  (2k-3/2)(2k)^{-1}R^{-1} \ \textup{for all}\ t\in [\ell_k,t_0]\big\},\\
\mathfrak{r}_k:=&\inf\big\{t_0\in [r_{k+1},r_k] \ |\  {\mathcal{J}}_{k}(t)> (2k-3/2)(2k)^{-1}R^{-1}\ \textup{for all}\ t\in [t_0,r_k]\big\}.
\end{align*}
We set $\mathfrak{l}_k=\ell_k$ or $\mathfrak{r}_k=r_k$ if the set is empty respectively.\\[-0.3cm]

Consider the event
\begin{align*}
\mathsf{F}'_k:=\left\{  {\mathcal{J}}_k(\mathfrak{l}_k)= {\mathcal{J}}_k(\mathfrak{r}_k)= (2k-3/2)(2k)^{-1} R^{-1} \right\}.
\end{align*}
Because $f\in\mathcal{G}(k,L,R)$, $ {\mathcal{J}}_k(\ell_k), {\mathcal{J}}_k(r_k)\geq R^{-1}$. This implies $ \mathsf{A}_k\subset\mathsf{F}'_k$. In view of Lemma \ref{lem:lowerbound-k}, $$\tilde{\mathbb{P}} \left(\cap_{j=1}^{k-1} \mathsf{A}_j \cap \mathsf{F}'_k \right)\geq  \delta.$$ 
We would like to have $ {\mathcal{J}}_k$ stay in the preferable region $[(2k-2)(2k)^{-1}R^{-1} ,(2k-1)(2k)^{-1}R^{-1}] $ over $[\mathfrak{l}_k,\mathfrak{r}_k]$. Let
\begin{align*}
\mathsf{F}_k''\coloneqq\{{\mathcal{J}}_k \in [(2k-2)(2k)^{-1}R^{-1} ,(2k-1)(2k)^{-1}R^{-1}]\ \textup{on}\ [\mathfrak{l}_k,\mathfrak{r}_k ]\}.
\end{align*}
Note that the occurrence of $\mathsf{A}_{k-1}\cap\mathsf{F}''_{k}$ implies ${\mathcal{J}}_{k-1}$ and ${\mathcal{J}}_{k}$ are ordered. In other words, $\mathsf{A}_{k-1}\cap\mathsf{F}''_{k}\subset \{ {\mathcal{J}}_{k-1}< {\mathcal{J} }_k< f_{k+1}\ \textup{on}\  [\mathfrak{l}_k, \mathfrak{r}_k ]\}.$ Hence
\begin{align*}
 \mathbbm{1}\{  {\mathcal{J}}_{k-1}<  {\mathcal{J} }_k< f_{k+1}\ \textup{on}\  [\mathfrak{l}_k, \mathfrak{r}_k ]\}  \geq  \mathbbm{1}_{\mathsf{A}_{k-1} }\cdot \mathbbm{1}_{\mathsf{F}_{k}'' } . 
\end{align*} 
It follows that
\begin{align*}
&\quad\ \mathbbm{1}_{\mathsf{A}_{k-1}}\cdot \mathbbm{1}_{\mathsf{F}'_{k}}\cdot \tilde{\mathbb{E}} [\mathbbm{1}_{\mathsf{F}''_k}\,|\,\Fext(\{k\}\times (\mathfrak{l}_k,\mathfrak{r}_k))]\\
&=\mathbbm{1}_{\mathsf{A}_{k-1}}\cdot \mathbbm{1}_{\mathsf{F}'_{k}}\cdot\mathbb{P}^{k,k,(\mathfrak{l}_k ,\mathfrak{r}_k),x,y, {\mathcal{J}}_{k-1},f_{k+1} }_{(-L,L)}\left( \mathsf{F}_k''  \right) \\
& \geq \mathbbm{1}_{\mathsf{A}_{k-1}}\cdot \mathbbm{1}_{\mathsf{F}'_{k}}\cdot{\mathbb{E}^{k,k,(\mathfrak{l}_k ,\mathfrak{r}_k),x,y }_{\free}\left[\mathbbm{1}_{\mathsf{F}_k''} \cdot \mathbbm{1}\{  {\mathcal{J}}_{k-1}<  {\mathcal{J} }_k< f_{k+1}\ \textup{on}\  [\mathfrak{l}_k, \mathfrak{r}_k ]\}  \right] } \\
&\geq \mathbbm{1}_{\mathsf{A}_{k-1}}\cdot \mathbbm{1}_{\mathsf{F}'_{k}} \cdot \mathbb{P}^{k,k,(\mathfrak{l}_k ,\mathfrak{r}_k),x,y   }_{\free} (\mathsf{F}''_k).
\end{align*}
Here $x=y=(2k-3/2)(2k)^{-1}R^{-1}$. Take 
\begin{align*}
b_k\coloneqq \inf_{\ell, r}\left\{ \mathbb{P}^{k,k,(\ell ,r),x,y   }_{\free} \left(\mathcal{J} _k \in [ (2k-2)(2k)^{-1} R^{-1}, (2k-1)(2k)^{-1} R^{-1}] \right) \right\}.
\end{align*}
The infimum is taken over all $\ell\in [\ell_k,\ell_{k+1}]$ and $ r\in [r_{k+1},r_k]$. This implies 
$$\mathbbm{1}_{\mathsf{A}_{k-1}}\cdot \mathbbm{1}_{\mathsf{F}'_{k}}\cdot \tilde{\mathbb{E}} [\mathbbm{1}_{\mathsf{F}''_k}\,|\,\Fext(\{k\}\times (\mathfrak{l}_k,\mathfrak{r}_k))]\geq b_k \mathbbm{1}_{\mathsf{A}_{k-1}}\cdot \mathbbm{1}_{\mathsf{F}'_{k}}.$$
Together with $\cap_{j=1}^{k-1} \mathsf{A}_j \cap \mathsf{F}'_k\cap\mathsf{F}_k''\subset \cap_{j=1}^{k-1} \mathsf{A}_j \cap \mathsf{F}_k,$ we deduce
\begin{align*}
\tilde{\mathbb{P}}  \left(  \cap_{j=1}^{k-1} \mathsf{A}_j \cap   \mathsf{F}_k\right)\geq & \tilde{\mathbb{P}}   \left(  \cap_{j=1}^{k-1} \mathsf{A}_j \cap \mathsf{F}'_k\cap \mathsf{F}''_k\right)= \tilde{\mathbb{E}} \bigg[\prod_{j=1}^{k-1}\mathbbm{1}_{\mathsf{A}_j}\cdot\mathbbm{1}_{\mathsf{F}'_{k }}\cdot \tilde{\mathbb{E}} [\mathbbm{1}_{\mathsf{F}''_k}\,|\,\Fext(\{k\}\times (\mathfrak{l}_k,\mathfrak{r}_k))] \bigg]\\
\geq & b_k \cdot\tilde{\mathbb{P}} \left(  \cap_{j=1}^{k-1} \mathsf{A}_j \cap \mathsf{F}'_k \right) > \delta .
\end{align*}
In the last inequality, we used Lemma~\ref{lem:compactcoupling} to obtain the positivity of $b_k$.

\vspace{0.2cm}

We next proceed by a {\bf reversed} induction. Assume for some $1\leq i\leq k-1$, we have 
\begin{align*}
\tilde{\mathbb{P}} \left(\cap_{j=1}^{i} \mathsf{A}_j\cap \cap_{j=i+1}^{k} \mathsf{F}_j \right)\geq \delta.
\end{align*}
We aim to show that for a smaller $\delta>0$, it holds that 
\begin{align*}
\tilde{\mathbb{P}}\left(\cap_{j=1}^{i-1} \mathsf{A}_j\cap \cap_{j=i}^{k} \mathsf{F}_j \right)\geq  \delta.
\end{align*}
Here we adopt the convention that $\cap_{j=1}^{0} \mathsf{A}_j$ means the total probability space. \\[-0.3cm]

Let $[\mathfrak{l}_{i}, \mathfrak{r}_{i}]$ be a $\{i\}$-stopping domain such that
\begin{align*}
\mathfrak{l}_{i}:=&\sup\big\{ t_0\in [\ell_i ,\ell_{i+1}] \ |\  {\mathcal{J}}_{i}(t)>  (2i-3/2)(2k)^{-1} R^{-1}\ \textup{for all}\ t\in [\ell_i,x_0]\big\},\\
\mathfrak{r}_{i}:=&\inf\big\{ t_0\in [r_{i+1} ,r_{i }] \ |\  {\mathcal{J}}_{i}(t)  > (2i-3/2)(2k)^{-1}\cdot R^{-1}\ \textup{for all}\ t\in [ t_0,r_i]\big\}.
\end{align*}
We set $\mathfrak{l}_k=\ell_i$ or $\mathfrak{r}_i=r_i$ if the set is empty respectively. Consider the event
\begin{align*}
\mathsf{F}'_i:=\left\{  {\mathcal{J}}_i(\mathfrak{l}_i)={\mathcal{J}}_i(\mathfrak{r}_i)=(2i-3/2)(2k)^{-1} R^{-1}  \right\}.
\end{align*}
Because $f\in\mathcal{G}(k,L,R)$, $ {\mathcal{J}}_i(\ell_i), {\mathcal{J}}_i(r_i)\geq R^{-1} $. This implies $ \mathsf{A}_i\subset\mathsf{F}'_i$. 

We would like to have $ {\mathcal{J}}_i$ stay in the preferable region $[(2i-2)(2k)^{-1}R^{-1},(2i-1)(2k)^{-1}R^{-1}]$ over $[\mathfrak{l}_i,\mathfrak{r}_i]$.
Let
\begin{align*}
\mathsf{F}_i''\coloneqq \left\{ \mathcal{J}_i \in [(2i-2)(2k)^{-1}R^{-1},(2i-1)(2k)^{-1}R^{-1}]\ \textup{on}\ [\mathfrak{l}_i,\mathfrak{r}_i ]\right\}.
\end{align*}
Note that the occurrence of $\mathsf{A}_{i-1}\cap\mathsf{F}''_{i}\cap\mathsf{F}_{i+1}$ implies ordering between $ {\mathcal{J}}_{i-1}$, $ {\mathcal{J}}_{i}$ and $ {\mathcal{J}}_{i+1,f}$, i.e. $\mathsf{A}_{i-1}\cap\mathsf{F}''_{i}\cap\mathsf{F}_{i+1}\subset \{  {\mathcal{J}}_{i-1}<  {\mathcal{J} }_i< {\mathcal{J}}_{i+1,f} \ \textup{in}\  [\mathfrak{l}_i, \mathfrak{r}_i ]\}.$ Hence
\begin{align*}
\mathbbm{1}\{  {\mathcal{J}}_{i-1}<  {\mathcal{J} }_i< {\mathcal{J}}_{i+1,f} \ \textup{on}\  [\mathfrak{l}_i, \mathfrak{r}_i ]\}\geq \mathbbm{1}_{\mathsf{A}_{i-1} }\cdot \mathbbm{1}_{\mathsf{F}_{i}'' }.
\end{align*} 
As a result,
\begin{align*}
&\quad\ \mathbbm{1}_{\mathsf{A}_{i-1}}\cdot \mathbbm{1}_{\mathsf{F}'_{i}}\cdot\mathbbm{1}_{\mathsf{F}_{i+1}}\cdot \tilde{\mathbb{E}} [\mathbbm{1}_{\mathsf{F}''_i}\,|\,\Fext(\{i\}\times (\mathfrak{l}_i,\mathfrak{r}_i))]\\
&=\mathbbm{1}_{\mathsf{A}_{i-1}}\cdot \mathbbm{1}_{\mathsf{F}'_{i}}\cdot\mathbbm{1}_{\mathsf{F}_{i+1}}\cdot\mathbb{P}^{i,i,(\mathfrak{l}_i ,\mathfrak{r}_i),x,y, \mathcal{J}_{i-1}, \mathcal{J}_{f,i+1}  }_{ (-L,L)}\left( \mathsf{F}_i''  \right) \\
& \geq \mathbbm{1}_{\mathsf{A}_{i-1}}\cdot \mathbbm{1}_{\mathsf{F}'_{i}}\cdot\mathbbm{1}_{\mathsf{F}_{i+1}}\cdot{\mathbb{E}^{i,i,(\mathfrak{l}_i ,\mathfrak{r}_i),x,y }_{\free}\left[\mathbbm{1}_{\mathsf{F}_i''}\cdot \mathbbm{1}\{ \mathcal{J}_{i-1}< \mathcal{J}_i<\mathcal{J}_{i+1,f} \ \textup{in}\  [\mathfrak{l}_i, \mathfrak{r}_i ]\}  \right] } \\
&\geq \mathbbm{1}_{\mathsf{A}_{i-1}}\cdot \mathbbm{1}_{\mathsf{F}'_{i}}\cdot\mathbbm{1}_{\mathsf{F}_{i+1}}   \cdot \mathbb{P}^{i,i,(\mathfrak{l}_i ,\mathfrak{r}_i),x,y   }_{\free} (\mathsf{F}''_i).
\end{align*}
Here $x=y=(2i-3/2)(2k)^{-1}R^{-1}$. Let \begin{align*}
b_i\coloneqq \inf_{\ell, r}\left\{ \mathbb{P}^{i,i,(\ell ,r),x,y   }_{\free} \left(\mathcal{J}_i \in [(2i-2)(2k)^{-1}R^{-1} ,(2i-1)(2k)^{-1}R^{-1} ] \right)\right\}.
\end{align*}
The infimum is taken over all $\ell\in [\ell_i,\ell_{i+1}]$ and $r\in [r_{i+1},r_i]$. Then we have
$$\mathbbm{1}_{\mathsf{A}_{i-1}}\cdot \mathbbm{1}_{\mathsf{F}'_{i}}\cdot\mathbbm{1}_{\mathsf{F}_{i+1}}\cdot \tilde{\mathbb{E}} [\mathbbm{1}_{\mathsf{F}''_i}\,|\,\Fext(\{i\}\times (\mathfrak{l}_i,\mathfrak{r}_i))]\geq  b_i \mathbbm{1}_{\mathsf{A}_{i-1}}\cdot \mathbbm{1}_{\mathsf{F}'_{i}}\cdot\mathbbm{1}_{\mathsf{F}_{i+1}}.$$ Together with $ \mathsf{F}'_i\cap\mathsf{F}_i''\subset   \mathsf{F}_i,$ we have
\begin{align*}
\tilde{\mathbb{P}} \left(\cap_{j=1}^{i-1} \mathsf{A}_j\cap \cap_{j=i}^{k} \mathsf{F}_j \right)\geq & \tilde{\mathbb{P}} \left(\cap_{j=1}^{i-1} \mathsf{A}_j  \cap \cap_{j=i}^{k} \mathsf{F}_j \cap\mathsf{F}_i'\cap\mathsf{F}_i''\right)\\
=& \tilde{\mathbb{E}} \bigg[\prod_{j=1}^{i-1}\mathbbm{1}_{\mathsf{A}_j}\cdot  \prod_{j=i+1}^k \mathbbm{1}_{\mathsf{F}_j}\cdot\mathbbm{1}_{\mathsf{F}'_{i }} \cdot \tilde{\mathbb{E}} [\mathbbm{1}_{\mathsf{F}''_i}\,|\,\Fext(\{i\}\times (\mathfrak{l}_i,\mathfrak{r}_i))] \bigg]\\
\geq & b_i \cdot \tilde{\mathbb{P}} \left(\cap_{j=1}^{i-1} \mathsf{A}_j \cap \cap_{j=i+1}^{k} \mathsf{F}_j \cap\mathsf{F}_i' \right)\\
\geq & b_i \cdot \tilde{\mathbb{P}} \left(\cap_{j=1}^{i} \mathsf{A}_j \cap \cap_{j=i+1}^{k} \mathsf{F}_j  \right) \geq \delta.
\end{align*}
In the second to last inequality, we used $\mathsf{A}_i\subset \mathsf{F}'_i$. In the last inequality, we used the induction hypothesis and Lemma~\ref{lem:compactcoupling}. The induction argument is finished and this finishes the argument.
\end{proof}

\section{Proof of the main theorem}\label{sec:proofmainthm}
In this section we present the proof of our main theorem, Theorem \ref{thm:main}. That is, we show that the scaled line ensembles $\mathcal{L}^{N,\alpha}$ defined in \eqref{eqn:edgeLE2} is tight as $N$ varies. Moreover, any subsequential limit also enjoys the squared Bessel Gibbs property.\\[-0.3cm]

For the tightness of $\mathcal{L}^{N,\alpha}$ in the locally uniform topology, we use a criterion from \cite[Theorem 7.3]{Bilconv}. Similar to the Arzel\`a–Ascoli theorem, one needs to check the one-point tightness (as real-valued random variables) and control the modulus of continuity. See Lemma~\ref{lem:tightcriterion} for more details. For us, one-point tightness follows from Theorem~\ref{thm:1} and we focus on the modulus of continuity. The idea is to use the squared $\alpha$-Bessel Gibbs property, together with the lower bound on the random normalizing constants in Proposition~\ref{pro:Z}, to show that $\mathcal{L}^{N,\alpha}$ locally behave like independent squared $\alpha$-Bessel bridges, whose modulus of continuity is controlled by Lemma~\ref{lem:Besselmodulus}.


To show any subsequential limit has the Bessel Gibbs property, we adopt the framework introduced in \cite{CH14}. Let us briefly explain the main idea. Let $(a,b)\subset\mathbb{R}$, $x,y\in[0,\infty) $ and $f(t),g(t)$ be two continuous functions defined on $[a,b]$. We assume $f(t)<g(t)$, $f(a)<x<g(a)$ and $f(b)<y<g(b)$. Let $\mathcal{Q}_*(t)$ be a squared $\alpha$-Bessel bridge on $[a,b]$ which has entrance and exit values $x$ and $y$, and is conditioned not to intersect with $f(t)$ and $g(t)$. A natural way to sample $\mathcal{Q}_*(t)$ is to consider countable independent squared $\alpha$-Bessel bridges $\mathcal{Q}_1(t),\mathcal{Q}_2(t),\dots $ with the same entrance and exit value. Let $\ell$ be the (random) minimum integer that $f(t)<\mathcal{Q}_\ell(t)<g(t)$. Then $\mathcal{Q}_{\ell}(t)$ has the same distribution as $\mathcal{Q}_*(t)$. Compared to $\mathcal{Q}_*(t)$, $\mathcal{Q}_\ell(t)$ is more tractable when the boundary data $(x,y,f,g)$ change. This point of view allows us to prove the limiting squared $\alpha$-Bessel Gibbs property. Once again, we will denote $\mathcal{L}^{N,\alpha}$ by $\mathcal{L}^N$ in the rest of the section. 

\subsection{Proof of tightness}
The argument relies on a tightness criterion for continuous random functions. Fix $k\in\mathbb{N}$ and $[a,b]\subset\mathbb{R}$. Recall that the modulus of continuity of multiple functions is defined in \eqref{def:modu-2}. Define the set
\begin{equation*}
U_{[a,b],k}\big(\rho,r\big)\coloneqq  \Big\{\mathcal{J}\in C([1,k]_{\mathbb{Z}}\times [a,b],\mathbb{R})\,\big|\,  \omega_{[a,b],k}\big(\mathcal{J} ,r\big)\leq \rho\Big\}.
\end{equation*}
The following tightness criterion is an immediate generalization of \cite[Theorem 7.3]{Bilconv}.

\begin{lemma}\label{lem:tightcriterion}
A sequence $\mathbb{P}_N$ of probability measures on $C([1,k]_{\mathbb{Z}}\times [a,b],\mathbb{R})$ is tight if the following two conditions are met. 
\begin{enumerate}[label=(\roman*)]
\item There exists $t_0\in [a,b]$ such that the one-point distribution of $\mathcal{J}_i(t_0)$ is tight for all $i\in [1,k]_{\mathbb{Z}}$
\item For each $\rho>0$ and $\eta>0$, there exist $r_0>0$ and an integer $N_0$ such that for all $N\geq N_0$, it holds that
\begin{equation*}
\mathbb{P}_N \big(U_{[a,b],k}(\rho,r_0)\big)\geq 1-\eta.
\end{equation*}
\end{enumerate}
\end{lemma}

\vspace{0.2cm}

We apply this tightness criterion to $\cL^N$. More precisely, we seek to prove that for any $k\in \mathbb{N}$ and $L\geq 1 $, the restriction of $\cL^N$ to  $[1,k]_{\mathbb{Z}}\times [-L,L]$ (i.e. $\{\cL^N_i(t)| 1\leq i\leq k, t\in[-L, L]\}$) is tight as $N$ varies. One-point tightness follows from Theorem~\ref{thm:1}, hence (i) holds. It remains to control the modulus of continuity, i.e. to verify (ii).\\[-0.3cm]

Starting from now, we fix $k\in \mathbb{N}$ and $L\geq 1 $.  Write $\mathbb{P}_N$ for the law of $\mathcal{L}^N$ and denote event $$\mathsf{U}^N(\rho,r)\coloneqq\left\{\mathcal{L}^N\big|_{[1,k]_{\mathbb{Z}}\times [-L,L]}\in  {U}_{[-L,L],k}(\rho,r)\right\}.$$  
We aim to verify that, for all $\rho, \eta>0$,  there exist $r_0$ and $N_0$ such that for $N\geq N_0$, it holds that
	\begin{equation}\label{eq:P_N}
	\mathbb{P}_N\big(\mathsf{U}^N(\rho,r_0) \big)\geq 1-\eta.
	\end{equation}
	
For $R>0$, we define the event
	\begin{equation*}
	\mathsf{S}^N(R)=\cap_{i=1}^k \big\{ \mathcal{L}^N_i(\pm L)\in [0,R] \big\}.
	\end{equation*}
Let $Z_N$ be the shorthand of the normalizing constant $Z^{1,k,(-L,L),\vec{x},\vec{y},0,g}$, with ${x}_i=\cL^N_i(-L)$, ${y}_i=\cL^N_i(L)$ and $g=\cL^N_{k+1}\big|_{[-L,L]}$. See Definition~\ref{def:Bessel bridge LE} for the definition of $Z^{1,k,(-L,L),\vec{x},\vec{y},0,g}$.
\begin{lemma}\label{lem:u(-T,T)}
Fix $\alpha\geq 0$, $k\in\mathbb{N}$ and $\rho,\eta,\delta,R, L>0$. There exists $r_0$ depending on $\alpha,k,\rho,\eta,\delta,R$ and $L$  such that 
	\begin{equation*}
\mathbbm{1}\{Z_N\geq \delta\}\cdot\mathbbm{1}_{\mathsf{S}^N(R)}\cdot\mathbb{E}_N\big[\mathbbm{1}_{\mathsf{U}^N(\rho,r_0)}\big\vert \Fext\left([1,k]_{\mathbb{Z}}\times(-L,L)\right)\big] \geq (1-\eta/2)\cdot  \mathbbm{1}\{Z_N\geq \delta\}\cdot\mathbbm{1}_{\mathsf{S}^N(R)}.
	\end{equation*}
	\end{lemma} 
	
\begin{proof}
The proof is a combination of Lemma \ref{lem:Besselmodulus} and the Gibbs property Proposition~\ref{pro:GibbsforL}. Let $\vec{x},\vec{y}\in\mathbb{R}^k$ be vectors with $x_i,y_i\in [0,R]$.  Lemma \ref{lem:Besselmodulus} implies for $r_0$ small enough, it holds that
	\begin{align*}
	\mathbb{P}^{1,k,(-L,L),\vec{x},\vec{y}}_{\free}(\mathsf{U}^N(\rho,r_0))\geq 1-\delta\eta/2.
	\end{align*}
 The Radon-Nykodym relation in \eqref{eqn:RN} then implies
 \begin{equation*}
\mathbbm{1}\{Z_N\geq \delta\}\cdot\mathbbm{1}_{\mathsf{S}^N(R)}\cdot\mathbb{E}_N\big[1-\mathbbm{1}_{\mathsf{U}^N(\rho,r_0)}\big\vert \Fext\left([1,k]_{\mathbb{Z}}\times(-L,L)\right)\big] \leq 2^{-1}\eta \cdot  \mathbbm{1}\{Z_N\geq \delta\}\cdot\mathbbm{1}_{\mathsf{S}^N(R)}.
	\end{equation*}
The desired result then follows.  	\end{proof}

\begin{proof}[Proof of Theorem~\ref{thm:main}(i)] To prove \eqref{eq:P_N}, it is enough to verify the following statement. For any $\rho, \eta >0$, there exists $\delta, R, r_0>0$ and $N_0$ such that for all $N\geq N_0$, we have
	\begin{equation}\label{eq:W(rho,eta)}
	\mathbb{P}_N\Big(\mathsf{U}^N(\rho,r_0) \cap\{Z_N\geq \delta\}\cap \mathsf{S}^N(R) \Big)> 1-\eta.
	\end{equation}

Observe that the events $\{Z_N\geq \delta\}$ and $\mathsf{S}^N(R) $ are $\Fext([1,k]_{\mathbb{Z}}\times(-L,L))$-measurable. We can rewrite the left-hand side of (\ref{eq:W(rho,eta)}) as
\begin{equation}\label{eq.noleft}
	\mathbb{E}_N\Big[\mathbbm{1}\{Z_N\geq \delta\}\cdot\mathbbm{1}_{\mathsf{S}^N(R)}\cdot\mathbb{E}_N\big[\mathbbm{1}_{\mathsf{U}^N(\rho,r_0)}\big\vert \Fext\left([1,k]_{\mathbb{Z}}\times(-L,L)\right)\big]\Big].
	\end{equation}
By choosing $r_0$ to be the constant in Lemma~\ref{lem:u(-T,T)}. From Lemma \ref{lem:u(-T,T)},  
	\begin{equation*}
	\eqref{eq.noleft} \geq (1-\eta/2)\mathbb{P}_N[ \{Z_N\geq \delta\}\cap   \mathsf{S}^N(R) ].
	\end{equation*}
	
Let $\delta=\delta_1(\alpha, k,L,\eta/4)$ and $N_6=N_6(\alpha,k,L,\eta/4)$ be the constants in Proposition \ref{pro:Z}. By Proposition \ref{pro:Z}, for $N\geq N_6$, it holds that $\mathbb{P}_N(\{Z_N<\delta\})\leq \eta/4$. Let $R=R_1(\alpha,\eta/4,k,2L)$ and $N_1=N_1(\alpha,\eta/4,k,L)$ be the constants in Proposition \ref{pro:sup}. Proposition \ref{pro:sup} implies that for $N\geq N_1$, $\mathbb{P}_N(\mathsf{S}^N(R))\geq 1-\eta/4.$ This implies that
	\begin{equation}\label{eq:ZSM}
	\mathbb{P}_N(\{Z_N\geq \delta\}\cap \mathsf{S}^N(R))\geq 1-\eta/2.
	\end{equation}
We conclude that
\[
\eqref{eq.noleft} \geq (1-\eta/2)^2 > 1-\eta,
\]
which completes the proof of the inequality in \eqref{eq:W(rho,eta)}. The desired tightness then follows.
\end{proof}

\subsection{Proof of limiting squared Bessel Gibbs property} In this section we seek to prove Theorem \ref{thm:main}(ii), i.e., any subsequential limit of $\mathcal{L}^N$ enjoys the squared $\alpha$-Bessel Gibbs property. Suppose $\mathcal{L}^{\infty}$ is a subsequential limit of $\mathcal{L}^N$. We abuse the notation and assume that $\mathcal{L}^N$ converges weakly to $\mathcal{L}^{\infty}$. We start by showing that $\mathcal{L}^\infty$ is strictly ordered with probability $1$.
\begin{lemma}\label{lem:7.3}
Fix $\alpha\geq 0$, $k\in \mathbb{N}$ and $L>0$. For any $\varepsilon>0$, there exists $\rho=\rho(\alpha,k,L,\varepsilon)>0$ such that 
\begin{align*}
\mathbb{P}\left(\inf_{t\in [-L,L]} (\mathcal{L}^\infty_{k+1}(t)-\mathcal{L}^\infty_{k}(t))<\rho \right)\leq  \varepsilon.
\end{align*}
\end{lemma}
\begin{proof}
Fix $\varepsilon>0$. Let $R_1=R_1(\alpha,8^{-1}\varepsilon,k+1,2L)$ be given in Proposition~\ref{pro:sup} and $r_2=r_2(\alpha,8^{-1}\varepsilon)$ be given in Lemma~\ref{lem:oneptL}. Set $R\coloneqq \max(R_1,r_2^{-1})$ and
\begin{equation*}
	\mathsf{S}^N \coloneqq \cap_{i=1}^{k+1} \big\{ \mathcal{L}^N_{k+1}(\pm L)\in [R^{-1},R] \big\}.
\end{equation*}
Denote by $\mathbb{P}_N$ the law of $\mathcal{L}^N$. From Proposition~\ref{pro:sup} and Lemma~\ref{lem:oneptL}, we have $\mathbb{P}_N(\mathsf{S}^N)\geq 1-2^{-1}\varepsilon$ for $N$ large enough. 

Let $Z_N$ be the shorthand of the normalizing constant $Z^{1,k+1,(-L,L),\vec{x},\vec{y},0,g}$, with ${x}_i=\cL^N_i(-L)$, ${y}_i=\cL^N_i(L)$ and $g=\cL^N_{k+2}\big|_{[-L,L]}$. Let $\delta_1=\delta_1(\alpha,k+1,L,4^{-1}\varepsilon)>0$ be the constant given in Proposition~\ref{pro:Z}. From Proposition~\ref{pro:Z}, we have $\mathbb{P}(Z_N<\delta_1)\leq 4^{-1}\varepsilon$ for $N$ large enough.

Set $\eta=4^{-1}\delta_1^{-1}\varepsilon$ and $\rho=\rho(\alpha,R,\eta,2L)$ be the constant given in Lemma~\ref{lem:twoBessel}. Denote $$\mathsf{E}^N\coloneqq \left\{\inf_{t\in [-L,L]} (\mathcal{L}^N_{k+1}(t)-\mathcal{L}^N_{k}(t))\in (-\rho, \rho)  \right\}.$$
We compute
\begin{align*}
\mathbb{P}_N(\mathsf{S}^N\cap\{Z_N\geq \delta_1\}\cap \mathsf{E}^N)=\mathbb{E}_N\left[\mathbbm{1}_{\mathsf{S}^N}\cdot\mathbbm{1}\{Z_N\geq \delta_1\} \mathbb{E}_N\left[ \mathbbm{1}_{\mathsf{E}_N} \,|\,\Fext([1,k+1]_{\mathbb{Z}} \times (-L,L)) \right] \right].
\end{align*}
From the Gibbs property, we have
\begin{align*}
 \mathbb{E}_N\left[ \mathbbm{1}_{\mathsf{E}_N} \,|\,\Fext([1,k+1]_{\mathbb{Z}} \times (-L,L)) \right]=&\mathbb{P}^{1,k+1,(-L,L),\vec{x},\vec{y},0,g}(\mathsf{E}_N)  \leq   Z_N^{-1}\mathbb{P}_{\free}^{1,k+1,(-L,L),\vec{x},\vec{y} }(\mathsf{E}_N).
\end{align*}
Here ${x}_i=\cL^N_i(-L)$, ${y}_i=\cL^N_i(L)$ and $g=\mathcal{L}^{N}_{k+2}\big|_{[-L,L]}$. From Lemma~\ref{lem:twoBessel}, we have
\begin{align*}
\mathbbm{1}_{\mathsf{S}^N}\cdot\mathbbm{1}\{Z_N\geq \delta_1\} \cdot Z_N^{-1}\mathbb{P}_{\free}^{1,k+1,(-L,L),\vec{x},\vec{y} }(\mathsf{E}_N)\leq \delta_1^{-1}\eta=4^{-1}\varepsilon.
\end{align*}	
Therefore,
\begin{align*}
\mathbb{P}_N(\mathsf{S}^N\cap\{Z_N\geq \delta_1\}\cap \mathsf{E}^N)\leq 4^{-1}\varepsilon.
\end{align*}
Combining the above, we conclude that for $N$ large enough,
\begin{align*}
\mathbb{P}_N(\mathsf{E}^N)\leq \mathbb{P}_N(\mathsf{S}^N\cap\{Z_N\geq \delta_1\}\cap \mathsf{E}^N)+\mathbb{P}((\mathsf{S}^N)^{\textup{c}})+\mathbb{P}_N(Z_N<\delta_1)\leq \varepsilon.
\end{align*}
Then the conclusion follows by taking $N$ to infinity.
\end{proof}

Fix an index $i\in\mathbb{N}$ and an interval $(a,b)\in \mathbb{R}$. We will show that the law of $\mathcal{L}^{\infty}$ is unchanged when one resamples the trajectory of $\mathcal{L}_i^{\infty}$ between $(a,b)$ according to a squared Bessel bridge which avoids $\mathcal{L}_{i-1}^{\infty}$ and $\mathcal{L}_{i+1}^{\infty}$. Note that this is equivalent to the squared Bessel Gibbs property for the $i$-th curve restricted to the interval $(a,b)$. The same argument could be easily generalized to take care of multiple curves resampling so we choose to illustrate the argument with the single curve resampling. See Figure \ref{figure:BGP} for an illustration.

Note that $C(\mathbb{N} \times \mathbb{R},\mathbb{R}) $ (with the topology given at the end of Section~\ref{sec:Introduction}) is separable due to the Stone–Weierstrass theorem. Hence the Skorohod representation theorem \cite[Theorem 6.7]{Bilconv} applies. There exists a probability space $(\Omega,\mathcal{B},\mathbb{P})$ on which $\mathcal{L}^N$ for $N\in\mathbb{N}\cup\{\infty\}$ are defined and almost surely $\mathcal{L}^N(\omega)\to\mathcal{L}^\infty(\omega)$ in the topology of $C(\mathbb{N} \times \mathbb{R},\mathbb{R})$. 

Now we take countable, independent copies of the squared Bessel bridges constructed in Proposition~\ref{lem:BEcoupling}. That is, for all $\ell\in\mathbb{N}$, we have squared Bessel bridges $\mathcal{Q}_\ell(x,y)(t)$ defined on $[a,b]$ with entrance and exit data $(x,y)$. We note that because $\mathcal{Q}_\ell(x,y)(t)$ depends continuously on $x,y$, it is also measurable in $x,y$. We define the $\ell$-th candidate of the resampling trajectory. For $N\in\mathbb{N}\cup \{\infty\}$, define
\begin{align*}
\mathcal{L}^{N,\ell}_i(t)\coloneqq  \left\{ \begin{array}{cc}
\mathcal{Q}_{\ell}(\mathcal{L}^{N}_i(a),\mathcal{L}^{N}_i(b))(t), & t\in [a,b],\\[0.1cm]
\mathcal{L}^{N}_i(t),\ & t\in (-\infty,a)\cup (b,\infty).
\end{array}  \right.	
\end{align*}

For $N\in\mathbb{N}\cup\{\infty\}$, we \textbf{accept} the candidate resampling $\mathcal{L}^{N,\ell}_i$ if it does not intersect $\mathcal{L}^{N}_{i-1}$ or $\mathcal{L}^{N}_{i+1}$ on $[a,b]$. For $N\in \mathbb{N}\cup\{\infty\}$, define $\ell(N)$ to be the minimum value of $\ell$ of which we accept $\mathcal{L}^{N,\ell}_i$. That is,
\begin{align*}
\ell(N)\coloneqq  \inf\{\ell\in\mathbb{N}\,|\, \mathcal{L}^{N}_{i-1} <\mathcal{L}^{N,\ell}_{i} <\mathcal{L}^{N}_{i+1} \ \textup{on}\ [a,b]\}.
\end{align*} Write $\mathcal{L}^{N,\textup{re}}$ for the line ensemble with the $i$-th curve replaced by $ \mathcal{L}^{N,\ell(N)}_i$. The line ensemble $\mathcal{L}^{N,\textup{re}}$ satisfies the squared Bessel Gibbs property on $\{i\}\times [a,b]$.

\begin{figure}
\begin{tikzcd}  
\mathcal{L}^{N,\textup{re}} \arrow[rr,equal ,"(d)","(1)"'] \arrow[dd,"\textup{a.s.}","(2)'"']& & \mathcal{L}^{N } \arrow[dd,"\textup{a.s.}","(2)"']  \\
 & &\\
\mathcal{L}^{\infty,\textup{re}} \arrow[rr,equal ,"(d)","(1)'"'] & & \mathcal{L}^{\infty }
\end{tikzcd}
\caption{(1) is equivalent to the squared Gibbs property when resampling a single curve. The {\bf goal} is to prove (1)', which implies the squared Bessel Gibbs property for the subsequential limit line ensemble.   (1)' follows from the convergence in (2) and (2)'. (2) follows from the Skorohod representation theorem and (2)' is proved in Lemma \ref{lem:ellconvergence}.}\label{figure:BGP}
\end{figure}

From Lemma~\ref{lem:7.3}, $\{\mathcal{L}^\infty_{i-1}<\mathcal{L}^\infty_{i}<\mathcal{L}^\infty_{i+1}\}$ holds almost surely. Hence almost surely $\ell(\infty)$ is finite. Suppose that $\ell(N)$ converges to $\ell(\infty)$  almost surely. Then $\mathcal{L}^{N,\textup{re}}$ converges to $\mathcal{L}^{\infty,\textup{re}}$ in $C( \mathbb{N} \times \mathbb{R},\mathbb{R})$ almost surely. Hence we have $\mathcal{L}^{N,\textup{re}}$ converges weakly to $\mathcal{L}^{\infty,\textup{re}}$ as $\mathbb{N}\times\mathbb{R}$-indexed line ensembles. See Definition \ref{def:weakconvergence} As a consequence, $\mathcal{L}^{\infty,\textup{re}}$ has the same distribution as $\mathcal{L}^{\infty}$.

\begin{lemma}\label{lem:ellconvergence}
Almost surely $\ell(N)$ converges to $\ell(\infty)$.
\end{lemma}
\begin{proof}
Let $\mathsf{E}$ be the event such that the following conditions hold \\[-0.3cm]
\begin{enumerate}
\item $\ell(\infty)<\infty$\\[-0.3cm]
\item $\displaystyle\inf_{t\in [a,b]} \mathcal{L}^{\infty,\ell}_i(t) -\mathcal{L}^{\infty}_{i-1}(t)  \neq 0$ for all $\ell\in\mathbb{N}$\\[0.05cm]
\item $\displaystyle\inf_{t\in [a,b]} \mathcal{L}^{\infty}_{i+1}(t) -  \mathcal{L}^{\infty,\ell}_i(t) \neq 0$ for all $\ell\in\mathbb{N}$\\[0.05cm]
\item $\mathcal{L}^N$ converges to $\mathcal{L}^\infty$ in $C(\mathbb{N}\times\mathbb{R},\mathbb{R})$.\\[-0.3cm]
\end{enumerate}
From the above discussion, conditions (1) and (4) hold with probability $1$. Condition (2) requires a squared Bessel bridge not to be ``tangent" to $\mathcal{L}^\infty_{i-1}$. From Lemma~\ref{lem:not touching} and the independence between $\mathcal{Q}_\ell(x,y)$ and $\mathcal{L}^\infty$, (2) holds with probability $1$. The same argument holds for (3). In short, $\mathsf{E}$ has probability $1$. \\[-0.3cm]

We will show that when $\mathsf{E}$ occurs, $\ell(N)$ converges to $\ell(\infty)$. 
From now on we fix $\omega\in \mathsf{E}$ and the constants below may depend on $\omega$. By the definition of $\ell(\infty)$, for all $t\in [a,b]$,
\begin{align*}
  \mathcal{L}^{\infty}_{i-1}(t)<\mathcal{L}^{\infty,\ell(\infty)}_{i}(t) <\mathcal{L}^{\infty}_{i+1}(t) .
\end{align*} 
Because of \eqref{equ:Becoupling} and the convergence of $(\mathcal{L}^N_i(a),\mathcal{L}^N_i(b))$  to $(\mathcal{L}^\infty_i(a),\mathcal{L}^\infty_i(b))$, $\mathcal{L}^{N,\ell(\infty)}_i $ converges to $\mathcal{L}^{\infty,\ell(\infty)}_i $ uniformly on $[a,b]$. Together with the fact that $\mathcal{L}^{N}_{i\pm 1}$ converges to $\mathcal{L}^{\infty}_{i\pm 1}$ uniformly on $[a,b]$, we have for $N$ large enough, $\mathcal{L}^{N}_{i-1}(t)<\mathcal{L}^{N,\ell(\infty)}_{i}(t) <\mathcal{L}^{N}_{i+1}(t)$ for all $t\in [a,b]$. Therefore,
\begin{align*}
\limsup_{N\to\infty}\ell(N)\leq \ell(\infty). 
\end{align*}

On the other hand, for all $1\leq\ell<\ell(\infty)$, we have either $\inf_{t\in [a,b]} \mathcal{L}^{\infty,\ell}_i(t) -\mathcal{L}^{\infty}_{i-1}(t)  < 0$ or $\inf_{t\in [a,b]} \mathcal{L}^{\infty}_{i+1}(t) -  \mathcal{L}^{\infty,\ell}_i(t) < 0.$ We assume that $\inf_{t\in [a,b]} \mathcal{L}^{\infty,\ell}_i(t) -\mathcal{L}^{\infty}_{i-1}(t)  < 0$ occurs . Then for $N$ large enough, we have
\begin{align*}
\displaystyle\inf_{t\in [a,b]} \mathcal{L}^{N,\ell}_i(t) -\mathcal{L}^{N}_{i-1}(t)<0
\end{align*} 
As a consequence,
\begin{align*}
\liminf_{N\to\infty}\ell(N)\geq \ell(\infty).
\end{align*}
Hence $\ell(N)$ converges to $\ell(\infty)$  and the proof is finished.
\end{proof}

\begin{appendix}
\section{Results about (squared) Bessel Processes}\label{sec:BE}
In this section we record some basic properties of the transition probability of (squared) Bessel processes. These results serve as inputs for the stochastic monotoncity Proposition~\ref{lem:monotonicity} and uniform lower bounds in Section~\ref{sec:uniformbounds}.
\subsection{Squared Bessel Process}

Let $I_\alpha(z)$ be the modified Bessel function with index $\alpha$. It solves the modified Bessel equation \cite[(9.6.1)]{AS}
\begin{align}\label{equ:IODE}
z^2 I''_{\alpha}(z)+zI'_{\alpha}(z)-(z^2+\alpha^2)I_{\alpha}(z)=0.
\end{align} 
Recall that $h_\alpha(z)\coloneqq z^{-\alpha}I_\alpha(z)$.  From \eqref{equ:IODE}, $h_{\alpha}(z)$ solves the equation
\begin{align}\label{equ:hODE}
zh''_{\alpha}(z)+(2\alpha+1)h'_\alpha(z)-zh_\alpha(z)=0.
\end{align} 

\begin{lemma}\label{lem:h-convex}
Fix $\alpha\geq 0$. For any $z\geq 0$, it holds that
\begin{align}\label{h:convex}
 \frac{h_\alpha''(z)}{h_\alpha(z)}+z^{-1}\frac{h_\alpha'(z)}{h_\alpha(z)}-\left(\frac{h_\alpha'(z)}{h_\alpha(z)} \right)^2> 0.
 \end{align}
\end{lemma}

\begin{remark}
In terms of the modified Bessel function, when $z>0$, \eqref{h:convex} is equivalent to 
\begin{align}\label{equ:convexI}
  1+ \frac{\alpha^2}{z^2} - \left(\frac{I'_\alpha(z)}{I_\alpha(z)}\right)^2>0.
\end{align} 
\eqref{equ:convexI} was proved for $\alpha>0$ and $z>0$ by Gronwall in \cite{Gro} motivated by a problem in wave mechanics. 
\end{remark}

\begin{proof}
Let $H_\alpha(z)\coloneqq z^{-1}h'_{\alpha}(z)/h_{\alpha}(z)$. By \eqref{equ:hODE}, \eqref{h:convex} is equivalent to
\begin{align*}
-z^2H_\alpha^2(z)-2\alpha H_\alpha(z)  +1>0. 
\end{align*}
Note that $-z^2 u^2-2\alpha u  +1=0$ has a positive solution $u=G_\alpha(z)\coloneqq \frac{1}{\alpha+\sqrt{\alpha^2+z^2}}$ and a negative solution. In view of \eqref{equ:hexpansion}, $H_\alpha(z)>0$ for all $z\geq 0$. Therefore, it suffices to show that $H_\alpha(z)< G_\alpha(z)$. From \eqref{equ:hODE}, we derive
\begin{align*}
zH_\alpha'(z)+z^2 H_\alpha^2(z)+2(\alpha+1)H_\alpha(z)-1=0.
\end{align*}
Through a direct calculation,
\begin{align*}
zG_\alpha'(z)+z^2 G_\alpha^2(z)+2(\alpha+1)G_\alpha(z)-1=\frac{1}{\sqrt{z^2+\alpha^2}}>0.
\end{align*}
If $\alpha>0$, then $H_\alpha(0)=1/(2\alpha+2)<1/(2\alpha)=G(0)$. Through ODE comparison, $H_\alpha(z)<G(z)$ for all $z\geq 0$. When $\alpha=0$, the argument is similar using $\lim_{z\to 0^+}G_0(z)=\infty$. The proof is finished.
\end{proof}

\begin{corollary}\label{cor:q-convex}
Fix $\alpha\geq 0$. For any $t>0$ and $x,y\geq 0$, it holds that
\begin{equation}\label{q:convex}
 \frac{\partial^2}{\partial x\partial y}\log q_t(x,y)>  0. 
\end{equation}
\end{corollary}
\begin{proof}
By the scaling property \eqref{q_scaling}, it suffices to consider the case $t=1$. By a direct computation,
\begin{align}\label{equ:cal}
4\frac{\partial^2}{\partial x\partial y}\log q_1(x,y)= \frac{h_\alpha''(z)}{h_\alpha(z)}+z^{-1}\frac{h_\alpha'(z)}{h_\alpha(z)}-\left(\frac{h_\alpha'(z)}{h_\alpha(z)} \right)^2, 
\end{align}
where $z=\sqrt{xy}$. Then \eqref{q:convex} follows Lemma~\ref{lem:h-convex}.
\end{proof}

The following Corollary concerns the upper bound and lower bound of $\det\big(q_t(x_i,y_j)\big)_{1\leq i,j\leq 2}$.
\begin{corollary}\label{cor:qqqq}
Fix $\alpha\geq 0$. For any $L>0$, there exists a constant $C=C(\alpha, L)>0$ such that the following statement holds. For any $x_2\geq x_1\geq 0$, $y_2\geq y_1\geq 0$ and $t>0$ that satisfy $x_2y_2\leq L t^2$, we have
\begin{align*}
 &\det\big(q_t(x_i,y_j)\big)_{1\leq i,j\leq 2} \geq C^{-1}  t^{-2}  q_t(x_1,y_2)q_t(x_2,y_1) {(x_2-x_1)(y_2-y_1)}  
   ,\\[0.1cm]
&\det\big(q_t(x_i,y_j)\big)_{1\leq i,j\leq 2}  \leq C    t^{-2}  q_t(x_1,y_2)q_t(x_2,y_1) {(x_2-x_1)(y_2-y_1)}.  
\end{align*}
\end{corollary}

\begin{proof}
As both hands of the above inequalities are continuous, it is sufficient to consider the case $x_2>x_1>0$ and $y_2>y_1>0$. Moreover, we may assume $t=1$ due to the scaling invariance \eqref{q_scaling}.\\[-0.2cm]

Note that 
\begin{align*}
\frac{\det\big(q_1(x_i,y_j)\big)}{q_1(x_1,y_2)q_1(x_2,y_1)}= \frac{ q_1(x_1,y_1)q_1(x_2,y_2)}{ q_1(x_1,y_2)q_1(x_2,y_1)}-1.
\end{align*} 
It suffices to show that
\begin{align}\label{eq:detq1}
\frac{1}{(x_2-x_1)(y_2-y_1)} \left(\frac{ q_1(x_1,y_1)q_1(x_2,y_2)}{ q_1(x_1,y_2)q_1(x_2,y_1)}-1\right)
\end{align}
is uniformly bounded from above and from below for $x_1, x_2, y_1, y_2$ as in the statement.\\[-0.2cm]

Since
\begin{align*}
\log \left(\frac{ q_1(x_1,y_1)q_1(x_2,y_2)}{ q_1(x_1,y_2)q_1(x_2,y_1)}\right)= \log  q_1(x_1,y_1)+\log q_1(x_2,y_2)-\log q_1(x_1,y_2)-\log q_1(x_2,y_1)
\end{align*}
is a double difference term. Apply twice the mean value theorem and \eqref{equ:cal}, we have
\begin{align*}
 &\frac{4}{(x_2-x_1)(y_2-y_1)}\log \left(\frac{ q_1(x_1,y_1)q_1(x_2,y_2)}{ q_1(x_1,y_2)q_1(x_2,y_1)}\right) =\frac{h_\alpha''(z)}{h_\alpha(z)}+z^{-1}\frac{h_\alpha'(z)}{h_\alpha(z)}-\left(\frac{h_\alpha'(z)}{h_\alpha(z)} \right)^2.
\end{align*} 
Here $z=\sqrt{uv}$ for some $u\in (x_1,x_2)$ and $v\in (y_1,y_2)$. Under the assumption $x_2y_2\leq L$, $0\leq z\leq L$. From \eqref{h:convex}, the above is bounded from above by $C$ and from below by $C^{-1}$ for some $C$ depending only on $L$. This implies
\begin{align}\label{eq:detq2}
\exp(C^{-1}(x_2-x_1)(y_2-y_1)) \leq  \frac{ q_1(x_1,y_1)q_1(x_2,y_2)}{ q_1(x_1,y_2)q_1(x_2,y_1)}\leq \exp(C(x_2-x_1)(y_2-y_1))
\end{align}
Notice that $0\leq (x_2-x_1)(y_2-y_1)\leq x_2y_2\leq L$. Subtracting 1 in \eqref{eq:detq2} and dividing both sides by $(x_2-x_1)(y_2-y_1)$, then the desired assertion \eqref{eq:detq1} follows from the mean value theorem. 
\end{proof}

\begin{lemma}\label{lem:2sample}
Fix $\alpha\geq 0$. For any $L>0$, there exists a constant $C=C(\alpha, L)$ such that the following statement holds. Given $\tau,T>0$, $x_2\geq  x_1> 0,$ $y_2> y_1> 0$ and $z_*>0$. Assume that
\begin{align*}
x_2z_*\leq L\tau^2,\ y_2z_*\leq LT^2\ \textup{and that}\ x_2y_2\leq L(\tau+ T)^2.
\end{align*} 
Let $\vec{x}=(x_1,x_2) $ and $\vec{y}=(y_1,y_2) $. Let $(\mathcal{J}_1(t),\mathcal{J}_2(t))$ be distributed according to $\mathbb{P}^{1,2,-\tau,T,\vec{x},\vec{y},0,\infty}$. \\[-0.3cm]

Then the joint density of $(\mathcal{J}_1(0),\mathcal{J}_2(0))=(z_1,z_2)$ is bounded from below by
\begin{align*}
C^{-1} \left(\tau^{-1}+T^{-1} \right)^2 \, \frac{q_\tau(x_1,z_2)q_\tau(x_2,z_1)q_T(z_1,y_2)q_T(z_2,y_1)}{q_{\tau+T}(x_1,y_2)q_{\tau+T}(x_2,y_1) }\cdot(z_2-z_1)^2\cdot\mathbbm{1}(0<z_1<z_2\leq z_*).
\end{align*}
\end{lemma}

\begin{proof}
Through taking limits, it is enough to consider the case $x_2>x_1>0$ and $y_2>y_1>0$. By the Karlin-McGregor formula \cite{KM}, the joint density of $(\mathcal{J}_1(0),\mathcal{J}_2(0))=(z_1,z_2)$ is given by
\begin{align*}
\frac{\det(q_\tau(x_i,z_j))_{1\leq i,j\leq 2}\det(q_T(z_i,y_j))_{1\leq i,j\leq 2}}{\det(q_{\tau+T}(x_i,y_j))_{1\leq i,j\leq 2}}\cdot\mathbbm{1}(0<z_1<z_2) .
\end{align*}
When $z_2\leq z_*$, we have $x_2z_2\leq L\tau^2$ and $y_2z_2\leq LT^2$. 
Then the assertion follows by applying Corollary~\ref{cor:qqqq} three times for the three determinant terms respectively.
\end{proof}

\subsection{Bessel Process}
We consider in this section $p_t(x,y)$, the transition probability of Bessel processes. See \eqref{def:pF} for its explicit form. From \eqref{q:convex}, we have 
\begin{equation}\label{p:convex}
 \frac{\partial^2}{\partial x\partial y}\log p_t(x,y)\geq  0. 
\end{equation}

\begin{lemma}\label{lem:ystar}
For any $x\in [0,\infty)$, $p_1(x,y)$ is strictly log-concave in $y\in (0,\infty)$.
\end{lemma}
\begin{proof}
For $x=0$, strict log-concavity follows directly from \eqref{def:pF}. For $x>0$, from \eqref{def:pF}, it suffices to show the strict strictly log-concavity of $z^{2\alpha+1}h_\alpha(z)$ for $z\in (0,\infty)$. This is proved in \cite[Theorem 4]{NNW} and we present the argument below here for the reader's convenience. Using $h_\alpha(z)=z^{-\alpha}I_\alpha (z)$, we compute
\begin{align*}
z^2 h^2_\alpha\cdot \big(\log(z^{2\alpha+1}h_\alpha   )\big)''=  -(2\alpha+1)h_{\alpha}^2+z^2h_\alpha h''_\alpha-z^2(h'_\alpha)^2 .
\end{align*}
From \eqref{equ:hexpansion}, we can check that $h'_\alpha(z)=zh_{\alpha+1}(z)$. Then the above equals
\begin{align*}
z^2 h^2_\alpha\cdot \big(\log( z^{2\alpha+1}h_\alpha  )\big)''=   -(2\alpha+1)h_{\alpha}^2+z^2h_\alpha h_{\alpha+1}+z^4h_\alpha h_{\alpha+2}-z^4  h_{\alpha+1}^2 .
\end{align*}
Combining \eqref{equ:hexpansion} and the Chu–Vandermonde identity \cite[page 45]{Wolfram}, we have for any $\alpha,\beta\geq 0$, $h_{\alpha}(z)h_{\beta}(z)=\sum_{n=0}^\infty c_{\alpha,\beta,n}z^{2n} $ with 
\begin{align*}
c_{\alpha,\beta,n}=\frac{2^{-\alpha-\beta-2n} \Gamma(2n+\alpha+\beta+1)}{n!\Gamma(n+\alpha+\beta+1)\Gamma(n+\alpha +1)\Gamma(n +\beta+1)}.
\end{align*}
A straight forward calculation shows that
\begin{align*}
z^2 h_\alpha \big(\log(z^{2\alpha+1}h_\alpha  )\big)''= -(2\alpha+1)c_{\alpha,\alpha,0}- (2\alpha+1) \sum_{n=1}^\infty\frac{n+2\alpha-1 }{2n+2\alpha-1}\cdot c_{\alpha,\alpha,n} z^{2n}.  
\end{align*}
Because the coefficients in the above expansion are negative, the assertion follows.
 
\end{proof}

\begin{corollary}\label{cor:ystar}
For any $x\in [0,\infty)$, there exists a unique maximum point $y^*(x)$ of $p_1(x,\cdot)$. $\lim_{x\to\infty} |y^*(x)-x|=0$. $y^*(x)$ is a smooth function in $x$.  Moreover, $\sup_{x\in [0,\infty)}|y^*(x)-x|\leq L_0 $ for some $L_0>0$. 
\end{corollary}
\begin{proof}
The existence and uniqueness of the maximum point $p_1(x,\cdot)$ follows directly from Lemma~\ref{lem:ystar}. From $I_\alpha(z)= {(2\pi z)^{-1/2}}{e^z}\left(1+O(z^{-1})\right)$ when $z$ goes to infinity \cite[(9.7.1)]{AS}, it holds that
\begin{align}\label{equ:q_asymp}
p_1(x,y)=(2\pi)^{-1/2}(y/x)^{\alpha+1/2}e^{-(y-x)^2/2}\left(1+O(x^{-1}y^{-1})\right).
\end{align}
This implies $\lim_{x\to\infty} |y^*(x)-x|=0$. 
Clearly $y^*(x)$ is the unique solution to $\partial p_1(x,y)/\partial y=0$. By the inverse function theorem, $y^*(x)$ is a smooth function in $x$. Then the boundedness of $|y^*(x)-x|$ follows.
\end{proof}

\begin{lemma}\label{lem:BBsum_1}
There exist constants  ${C}_0$ and $C_1$ such that 
\begin{align}\label{q_sup}
\sup_{x,y\in [0,\infty)}p_1(x,y)&\leq C_0,
\end{align}
and 
\begin{align}\label{q_sum}
\sup_{x\in [0,\infty), M\geq 1} M^{-1}\sum_{j=0}^\infty p_1(x,jM^{-1})&\leq C_1.
\end{align}
\end{lemma}

\begin{proof}
We start with prove \eqref{q_sup}. From Corollary~\ref{cor:ystar}, $\sup_{y\in [0,\infty)}p_1(x,y)=p(x,y^*(x))$. Because $\lim_{x\to\infty} |y^*(x)-x|=0$, we have from \eqref{equ:q_asymp} that $\lim_{x\to\infty}p_1(x,y^*(x))=(2\pi)^{-1/2}$. Then \eqref{q_sup} follows from the continuity of $p(x,y^*(x))$.\\[-0.3cm]

Next, we turn to \eqref{q_sum}. Let $j_0=j_0(x,M)\coloneqq \lfloor My_*(x)  \rfloor$. From Lemma~\ref{lem:ystar}, $p_1(x,y)$ is non-decreasing for $y\in [0,y_*(x)]$ and is non-increasing for $y\in [y_*(x),\infty)$. By integral comparison,
\begin{align*}
\sum_{0\leq j\leq j_0-1 } M^{-1} p_1(x,jM^{-1})\leq& \int_0^{j_0 }p_1(x,y)\, dy,\ \sum_{j\geq j_0 +2 } M^{-1}p_1(x,jM^{-1})\leq \int_{j_0 +1}^{\infty}p_1(x,y)\, dy.
 \end{align*}  
Together with \eqref{q_sup}, for $x\in[0,\infty)$ and $M\geq 1$,
\begin{align*}
\sum_{j\geq 0 } M^{-1}p_1(x,jM^{-1})\leq \int_{0}^{\infty}p_1(x,y)\, dy+\frac{2}{M}\sup_{y\in [0,\infty)} p_1(x,y)\leq 1+2C_0.
\end{align*}
\eqref{q_sum} then follows by taking $C_1=1+2C_0$.
\end{proof}

\begin{lemma}\label{lem:BBsum_2}
For any $\varepsilon>0$, there exists $L>0$ such that
\begin{align}\label{q_outside}
\sup_{x\in [0,\infty), M\geq 1} M^{-1}\sum_{|jM^{-1}-x|\geq L}  p_1(x,jM^{-1})\leq \varepsilon.
\end{align}
\end{lemma}

\begin{proof} Let $y_*(x)$ be the unique maximum point of $p_1(x,y)$ given in Corollary~\ref{cor:ystar}. Let $L_0$ be the constant in Corollary~\ref{cor:ystar}. Then $y_*(x)\in [x-L/2,x+L/2]$ for all $L\geq 2L_0$. Together with the monotonicity of $p_1(x,y)$ on $[0,y_*(x)]$ and $[y_*(x),\infty)$, it holds that
\begin{align}\label{3}
M^{-1}\sum_{|jM^{-1}-z|\geq L}  p_1(x,jM^{-1})\leq \int_{ |y-x|\geq L/2,y\geq 0} p_1(x,y)\, dy.  
\end{align}
It suffices to show that by taking $L$ large enough, the right hand side of \eqref{3} is small for all $x\in [0,\infty)$.\\[-0.3cm]

Given $\varepsilon>0$. We aim to show that for $L$ large enough, for all $x\geq 0$ it holds that
\begin{align*}
\int_{|y-x|\leq L/2, y\geq 0} p_1(x,y)\, dy\geq 1-\varepsilon.
\end{align*}
Suppose the above fails. There exists a sequence $x_n$ such that 
\begin{align*}
\int_{|y-x_n|\leq n, y\geq 0} p_1(x_n,y)\, dy< 1-\varepsilon.
\end{align*}
In view of \eqref{equ:q_asymp}, the sequence $x_n$ is bounded. Let $x_{n_k}$ be a convergent subsequence of $x_n$ and $x_0$ be the limit. Because $p_1(x_{n_k},\cdot)$ converges to $p_1(x_{0},\cdot)$ locally uniformly, we have for all $L>0$,
\begin{align*}
\int_{|y-x_0|\leq L/2, y\geq 0} p_1(x_0,y)\, dy< 1-\varepsilon.
\end{align*}
This is impossible. The proof is finished.
\end{proof}

\section{Discrete Approximation of non-intersecting Bessel Bridge Ensemble}\label{sec:discrete}
In this section, we show that non-intersecting Bessel bridge ensembles can be approximated through discretization. This completes the {\em fourth step} in the proof of Proposition~\ref{lem:monotonicity}.

We begin by recalling the setting. Fix $\alpha\geq 0$, $k\in\mathbb{N}$. Let $(f,g)$ be a pair of functions defined on $[0,1]$ that satisfies the continuity assumption in Definition~\ref{def:continuityassumption}. Let $\vec{x}$, $\vec{y}\in\mathbb{R}^k_+$ be two vectors. We consider the $k$ independent squared $\alpha$-Bessel bridges on $[0,1]$ with entrance and exit data $(\vec{x},\vec{y})$. Their law is denoted by $\mathbb{P}_{\free}^{1,k,(0,1),\vec{x},\vec{y}}$. The law of the non-intersecting Bessel bridge ensemble, denoted by $\mathbb{P}^{1,k,(0,1),\vec{x},\vec{y},f,g}$, is obtained from conditioning $\mathbb{P}_{\free}^{1,k,(0,1),\vec{x},\vec{y}}$ on the event that all of the curves mutually avoid each other and $f(x), g(x)$. See Definition \ref{def:Bessel bridge LE} for details. To make sure $\mathbb{P}^{1,k,(0,1),\vec{x},\vec{y},f,g}$ is well defined, we assume that

\begin{equation}\label{equ:Z>0-app}
Z^{1,k,(0,1),\vec{x},\vec{y},f,g}>0
\end{equation} 
See \eqref{eqn:normalcont_Bessel} for the definition of $Z^{1,k,(0,1),\vec{x},\vec{y},f,g}$. 

For technical reasons, we prefer to work with Bessel bridges instead of squared Bessel bridges. We view $\mathbb{P}^{1,k,(0,1),\vec{x},\vec{y},f,g}$ and $\mathbb{P}_{\free}^{1,k,(0,1),\vec{x},\vec{y}}$ as Borel measures on $C([1,k]_{\mathbb{Z}}\times [0,1],\mathbb{R})$. We write $\mathbb{Q}$ and $\mathbb{Q}_{\free}$ for the measures obtained by pushing forward $\mathbb{P}^{1,k,(0,1),\vec{x},\vec{y},f,g}$ and $\mathbb{P}_{\free}^{1,k,(0,1),\vec{x},\vec{y}}$ through the map $h(t)\mapsto \sqrt{|h(t)|}$ respectively. The goal of this section is to prove Proposition~\ref{lem:discrete}, which approximates $\mathbb{Q}$ through discretization.\\[-0.3cm]

Take $\ell \in\N$ and $M\geq 1$ and let $K=2^{\ell}$. Recall that $\Omega_{M,\ell}$ is defined in \eqref{equ:OmegaMl} and that $\Omega_{M,\ell}$ can be identified as a subset of $C([1,k]_{\mathbb{Z}}\times [0,1],\mathbb{R})$ through \eqref{OmegaMktoC}. In other words, we divide $[0,L]$ into $K$ subintervals with equal lengths and use $M$ as a parameter to discretize the height of the curves.\\[-0.3cm]

The weights $W_{\free}(\mathbf{z})$ and $W(\mathbf{z})=W_{\free}(\mathbf{z})G(\mathbf{z})$ for $\mathbf{z}\in\Omega_{M,\ell}$ are defined in \eqref{def:weightfree}. Let $\mathbb{Q}_{M,\ell,\free}$ and $\mathbb{Q}_{M,\ell}$ be the probability measures on $\Omega_{M,\ell}$ which is proportional to $W_{\free}$ and $W$ respectively. In other words, for any $\mathbf{z}\in \Omega_{M,\ell}$,
\begin{equation}
\mathbb{Q}_{M,\ell,\free}(\mathbf{z})\propto W_{\free}(\mathbf{z})\ \textup{and}\ \mathbb{Q}_{M,\ell}(\mathbf{z})\propto W(\mathbf{z}).
\end{equation}
Through \eqref{OmegaMktoC}, $\mathbb{Q}_{M,\ell,\free}$ and $\mathbb{Q}_{M,\ell}$ can be viewed as probability measures on $C([1,k]_{\mathbb{Z}}\times [0,1],\mathbb{R})$. Moreover, $G(\mathbf{z})$ is the indicator function of the set $A_{\ell}\subset C([1,k]_{\mathbb{Z}}\times [0,1],\mathbb{R})$ defined by
\begin{equation}\label{def:Al}
\begin{split}
A_{\ell}\coloneqq\{h(t)\in C([1,k]_{\mathbb{Z}}\times [0,1],\mathbb{R})\, |\, &\sqrt{ f(t)}<h_1(t)<\dots<h_k(t)<\sqrt{g(t)}\\
 &\textup{for all}\ t=n/K, n\in [1,K-1]_{\mathbb{Z}}  \}.
\end{split}
\end{equation} 
Similarly, define
\begin{equation}\label{def:Anol}
\begin{split}
A\coloneqq\{h(t)\in C([1,k]_{\mathbb{Z}}\times [0,1],\mathbb{R})\, |\, &\sqrt{ f(t)} <h_1(t)<\dots<h_k(t)<\sqrt{ g(t)}\quad \textup{for all}\ t\in [0,1] \}.
\end{split}
\end{equation}
Because of \eqref{equ:Z>0-app}, there exists $\ell_0$ and $M_0$ such that for all $\ell\geq \ell_0$ and $M\geq M_0$, $\mathbb{Q}_{M,\ell,\free}(A_\ell)> 0$. From now on we always assume $\ell\geq \ell_0$ and $M\geq M_0$. Clearly $\mathbb{Q}_{M,\ell}$ equals $\mathbb{Q}_{M,\ell,\free}$ conditioning on $A_{\ell}$. In other words, for all Borel subsets $E\subset C([1,k]_{\mathbb{Z}}\times [0,1],\mathbb{R})$, 
\begin{align}\label{equ:QQQE}
\mathbb{Q}_{M,\ell}(E)=  {\mathbb{Q}_{M,\ell,\free}(E\cap A_{\ell})}\big/{\mathbb{Q}_{M,\ell,\free} ( A_{\ell})}.
\end{align}
To compare, the measures of non-intersection/free Bessel bridges have the relation
\begin{align}\label{equ:QE}
\mathbb{Q}(E)= {\mathbb{Q}_{\free}(E\cap A)}\big/{\mathbb{Q}_{ \free} ( A)}.
\end{align} 
The goal is to prove the following proposition.
\begin{proposition}\label{lem:discrete}
As $M$ and $\ell$ go to infinity, $\mathbb{Q}_{M,\ell} $ converges weakly to $ \mathbb{Q}$.
\end{proposition}

We first show that as $M$ goes to infinity, $\mathbb{Q}_{M,\ell,\free}$ converges the marginal law of $\mathbb{Q}_{\free}$ restricted on $t=n/K$. We write $\mathbb{Q}_{\ell, \free}$ for such a law.
\begin{lemma}\label{lem:freeconvergence-1}
As $M$ goes to infinity, $\mathbb{Q}_{M,\ell,\free}$ converges to $Q_{\ell,\free}$ weakly.
\end{lemma}
\begin{proof}
Under the law of $\mathbb{Q}_{M,\ell,\free}$ or $\mathbb{Q}_{\ell,\free}$, curves at different levels are independent of each other. Therefore it suffices to consider the special case $k=1$. Let $x,y\in [0,\infty)$ be the entrance and exit data. In this proof we always use the convention that $z_0=\sqrt{x}$ and $z_{K}=\sqrt{y}$. For simplicity, we denote $p_{K^{-1}}(\cdot,\cdot)$, the transition probability for Bessel processes \eqref{def:pF}, by $p(\cdot,\cdot)$\\[-0.3cm]

Let $(z_1, \cdots, z_{K-1})$ be distributed according to $\mathbb{Q}_{M,\ell,\free}$. Given $(w_1,w_2,\dots, w_{K-1})\in [0,\infty)^{K-1}$, it holds that
\begin{align*}
&\mathbb{Q}_{M,\ell,\free}(z_j\leq w_j,\ \textup{for all}\ j\in [1,K-1]_{\mathbb{Z}})=\sum_{\substack{z\in\Omega_{M,\ell}\\ z_j\leq w_j}}\prod_{j=1}^K p( z_{j-1},z_j)\bigg/\sum_{z\in\Omega_{M,\ell}}\prod_{j=1}^K p(  z_{j-1},z_j).
\end{align*}
It suffices to show that 
\begin{align}\label{totalmass}
\lim_{M\to\infty} M^{-K+1} \sum_{z\in\Omega_{M,\ell}}\prod_{j=1}^K p( z_{j-1},z_j)=\int_{\mathbb{R}^{K-1}} \prod_{j=1}^{K} p( z_{j-1},z_j)\, \prod_{j=1}^{K-1} dz_j
\end{align}
and that
\begin{align}\label{partialmass}
\lim_{M\to\infty}M^{-K+1} \sum_{\substack{z\in\Omega_{M,\ell}\\ z_j\leq w_j}}\prod_{j=1}^K p(  z_{j-1},z_j)=\int_{\{z_j\leq w_j\} } \prod_{j=1}^{K} p( z_{j-1},z_j)\, \prod_{j=1}^{K-1} dz_j.
\end{align}

Next, we prove \eqref{totalmass} and \eqref{partialmass}. Because $p(x,y)$ is a continuous function, for any $L>0$, the following Riemann sum converges to the Riemann integral,
\begin{align}\label{15}
\lim_{M\to\infty} M^{-K+1} \sum_{\substack{z\in\Omega_{M,\ell}\\ |z_j-z_{j-1}|\leq L}}\prod_{j=1}^K p( z_{j-1},z_j)=\int_{\{|z_j-z_{j-1}|\leq L\}} \prod_{j=1}^{K} p( z_{j-1},z_j)\, \prod_{j=1}^{K-1} dz_j.
\end{align}
From Lemma \ref{lem:BBsum_1}, Lemma \ref{lem:BBsum_2} and \eqref{q_scaling}, for any $m\in [1,K]_{\mathbb{Z}}$, we have
\begin{align*}
M^{-K+1} \sum_{\substack{z\in\Omega_{M,\ell}\\ |z_m-z_{m-1}|> L}}\prod_{j=1}^K p(z_{j-1},z_j)\leq o_L (1) .
\end{align*}
Here $o_L(1)$ denotes a quantity which converges to zero when $L$ goes to infinity.
In particular, for all $M\geq 1$,
\begin{align}\label{16}
M^{-K+1}\left| \sum_{\substack{z\in\Omega_{M,\ell}\\ |z_j-z_{j-1}|\leq L}}\prod_{j=1}^K p( z_{j-1},z_j)-\sum_{  z\in\Omega_{M,\ell} }\prod_{j=1}^K p( z_{j-1},z_j)\right|\leq o_L(1).
\end{align}
Combining \eqref{15} and \eqref{16}, \eqref{totalmass} follows. The argument for \eqref{partialmass} is similar and we omit it. The proof is finished.
\end{proof}
\begin{lemma}\label{lem:freeconvergence-2}
As $\ell$ goes to infinity, $\mathbb{Q}_{\ell,\free}$ converges to $\mathbb{Q}_{\free}$ weakly.
\end{lemma}
\begin{proof}
It suffices to consider the case $k=1$. Fix the entrance and exit data $x,y\in (0,\infty)$. Let $\mathcal{S}$ be a Bessel bridge on $[0,1]$ with $\mathcal{S}(0)=\sqrt{x}$ and $\mathcal{S}(1)=\sqrt{y}$. Note that $\mathbb{Q}_{\free}$ is the law of $\mathcal{S}$. Let
\begin{align*}
\mathcal{S}_\ell(t)=\left\{ \begin{array}{cc}
\mathcal{S} (t) & t=j/2^{\ell},\ j\in [0,2^\ell]_{\mathbb{Z}},\\
\textup{linear interpolation}, &\ \textup{others}. 
\end{array} \right.
\end{align*} 
Then $\mathbb{Q}_{\ell,\free}$ is the law of $\mathcal{S}_\ell$. Because Bessel bridges are continuous, we have $\mathcal{S}_\ell$ converges uniformly to $\mathcal{S}$ almost surely. This implies $\mathbb{Q}_{\ell,\free}$ converges to $\mathbb{Q}_\ell$ weakly.
\end{proof}

\begin{lemma}\label{lem:B4}
Let $\bar{A}_{\ell}$ and $\bar{A}$ be the topological closure of $A_{\ell}$ and $A$ (defined in \eqref{def:Al} and \eqref{def:Anol}) in $C([1,k]_{\mathbb{Z}}\times [0,1],\mathbb{R})$ respectively. Then $\bar{A}=\cap_{\ell=1}^\infty \bar{A}_{\ell}$.
\end{lemma}

\begin{proof}
Because of \eqref{equ:Z>0-app}, $A_\ell$ and $A$ are non-empty. It is straightforward to see that
\begin{equation*} 
\begin{split}
\cap_{\ell=1}^\infty \bar{A}_{\ell}=\{h(t)\in C([1,k]_{\mathbb{Z}}\times [0,1],\mathbb{R})\, |\, &\sqrt{ f(t)}\leq  h_1(t)\leq \dots\leq h_k(t)\leq\sqrt{g(t)}\\
 &t\in 2^{-\ell}\mathbb{Z}\cap [0,1]\ \textup{for some}\ \ell\in\mathbb{N}   \},
\end{split}
\end{equation*} 
and that
\begin{equation*} 
\begin{split}
\bar{A}=\{h(t)\in C([1,k]_{\mathbb{Z}}\times [0,1],\mathbb{R})\, |\, &\sqrt{ f(t)} \leq h_1(t)\leq  \dots\leq h_k(t)\leq \sqrt{ g(t)}\quad \textup{for all}\ t\in [0,1] \}.
\end{split}
\end{equation*}
Since $(f,g)$ satisfies the continuity assumption in Definition~\ref{def:continuityassumption}, $\sqrt{f}$ and $\sqrt{g}$ are one-sided continuous on $(0,1)$. This ensures $\bar{A}=\cap_{\ell=1}^\infty \bar{A}_\ell$.
\end{proof}
\begin{proof}[Proof of Proposition~\ref{lem:discrete}] 
Let $\partial A_{\ell}$ and $\partial A$ be the topological boundaries of $A_\ell$ and $A$ (defined in \eqref{def:Al} and \eqref{def:Anol}) respectively. Because Bessel bridges have continuous transition densities, $\mathbb{Q}_\ell(\partial A_\ell)=0$. Together with Lemma~\ref{lem:freeconvergence-1}, we have  
\begin{align}\label{equ:MQ}
 \lim_{M\to\infty}\mathbb{Q}_{M,\ell,\free}(A_\ell)=\mathbb{Q}_{\ell,\free}(A_\ell).
\end{align}
Next, we aim to show that
\begin{align}\label{equ:KQ}
 \lim_{\ell\to\infty}\mathbb{Q}_{\ell,\free}(A_\ell)=\mathbb{Q}_{\free}(A).
\end{align} 
Because $g(x)$ is lower semi-continuous and $f(x)$ is upper semi-continuous, $ {A}$ is an open set in $C([1,k]_{\mathbb{Z}}\times [0,1],\mathbb{R})$. Together with $ {A}\subset {A}_\ell$ and Lemma~\ref{lem:freeconvergence-2}, we have
\begin{align*}
 \liminf_{\ell\to\infty} \mathbb{Q}_{\ell,\free}(     A_\ell )\geq\liminf_{\ell\to\infty} \mathbb{Q}_{\ell,\free}(A)\geq  \mathbb{Q}_{\free}(A)
\end{align*} 
Also, for any $j\in\mathbb{N}$,
\begin{align*}
\limsup_{\ell\to\infty} \mathbb{Q}_{\ell,\free}(   {A}_\ell ) \leq \limsup_{\ell\to\infty} \mathbb{Q}_{\ell,\free}( \bar{ {A}}_j)\leq \mathbb{Q}_{ \free}( \bar{ {A}}_j).   
\end{align*}
From Lemma~\ref{lem:B4}, $\bar{{A}}=\cap_{j=1}^\infty \bar{{A}}_j$. Hence we have $\limsup_{\ell \to\infty} \mathbb{Q}_{\ell,\free}(   {A}_\ell )\leq \mathbb{Q}_{\free}(\bar{ {A}})$.  From Lemmas~\ref{lem:not touching} and \ref{lem:QQ'}, $\mathbb{Q}(\partial A)=0$. Therefore, $ \limsup_{\ell \to\infty} \mathbb{Q}_{\ell,\free}(   {A}_\ell )\leq\mathbb{Q}_{\free}( { {A}})$. Then \eqref{equ:KQ} follows. Combining \eqref{equ:MQ} and \eqref{equ:KQ}, we have
\begin{equation}\label{equ:LLQ}
\lim_{\ell\to\infty} \lim_{M\to\infty}\mathbb{Q}_{M,\ell,\free}(A_\ell)=\mathbb{Q}_{\free}(A).
\end{equation}  

To show that $\mathbb{Q}_{M,\ell}$ converges to $\mathbb{Q}$ weakly, it suffices to show that for any open subset $E\subset C([1,k]_{\mathbb{Z}}\times[0,1],\mathbb{R})$, we have
\begin{align}\label{equ:06230645}
\liminf_{\ell\to\infty} \liminf_{M\to\infty}\mathbb{Q}_{M,\ell }(E)\geq \mathbb{Q}(E).
\end{align}
From now on, we fix an open set $E\subset C([1,k]_{\mathbb{Z}}\times[0,1],\mathbb{R})$. Since $E\cap A_\ell$ is also open, we have from Lemma~\ref{lem:freeconvergence-1} $\liminf_{M\to\infty}\mathbb{Q}_{M,\ell ,\free}(E\cap A_\ell)\geq \mathbb{Q}_{\ell,\free}(E\cap A_\ell).$ Therefore, $$\liminf_{\ell\to\infty} \liminf_{M\to\infty}\mathbb{Q}_{M,\ell,\free}(E\cap A_\ell) \geq \liminf_{\ell\to\infty} \mathbb{Q}_{\ell,\free}(E\cap A_\ell).$$
Together with $A\subset A_\ell$ and Lemma~\ref{lem:freeconvergence-2}, we have
\begin{equation*}
 \begin{split}
\liminf_{\ell\to\infty} \mathbb{Q}_{\ell,\free}(E\cap A_\ell) \geq \liminf_{\ell\to\infty} \mathbb{Q}_{\ell,\free}(E\cap A ) \geq \mathbb{Q}_{\free}(E\cap A ).
 \end{split}
\end{equation*} 
Therefore,
\begin{equation}\label{equ:KEQ}
\liminf_{\ell\to\infty} \liminf_{M\to\infty}\mathbb{Q}_{M,\ell,\free}(E\cap A_\ell) \geq \mathbb{Q}_{\free}(E\cap A ).
\end{equation}
Then \eqref{equ:06230645} follows by combining \eqref{equ:QQQE}, \eqref{equ:QE}, \eqref{equ:LLQ} and \eqref{equ:KEQ}. This finishes the proof.
\end{proof}

\section{Extended Bessel Kernel}\label{sec:BEkernel}
In this section we derive the multi-time correlation kernel for the non-intersecting squared Bessel process and prove the convergence under the hard edge scaling in \eqref{eqn:edgeLE1}. The results could also be found in \cite[Section 11.7.3]{For10}. We provide the arguments for the reader's convenience. Note that our extended Bessel kernel differs from the one in \cite{For10} up to some coefficients. This comes from that the scaling in \cite{For10} centers around $t=1/2$ while we center around $t=1$.\\[-0.3cm]

Fix $\alpha\geq 0$ and $N\in\mathbb{N}$. Consider an $N$ non-colliding squared $\alpha$-Bessel process $Y^{N,\alpha}_1(t)<Y^{N,\alpha}_2(t)<\dots <Y^{N,\alpha}_N(t)$ with $Y^{N,\alpha}_j(0)=0$. We also use $Y^{N,\alpha}(t)$ to denote the vector $(Y^{N,\alpha}_1(t),\dots,Y^{N,\alpha}_N(t))$. The density of $Y^{N,\alpha}(1/2)$ is given in \eqref{equ:06261715}. Together with the scaling property \eqref{equ:06261716}, the density of $Y^{N,\alpha}(t)$ is 
\begin{align}\label{equ:onepoint}
2^{-\alpha N-N^2} t^{-N^2} C(N,\alpha)\prod_{j=1}^N (x_j/t)^\alpha e^{-(2t)^{-1}x_j}\times\left(\Delta(\vec{x})\right)^2 \mathbbm{1}(\vec{x}\in \mathbb{W}^N_+) \prod_{j=1}^N dx_j.
\end{align}
Here $\Delta(\vec{ x})=\prod_{1\leq i<j\leq N}(x_j-x_i)$ is the Vendermonde determinant.

For any $s>t>0$, the transition probability from $Y^{N,\alpha}(t)=\vec{x}$ to $Y^{N,\alpha}(s)=\vec{y}$ can be computed through the Karlin-McGregor formula \cite{KM} (see also \cite[(1.6)]{KT}). 
\begin{align}\label{equ:transition}
\det\left[ q_{s-t}(x_i,y_j)_{1\leq i,j\leq N} \right]\frac{\Delta(\vec{y}) }{\Delta(\vec{x})}\cdot \mathbbm{1}(\vec{y}\in \mathbb{W}^N_+)\,\prod_{j=1}^N dy_j.
\end{align}
Here $q_t(x,y)$ is the transition probability of squared Bessel processes defined in \eqref{def:q}. From \eqref{equ:onepoint} and \eqref{equ:transition}, we can derive the joint density of $Y^{N,\alpha}$ at multiple times. Fix an arbitrary $m\in\mathbb{N}$ and $0<t_1<t_2<\dots <t_m$.  The joint density of $(Y^{N,\alpha}(t_1),\dots,Y^{N,\alpha}(t_m))$ is given by
\begin{equation}\label{density}
\begin{split}
&2^{-\alpha N-N^2} C(N,\alpha)\times \prod_{k=1}^m \mathbbm{1}(\vec{x}^{(k)}\in\mathbb{W}^N_+)\\
& \times  t_1^{-N^2} \prod_{j=1}^N (t_1^{-1}x^{(1)}_j)^\alpha e^{-(2t_1)^{-1}x^{(1)}_j}\times \Delta(x^{(1)})\\
&\times\prod_{k=1}^{m-1} \det\left[ q_{t_{k+1}-t_k}(x_i^{(k)},x_j^{(k+1)})_{1\leq i,j\leq N} \right] \times \Delta (x^{(m)})\times \prod_{k=1}^m\prod_{j=1}^N dx^{(k)}_j.
\end{split}
\end{equation}

For $x,t>0$ and $j\in\mathbb{N}$, we define
\begin{align}
\phi^\alpha_j(t,x)\coloneqq &\frac{ \Gamma(j)}{2^{j+\alpha} \Gamma(\alpha+j)}t^{-j}(t^{-1}x)^{\alpha }e^{- (2t)^{-1}x }L^\alpha_{j-1}((2t)^{-1}x),\label{def:phi}\\[0.2cm]
\psi^\alpha_j(t,x)\coloneqq &2^{j-1} t^{j-1} L^\alpha_{j-1}((2t)^{-1}x).\label{def:psi}
\end{align}
Here $L^\alpha_{j-1}(x)$ is the generalized Laguerre polynomial of degree $j-1$, 
\begin{align}\label{def:Laguerre}
L^\alpha_{j-1}(x)\coloneqq \frac{x^{-\alpha}e^x}{\Gamma(j)}\frac{d^{j-1}}{dx^{j-1}}(x^{\alpha+j-1}e^{-x}).
\end{align}
Note that the following orthogonal relation holds.
\begin{align}\label{Lagorthogonal}
\int_0^\infty dx\, x^\alpha e^{-x} L^\alpha_i(x)L^\alpha_j(x)=\frac{\Gamma(\alpha+j-1)}{\Gamma(j-1)}\delta_{ij}.
\end{align}

Because the leading coefficient in $L^{\alpha}_{j-1}(x)$ is $\frac{(-1)^{j-1}}{\Gamma(j)}$, we have
\begin{align*}
&\det\left[\phi_i^\alpha( t , x_j)\right]_{1\leq i,j\leq N}=(-1)^{\frac{N(N-1)}{2}}\cdot 2^{-\alpha N-N^2} t^{-N^2}\prod_{j=1}^N \frac{1}{\Gamma(\alpha+j)}\times\prod_{j=1}^N (x_j/t)^{\alpha}e^{-(2t)^{-1}x_j } \times\Delta(\vec{x}),\\
&\det\left[\psi_i^\alpha ( t , x_j)\right]_{1\leq i,j\leq N}=(-1)^{\frac{N(N-1)}{2}} \prod_{j=1}^N \frac{1}{\Gamma(j)}   \times\Delta(\vec{x}).
\end{align*} 
Therefore, the joint density \eqref{density} can be expressed as
\begin{equation}\label{density2}
\begin{split}
&\det \left[\phi^\alpha_i( t^{(1)}, x^{(1)}_j)\right]_{1\leq i,j\leq N}\\
\times &\prod_{k=1}^{m-1} \det\left[ q_{t_{k+1}-t_k}(x_i^{(k)},x_j^{(k+1)})_{1\leq i,j\leq N} \right]\\
\times &\det \left[\psi_j^\alpha(t^{(m)}, x^{(m)}_i)\right]_{1\leq i,j\leq N}\times \prod_{k=1}^m \mathbbm{1}(\vec{x}^{(k)}\in\mathbb{W}^N_+)  \prod_{k=1}^m\prod_{j=1}^N dx^{(k)}_j.
\end{split}
\end{equation}

\begin{lemma}\label{lem:phiQpsi}
For any $0<t<s$, the following statements hold. 
\begin{align}\label{delta}
\int_{0}^\infty \phi_i^\alpha( t,x) \psi_j^\alpha( t,x)\, dx=\delta_{ij}.
\end{align}
\begin{align}\label{phiQ}
\int_{0}^\infty \phi_j^\alpha(t,x) q^\alpha_{s-t}(x,y) \, dx=\phi_j^\alpha(s,y).
\end{align}
\begin{align}\label{Qpsi}
\int_{0}^\infty q^\alpha_{s-t}(x,y)\psi_j^\alpha(s,y)\, dy=\psi_j^\alpha(t,x).
\end{align}
\end{lemma}
\begin{proof}
\eqref{delta} is a reformulation of the orthogonality of generalized Laguerre polynomials \eqref{Lagorthogonal}. The equality \eqref{phiQ} is proved in \cite[Lemma 3.4(ii)]{FF}.\\

We proceed to prove \eqref{Qpsi}. For $j\in\mathbb{N}$,  define 
\begin{align*}
\tilde{\psi}_j(x)\coloneqq \int_{0}^\infty q^\alpha_{s-t}(x,y)\psi_j^\alpha(s,y)\, dy
\end{align*}
We aim to show that $\tilde{\psi}_j(x)=\psi_j^\alpha(t,x).$ By the Cauchy-Binnet formula and \eqref{density2}, we have for all $N\in\mathbb{N}$,
\begin{align*}
\det (\tilde{\psi}_j(x_i))_{1\leq i,j\leq N}= (-1)^{N(N-1)/2} \prod_{j=1}^N \frac{1}{\Gamma(j)}   \times\Delta(x_1,x_2,\dots ,x_N).
\end{align*}
By induction, $\tilde{\psi}_j(x)$ is a polynomial of degree $j-1$. Also, from \eqref{phiQ} and \eqref{delta} we have
\begin{align*}
\int_{0}^\infty \phi_i^\alpha( t,x)\tilde{\psi}_j(x)\, dx=\int_{0}^\infty \phi_i^\alpha( s,x){\psi}_j^\alpha(s,x)\, dx=\delta_{ij}. 
\end{align*}
As a result, $\tilde{\psi}_j(x)=\psi_j^\alpha(t,x)$.
\end{proof}

In view of \eqref{density2}, $Y^{N,\alpha}$ is determinantal. Moreover, from Lemma~\ref{lem:phiQpsi} and the Eynard-Mehta Theorem \cite{EM,BR}, a correlation kernel is of the form
\begin{align}\label{eq:K_alpha}
K^N_\alpha((t,x);(s,y))=&- {q}_{s-t}(x,y)\mathbbm{1}(t<s )+\sum_{j=1}^N \psi^\alpha_j( t,x)\phi^\alpha_j(s,y).
\end{align}
Consider the following gauge transformation:
\begin{align}
 \tilde{K}^N_\alpha((t,x);(s,y))=& (x/y)^{\alpha/2} {K}^N_\alpha((t,x);(s,y)).
\end{align}
Note that $\tilde{K}^N$ is also a correlation kernel for $Y^{N,\alpha}$. Moreover, from \eqref{def:qF} and \eqref{def:phi}, $\tilde{K}^N_\alpha((t,x);(s,y))$ extends smoothly to $y=0$.\\[-0.2cm]

We consider the following hard edge scaling. For arbitrary $(t,x ), (s,y)\in \mathbb{R}\times [0,\infty) $, consider
\begin{align}\label{limit}
(4N)^{-1}\tilde{K}^N_\alpha((1+(4N)^{-1}t ,(4N)^{-1}x ),(1+(4N)^{-1}s,(4N)^{-1}y) ).
\end{align}

\vspace{0.3cm}

It is a correlation kernel for $\mathcal{L}^{N,\alpha}$ defined in \eqref{eqn:edgeLE2}. The following theorem proves the locally uniform convergence of the above correlation kernel.
\begin{theorem}\label{thm:extended BE}
For any $(t,x)$ and $(s,y)$ in $\mathbb{R}\times [0,\infty)$, the limit of \eqref{limit} is given by
\begin{align*}
 \left\{\begin{array}{cc}
 -\displaystyle\int_{1/8}^\infty e^{- 2(s-t) z}   J_\alpha ( 2z^{1/2}x^{1/2})J_\alpha (2z^{1/2} y^{1/2})\, dz, & t< s  ,\\[0.5cm]
 \displaystyle\int_0^{1/8} e^{- 2(s-t) z}   J_\alpha ( 2z^{1/2}x^{1/2})J_\alpha (2z^{1/2} y^{1/2})\, dz, & t\geq s.  
\end{array}\right.
\end{align*}
Moreover, for any $L>0$, the convergence is uniform on $([-L,L]\times [0,L])^2$.
\end{theorem}

\begin{proof}
In view of \eqref{eq:K_alpha}, \eqref{limit} consists of two parts. The first part contains the $q_{s-t}$ term and the second part is a finite sum of $\psi^\alpha_j\phi^{\alpha}_j$. For the first part, we use the integral representation which can be derived from \cite[(6.633)]{GR}.
\begin{align*}
q_t(x,y)=&(y/x)^{\alpha/2}\int_0^\infty e^{- 2tz}   J_\alpha ( 2z^{1/2}x^{1/2})J_\alpha (2z^{1/2} y^{1/2})\, dz.
\end{align*} 
Together with $(4N)^{-1} q_{t/(4N)}(x/4N,y/4N)=q_{t}(x,y)$, we have
\begin{equation}\label{limit1}
\begin{split}
&- (x/y)^{\alpha/2} (4N)^{-1}{q}^\alpha_{(s-t)/4N}(x/4N,y/4N) \cdot \mathbbm{1}(t<s )\\[0.3cm]
=&-\int_0^\infty e^{- 2(s-t)z}   J_\alpha ( 2z^{1/2}x^{1/2})J_\alpha (2z^{1/2} y^{1/2})\, dz \cdot \mathbbm{1}(t<s ).
\end{split}
\end{equation}

Now we turn to the second part in \eqref{limit}. From the definitions of $\phi_j^\alpha$ and $\psi_j^\alpha$, we may rewrite the expression of $(x/y)^{\alpha/2}\sum_{j=1}^N \psi^\alpha_j(t,x)\phi^\alpha_j(s,y).$ It could be directly checked that it equals 
\begin{align*}
\left( {t}/{s}\right)^{\alpha/2}e^{- {y}/(2s)} \times 2^{-\alpha-1} \left(  {xy}/({ts}) \right)^{\alpha/2} \sum_{j=1}^N  \frac{\Gamma(j)}{\Gamma(\alpha+j)}t^{j-1}s^{-j} L^\alpha_{j-1}\left( {x}/(2t)\right)L^\alpha_{j-1}\left( {y}/(2s)\right).
\end{align*}
Set $M=4N$. Under the hard edge scaling, we have
\begin{align*}
(4N)^{-1} (x/y)^{\alpha/2}\sum_{j=1}^N \psi^\alpha_j\left(1+t/(4N) ,x/(4N) \right)\phi^\alpha_j\left(1+s/(4N) , {y}/(4N) \right)
\end{align*}
equals
\begin{equation}\label{1}
\begin{split}
&\left( \frac{1+t/M}{1+s/M} \right)^{\alpha/2}\exp\left( \frac{y}{2M+2s} \right)\times (2M)^{-1} \left(\frac{xy}{(2M+2t)(2M+2s)} \right)^{\alpha/2}\\[0.3cm]
&\times \sum_{j=1}^N \frac{\Gamma(j)}{\Gamma(\alpha+j)}(1+t/M)^{j-1}(1+s/M)^{-j} \cdot L^\alpha_{j-1}\left(\frac{x}{2M+2t}\right)L^\alpha_{j-1}\left(\frac{y}{2M+2s}\right). 
\end{split}
\end{equation}

We apply the asymptotic limit of Laguerre polynomial for $j$ large.  For $j\in\mathbb{N}$, let $\tilde{j}= j+2^{-1}(\alpha-1) $. Denote
\begin{align}\label{equ:07070301}
E_{j-1}^\alpha(x)\coloneqq L^\alpha_{j-1}(x)- \frac{\Gamma(\alpha+j)}{\Gamma(j)} \cdot e^{x/2}x^{-\alpha/2}\cdot \tilde{j}^{-\alpha/2} \cdot J_\alpha(2(\tilde{j} x)^{1/2}).
\end{align}
Here $J_\alpha(x)$ is the Bessel function of the first kind. From \cite[(8.22.4)]{Sze}, for any $c>0$, there exists $C$ depending on $c$ such that for $x\leq cj^{-1}$
\begin{align}\label{asymptotics}
|E_{j-1}^\alpha(x)|\leq Ce^{x/2}x^2 j^{\alpha}.
\end{align}
Using \eqref{equ:07070301} to substitute $L^\alpha_{j-1}$ in \eqref{1}, it becomes
\begin{equation*} 
\begin{split}
&\left( \frac{1+t/M}{1+s/M} \right)^{\alpha/2}\exp\left( \frac{y}{2M+2s} \right)\times (2M)^{-1} \left(\frac{xy}{(2M+2t)(2M+2s)} \right)^{\alpha/2}\\[0.2cm]
&\times \sum_{j=1}^N \frac{\Gamma(j)}{\Gamma(\alpha+j)}(1+t/M)^{j-1}(1+s/M)^{-j}\\[0.2cm]
&\times \left[ \frac{\Gamma(\alpha+j)}{\Gamma(j)} e^{x/(4M+4t)}  \left( \frac{x}{2M+2t} \right)^{-\alpha/2} \tilde{j}^{-\alpha/2} J_\alpha\left( 2\left( \frac{\tilde{j}x}{2M+2t} \right)^{1/2} \right) +E^{\alpha}_{j-1}\left( \frac{x}{2M+2t} \right)\right]\\[0.2cm]
&\times \left[ \frac{\Gamma(\alpha+j)}{\Gamma(j)} e^{y/(4M+4s)} \left( \frac{y}{2M+2s} \right)^{-\alpha/2} \tilde{j}^{-\alpha/2} J_\alpha\left( 2\left( \frac{\tilde{j}y}{2M+2s} \right)^{1/2} \right) +E^{\alpha}_{j-1}\left( \frac{y}{2M+2s} \right)\right] . 
\end{split}
\end{equation*}

Through a direct computation,  we rewrite 
$$(1+t/M) \left( \frac{1+t/M}{1+s/M} \right)^{-\alpha/2}\exp\left( -\frac{y}{2M+2s} \right)\times \eqref{1}=\textsc{I}+\textsc{II}+\textsc{III}+\textsc{IV},$$
such that 
\begin{align*}
\textsc{I}=& (2M)^{-1}\sum_{j=1}^N e^{-j(s-t)/M} J_{\alpha}\left(2 \left( \frac{\tilde{j} x}{ 2M+2t } \right)^{1/2}\right)J_{\alpha}\left(2 \left( \frac{\tilde{j} y}{ 2M+2s } \right)^{1/2}\right)  \\[0.2cm]
&\times \left[ \frac{\Gamma(\alpha+j) }{\Gamma(j) }\tilde{j}^{-\alpha} \right]\times e^{x/(4M+4t)+y/(4M+4s)} \times\left[ \left(\frac{1+t/M}{1+s/M}\right)^j  e^{j(s-t)/M} \right].
\end{align*}
\begin{align*}
\textsc{II}=&(2M)^{-1}\left( \frac{x}{2M+2t}\right)^{\alpha/2} e^{y/(4M+4s)}   \\[0.2cm]
&\times \sum_{j=1}^N  \left(\frac{1+t/M}{1+s/M}\right)^j   \tilde{j}^{-\alpha/2}  J_{\alpha}\left(2 \left( \frac{\tilde{j} y}{ 2M+2s } \right)^{1/2}\right)   E_{j-1}^\alpha\left(\frac{x}{2M+2t } \right). 
\end{align*}
\begin{align*}
\textsc{III}=&(2M)^{-1}\left( \frac{y}{2M+2s}\right)^{\alpha/2} e^{x/(4M+4t)}   \\[0.2cm]
&\times \sum_{j=1}^N  \left(\frac{1+t/M}{1+s/M}\right)^j   \tilde{j}^{-\alpha/2}  J_{\alpha}\left(2 \left( \frac{\tilde{j} x}{ 4N+t } \right)^{1/2}\right)   E_{j-1}^\alpha\left(\frac{y}{4N+s} \right). 
\end{align*}
\begin{align*}
\textsc{IV}=&(2M)^{-1}\left( \frac{xy}{(2M+2t)(2M+2s)} \right)^{\alpha/2}\\[0.2cm]
&\times \sum_{j=1}^N \frac{\Gamma(j)}{\Gamma(\alpha+j)}\left(\frac{1+t/M}{1+s/M} \right)^j E^\alpha_{j-1}\left(\frac{x}{2M+2t} \right)E^\alpha_{j-1}\left(\frac{y}{2M+2s} \right) 
\end{align*}

From \eqref{asymptotics}, \textsc{II}, \textsc{III} and \textsc{IV} converge to zero as $N$ goes to infinity. And \textsc{I} (as a Riemann sum) converges to
\begin{align*}
   \displaystyle\int_{0}^{1/8} e^{- 2(s-t)z}   J_\alpha ( 2z^{1/2}x^{1/2})J_\alpha (2z^{1/2} y^{1/2})\, dz.  
\end{align*}
Therefore,
\begin{align}\label{limit2}
\lim_{N\to\infty}\textup{\eqref{1}}=\displaystyle\int_{0}^{1/8} e^{- 2(s-t)z}   J_\alpha ( 2z^{1/2}x^{1/2})J_\alpha (2z^{1/2} y^{1/2})\, dz.
\end{align}

The desired result now follows by combining \eqref{limit1} and \eqref{limit2}. It is straightforward to check that the convergence in \eqref{limit2} is uniform on $(t,x),(s,y)\in [-L,L]\times[0,L]$.
\end{proof}

\vspace{1cm}
\end{appendix}

\printbibliography
\end{document}